\documentclass[12pt,leqno]{amsart}
\usepackage{graphicx}%
\usepackage{multirow}%
\usepackage{amsmath,amssymb,amsfonts,mathtools}%
\usepackage{amsthm}%
\usepackage{mathrsfs}%
\usepackage[title]{appendix}%
\usepackage{xcolor}%
\usepackage{textcomp}%
\usepackage{manyfoot}%
\usepackage{booktabs}%
\usepackage{algorithm}%
\usepackage{algorithmicx}%
\usepackage{algpseudocode}%
\usepackage{listings}%
\usepackage{bbm}
\usepackage[pagebackref=true, colorlinks=true, citecolor=blue]{hyperref}

\theoremstyle{thmstyleone}%
\newtheorem{theorem}{Theorem}% 
\newtheorem{remark}{Remark}%
\newtheorem{definition}{Definition}%

\definecolor{ao}{rgb}{0.0, 0.5, 0.0}
\newcommand{\bn}{{\bf n}}
\newcommand{\sU}{{\mathscr{U}}}

\newcommand{\cred}{\color{black}}
\newcommand{\sB}{\mathcal B}
\newcommand{\sD}{\mathscr D}

\newcommand{\Tr}{\text{Tr}}
\newcommand{\bQ}{{\bf Q}}

\newcommand{\sC}{\mathcal C}

\newcommand{\sO}{\mathcal O}

\newcommand{\sS}{\mathcal S}

\newcommand{\sY}{\mathcal Y}

\newcommand{\ep}{\varepsilon}

\newcommand{\sL}{\mathscr{L}}
\newcommand{\sF}{\mathcal F}
\newcommand{\bu}{{\bf u}}
\newcommand{\bq}{{\bf q}}
\newcommand{\bP}{{\mathbb P}}
\newcommand{\bw}{{\bf w}}
\newcommand{\bW}{{\bf W}}
\newcommand{\br}{{\bf r}}
\newcommand{\bE}{{\mathbb E}}

\newcommand{\bD}{{\bf D}}
\newcommand{\bH}{{\bf H}}
\newcommand{\bC}{{\bf C}}

\newcommand{\bU}{{\bf U}}
\newcommand{\bL}{{\bf L}}
\newcommand{\sW}{\mathscr{W}}
\newcommand{\sV}{\mathscr{V}}

\newcommand{\cadlag}{c\`{a}dl\`{a}g~}

\renewcommand{\tilde}{\widetilde}

\newtheorem{assm}{Assumption}[section]
\newtheorem{lem}[theorem]{Lemma}
\newtheorem{cor}[theorem]{Corollary}

\numberwithin{theorem}{section} 

\newcommand{\bx}{{\bf x}}
\newcommand{\be}{{\bf e}}

\newcommand{\R}{\mathbb{R}}

\usepackage[foot]{amsaddr}
\usepackage[margin=1in]{geometry}

\begin{document}

	\title[Stochastic FSI problem]{Existence of martingale solutions to a nonlinearly coupled stochastic fluid-structure interaction problem}
	
	\author[K. Tawri, S. \v{C}ani\'{c}]{Krutika Tawri$^{1}$ and Sun\v{c}ica \v{C}ani\'{c}$^1$}
	
	\address{\newline	$^1$ Department of Mathematics, University of California Berkeley, CA, USA.}
	
	\email{ktawri@berkeley.edu (Krutika Tawri), canics@berkeley.edu (Sun\v{c}ica \v{C}ani\'{c})}

\begin{abstract}
	In this paper we study a {nonlinear} stochastic fluid-structure interaction problem {with}  a multiplicative, white-in-time noise. The problem consists of the  Navier-Stokes equations describing the flow of an incompressible, viscous fluid {in a 2D cylinder} interacting with {an elastic lateral wall whose elastodynamics is described by a membrane/shell equation}.  {The flow is driven by the inlet and outlet data, and by the stochastic forcing.} The stochastic noise is applied both to the fluid equations as a volumetric body force, and to the structure as an external forcing to the deformable fluid boundary. The fluid and the structure are {\bf{nonlinearly coupled}} via the	kinematic and dynamic conditions {assumed at the moving interface, which is a random variable not known a priori}. {The geometric nonlinearity due to the nonlinear coupling requires the development of new techniques to capture martingale solutions for this class of  stochastic fluid-structure interaction problems.} We introduce a constructive approach based on a Lie splitting scheme and prove the existence of martingale solutions to the system. To the best of our knowledge, this is {the first result in the field of stochastic PDEs that addresses existence of solutions on moving fluid domains {involving incompressible viscous fluids}, {where} {\bf{the displacement of the boundary {and} the fluid domain {are {\bf{random variables}} that} {are} not known a priori and {are parts} of the solution itself}}}. \\

\noindent Keywords: Stochastic moving boundary problems, Fluid-structure interaction, martingale solutions\\
\noindent MSC: 60H15, 35A01
\end{abstract}

\maketitle

\section{Introduction}\label{sec1}

This paper presents a constructive method to investigate solutions of a {nonlinearly coupled} stochastic fluid-structure interaction (FSI) {problem}. 
The focus is on a stochastically forced benchmark problem involving a linearly elastic {membrane/shell} that interacts with a {two dimensional flow of a} viscous, incompressible Newtonian fluid  {across a moving interface}. 
{This benchmark problem incorporates the main mathematical difficulties associated with nonlinearly coupled stochastic FSI problems.} The fluid is modeled by the 2D Navier-Stokes equations, while the membrane is modeled by the linearly elastic {membrane/shell} equations. {The fluid and the structure are coupled across the {\bf{moving interface}}}
via a two-way coupling that ensures continuity of velocities and continuity of contact forces at the fluid-structure interface. 
The {stochastic noise is applied both to the fluid equations as a volumetric body force, and to the structure as an external forcing to the deformable fluid boundary. The noise} is
multiplicative depending on the structure displacement, structure velocity and the fluid velocity, {and it is of the form $G(\bu,v,\eta) {dW}$, where 
	$W$ is a Wiener process. 
	The main result of this manuscript is a constructive proof of the existence of weak martingale solutions to this nonlinear stochastic fluid-structure interaction problem. 
	Namely, we prove the existence of solutions that are weak in the analytical sense and in the probabilistic sense.
	In other words, we show that
	despite the roughness,  the underlying nonlinear deterministic fluid-structure interaction problem is robust to noise.}
To the best of our knowledge, this is {the first result in the field of stochastic PDEs that
	addresses the question of existence of solutions on moving domains {involving incompressible fluids}, {where} {\bf{the displacement of the boundary {and} the motion of the fluid domain {are {\bf{random variables}} that}
			{are} not known a priori and {are parts} of the solution itself}}}. 

There are {many applications}  that suggest and provide  evidence to treating moving surfaces as random boundaries in order to deal with geometric uncertainties due to insufficient data or measurement errors, see e.g. \cite{CF09} in the context of wind-engineering and \cite{StochasticHeartFSI} {in the context of blood flow.} {In particular, the systolic and diastolic rhythm of the heart has a strong stochastic component, giving rise to stochasticity in FSI describing the flow of blood in coronary arteries that sit on the surface of the heart. In fact, in \cite{StochasticHeartFSI} it was shown that the stochastic fluctuations of the single ion channel and the sub-cellular dynamics in tissue and organ scale get reflected in the macroscopic random cardiac events \cite{StochasticHeartFSI}, which should be modeled using stochastic partial differential equations to capture such phenomena. }
In general, studying stochastic FSI  is important because well-posedness of stochastic FSI models  provides confidence that the deterministic FSI models are, indeed, robust to stochastic noise that occurs naturally in real-life problems.

Besides its applications, {the problem we consider in this manuscript} is interesting from a mathematical analysis point of view due to the challenges arising from the random nature of the moving fluid domains. {In this manuscript we develop techniques to deal with these challenges in a constructive proof of the existence of
	martingale solutions to this class of stochastic moving boundary problems.}

The first constructive existence
proof {{for}} a {{\bf{deterministic}} moving boundary problem {\bf{with a given moving boundary}}} was presented by Ladyzhenskaya in \cite{Lad70}.
{The proof was based on 
	using a time-discretization approach, known as Rothe's method, to construct solutions.} 
{In the context of deterministic fluid-structure interaction problems defined on  moving domains {\bf{not known a priori}} 
	a Rothe's-type method was first implemented in \cite{MC13} to construct weak solutions.}
{Rothe's  method was then extended to prove the existence of weak solutions to a number of different FSI problems including a FSI problem with 
	multi-layered  structures \cite{MC14} and a FSI problem  with
	the Navier-slip condition at the fluid-structure interface \cite{MC16}.
}

In terms of stochastic PDEs describing incompressible, viscous fluids, the existence of solutions to the stochastic Navier-Stokes equations and their properties have been the subject of numerous studies, see, e.g.,  \cite{BT73,DPD02,FG95}.
{However, in terms of analysis of stochastic {\bf{fluid-structure interaction}} involving incompressible fluids, to the best of our knowledge, there is only
	one work that addresses questions of existence of solutions to such problems \cite{KC23}. More precisely, in  \cite{KC23} the authors prove 
	the 	existence of a pathwise solution to a {\bf{linearly coupled}} FSI model, where the fluid and structure coupling conditions are evaluated along a {\bf{fixed fluid structure interface}}, with a stochastically forced membrane. The present work is a  nontrivial extension of the results from \cite{KC23} to the nonlinearly coupled case, and to the case in which a stochastic forcing is applied not only to structure equations, but also to the bulk fluid in the form of a stochastic fluid body force.} The difficulties faced in the present manuscript are unique since in \cite{KC23} the fixed geometry in the problem does not lead to the same considerations and issues as the ones in this manuscript that arise due to the random and time-dependent motion of the fluid domain. Additionally, in \cite{KC23}, the time-dependent Stokes equations are used to model the fluid flow whereas {in the present paper} we consider the Navier-Stokes equations.

{The proof in the present manuscript is based on}  discretizing the problem in time by partitioning the time interval $(0,T)$ and constructing approximate solutions using an operator splitting Lie scheme. For time-splitting methods for stochastic equations see e.g. \cite{BGR92}, \cite{CHP12} and the references therein. 
At each time step, our multi-physics problem
is split into two subproblems: the structure subproblem and the fluid subproblem. { In the first subproblem, the structure displacement and the structure velocity are updated while keeping the fluid velocity the same. In the second subproblem, keeping the structure displacement unchanged, the fluid velocity is updated while also ensuring that the kinematic coupling condition is satisfied. Based on this splitting scheme, uniform energy estimates in expectation are derived. To deal with the highly nonlinear nature of the problem, we then employ compactness arguments { by obtaining uniform estimates on fractional time derivatives of order $<\frac12$ of the approximate solutions. } 
	We then obtain the necessary almost sure convergence results for the approximate solutions/stochastic processes by using a variant of the Skorohod Representation theorem given in \cite{VW96} which gives the existence and characterization of a (possibly new) probability space and a sequence of random variables with the same laws {{as the original variables}} that converge almost surely
	{{to a martingale solution of the original problem}}.} 

{The {\bf{main new component of the present work}} is the treatment of the geometric nonlinearity associated with the stochastic fluid domain motion. 
	The {\bf{main mathematical issues}}  are related to the following: (1) the {\bf{fluid domain boundary is a random variable, not known a priori}}, 
	(2) the {\bf{fluid domain can possibly degenerate in a random fashion}}, e.g., the top boundary can touch the bottom of the fluid domain in a random fashion, and (3) 
	the {\bf{incompressibility condition}} gives rise to the difficulties in the construction of appropriate test functions on stochastic moving domains.
	While similar issues arise in the deterministic case as well, their resolution in the stochastic case is remarkably different. In particular, to deal with the stochastic motion of the fluid domain 
	we map the equations {{defined}} on stochastic moving domains onto a fixed reference domain via the Arbitrary Lagrangian-Eulerian (ALE) mappings, which are random variables defined 
	{{pathwise}}.  The use of the ALE mappings and the analysis that follows is valid for as long as the compliant tube walls do not touch each other, i.e., until there is a loss of strict positivity of the Jacobian of the ALE transformation. {\emph{Handling this no-contact condition in the stochastic case}} requires a delicate approach, which in this manuscript is done using a {\emph{cut-off function and a stopping time argument}}. The cut-off function artificially maintains a minimum distance of $\delta>0$ between the walls of the tube, preventing them from coming into contact while also providing the required regularity for any realization of the structure. This is crucial as it provides a {\it deterministic} lower (and upper) bound for the artificial structure displacement. The stopping time argument is then developed to 
	show that this cut-off function does not ``kick in'' until some stopping time which is {\bf{strictly positive}} almost surely. 
	Thus, we show that there exists an a.s. positive time until the original FSI problem has a martingale solution. 
}

{Finally, the {\bf{incompressibility condition}} is a challenge associated with the construction of test functions, which depend on the random moving domains, when they are mapped onto a fixed domain. 
	This problem arises, { among other places}, in the application of} the Skorohod representation theorem to upgrade convergence results. 

Using deterministic techniques, which are based on constructing appropriate test functions associated with approximate {\sl{deterministic}} moving domains, is
not applicable here.  To overcome this difficulty 
we {introduce} an {\emph{augmented system, which is not divergence free, but is an approximation of the original divergence-free system 	due to an added	singular term that penalizes the transformed divergence}} using the parameter $\ep>0$. { While this gets around the difficulties related to working with random test functions and random phase spaces, addition of the penalty term together with the low expected temporal regularity of the solutions create problems in obtaining compactness (tightness) results. Hence, we apply non-standard compactness arguments by constructing appropriate test functions that result into bounds on fractional time derivative of the solutions that are independent of $\ep$ and the time-step. By doing so, we reveal some hidden time regularity of the weak solutions to our FSI problem. Our tightness argument also bypasses the need for estimates on higher order moments of the solutions.} 
This new tightness argument is one of the novelties of this manuscript. 
We then show that the solutions to this approximate system indeed converge to a desired solution of the limiting equations as $\ep\rightarrow 0$

{ Since the {\bf{stochastic forcing}} appears not only in the structure equations but also in the {\bf{fluid equations themselves}}, we come across additional difficulties, which are associated with the construction of the 
	appropriate "test processes" on the approximate and limiting moving fluid domains.  Namely, along with the required divergence-free property on these domains, the test functions also have to satisfy appropriate measurability properties. We construct these approximate test functions in step 1 of the proof of Theorem \ref{exist2} by first constructing a Carath\'eodory function that gives the definition of a test function for the limiting equations
	 and then by "squeezing" this limiting test function in a way that preserves its desired properties  on the approximate domains,
	 see \eqref{r_ep} and \eqref{q_ep}.}

{Once these issues are taken care of, the final martingale solutions are obtained from our time-marching scheme in two steps.
	First, we consider the limits as the time step converges to zero
	to obtain solutions which satisfy an $\ep$ approximation of the the incompressibility condition.
	Then,  the penalty parameter $\ep$ is let go to zero to 
	obtain the martingale solutions to the original problem.  
}

\section{Problem setup}\label{sec:det_setup}

\subsection{The deterministic model and a weak formulation}
{We consider the flow of an incompressible, viscous fluid in a two-dimensional compliant cylinder with a  deformable lateral boundary whose 
	time-dependent motion is described by 
	a one-dimensional membrane/shell equation capturing displacement only in the vertical direction. The left and the right boundary of the cylinder
	are the inlet  and outlet for the time-dependent fluid flow. 
	We assume ``axial symmetry'' of the data and of the flow, rendering the central horizontal line as the axis of symmetry.
	This allows us to  consider the flow only in the upper half of the domain, with the bottom boundary fixed and equipped with the symmetry boundary conditions. 
	See Fig.~\ref{domain}.
	
	\begin{figure}[htp!]
\center
                    \includegraphics[width = 0.5 \textwidth]{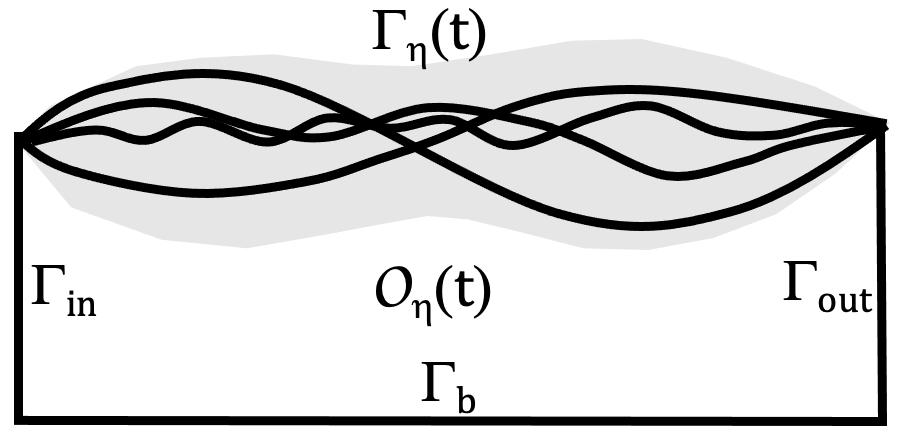}
  \caption{{\small\emph{A sketch of the fluid domain $\sO_{\eta}(t)$ with the elastic lateral boundary $\Gamma_\eta(t)$, the inlet and outlet boundaries
  $\Gamma_{in}$ and $\Gamma_{out}$, and the bottom (symmetry) boundary $\Gamma_b$. 
  The lightly shaded region represents a confidence interval of where the structure is likely to be.}}}
\label{domain}
\end{figure}

	The} time-dependent fluid domain, whose displacement is not known a priori will be denoted by 
$$\sO_{\eta}(t)=\{{\bx}=(z,r)\in\mathbb{R}^2:z\in (0,L), r \in (0,R+\eta(t,z))\},$$ where the top lateral boundary is given by $$\Gamma_{\eta}(t)=\{(z,r)\in\mathbb{R}^2:z\in (0,L), r =R+\eta(t,z)\}.$$

The inlet, outlet and bottom boundaries are	$\Gamma_{\text{in}}=\{0\}\times (0,R),\Gamma_{\text{out}}=\{L\}\times (0,R),\Gamma_b=(0,L)\times \{0\}$ respectively. \\

{\bf The fluid subproblem:}		
The fluid flow is modeled by the incompressible Navier-Stokes equations in the 2D time-dependent domains $\sO_{\eta}(t)$ 
\begin{equation}
\begin{split}\label{u}
\partial_t\bu + (\bu\cdot \nabla)\bu &= \nabla \cdot \sigma + F^{ext}_u\\
\nabla \cdot \bu&=0,
\end{split}
\end{equation}
where $\bu=(u_z,u_r)$ is the fluid velocity. The Cauchy stress tensor is $\sigma=-p I+2\nu \bD(\bu)$ where $p$ is the fluid pressure, $\nu$ is the kinematic viscosity coefficient and $\bD(\bu)=\frac12(\nabla\bu+(\nabla\bu)^T)$ is the symmetrized gradient. Here $F_u^{ext}$ represents any external forcing impacting the fluid,
which in this work will be stochastic.
	The fluid flow is driven by dynamic pressure data given at the inlet and the outlet boundaries as follows:
\begin{equation}
\begin{split}\label{bc:in-out}
p +\frac12 |\bu|^2&=P_{in/out}(t),\\
u_r&=0 \quad\text{ on } \Gamma_{in/out}.
\end{split}	\end{equation}
Whereas on the bottom boundary $\Gamma_b$ we prescribe the symmetry boundary condition:
\begin{align}
u_r=\partial_r u_z=0 \quad\text {on } \Gamma_b.\label{bc:bottom}
\end{align}		
{\bf The structure subproblem:} The elastodynamics problem is given
by the linearly elastic Koiter shell equations that describe the {vertical or }radial displacement
$\eta$:
\begin{align}\label{eta}
\partial^2_{t}\eta- \partial^2_{z}\eta +\partial^4_{z}\eta = F_\eta \quad \text{ in } (0,L),
\end{align}
where $F_{\eta}$ is the vertical component of the total force experienced by the structure.		{As we shall see below, $F_\eta$ will correspond to the difference between the fluid traction on one side and external random forcing on the other. }	
The above equation is supplemented with the following boundary conditions:
\begin{equation}\label{bc:eta}
\begin{split}
\eta(0)=\eta(L)=\partial_z\eta(0)=\partial_z\eta(L)=0.
\end{split}
\end{equation}
{\bf The non-linear fluid-structure coupling}: The coupling between the fluid and the structure takes place {across the current location of the fluid-structure interface, which is simply the current location of the membrane/shell, described above. We consider a two-way coupling described by }
the kinematic and dynamic coupling conditions that describe continuity of velocity and continuity of normal stress at the fluid-structure interface, respectively:
\begin{itemize}
	\item The kinematic coupling condition is:
	\begin{align}\label{kinbc}
	(0,\partial_t \eta(t,z))=\bu(t,z,R+\eta(t,z)).
	\end{align}
	\item
	The dynamic coupling condition is:
	\begin{align}
	F_\eta=-J(t,z) (\sigma {\bf n})|_{(t,z,\eta(t,z))}\cdot {\be_r} + F_\eta^{ext}\cdot {\be_r},
	\end{align}
	where $\bn$ is the unit outward normal to the top boundary, $\be_r$ is the unit vector in the {vertical/}radial direction and $J(t,z)=\sqrt{1+(\partial_z\eta(t,z))^2}$ is the Jacobian of the transformation from
	Eulerian to Lagrangian coordinates. As earlier, $F^{ext}_\eta$ denotes any external force impacting the structure, which in this work will be a stochastic.
\end{itemize}

This system is supplemented with the following initial conditions:
\begin{align}\label{ic}
\bu(t=0)=\bu_0,\eta(t=0)=\eta_0, \partial_t\eta(t=0)=v_0.
\end{align}
\vskip 0.1in
\noindent

\subsection*{Weak formulation on moving domains}
Before we derive the weak formulation of the deterministic system described in the previous subsection, we define the following relevant function spaces
{for the fluid velocity, the structure, and the coupled FSI problem}:
\begin{align*}
&\tilde\sV_F(t)= \{{\bf u}=(u_z,u_r)\in \bH^{1}(\sO_{\eta}(t)) :
\nabla \cdot {\bf u}=0,u_z=0 \text{ on } \Gamma_{\eta}(t), u_r=0 \text{ on }\partial \sO_{\eta}\setminus \Gamma_{\eta}(t) \}\nonumber\\
&\tilde\sW_F(0,T)=L^\infty(0,T;\bL^2(\sO_\eta(\cdot))) \cap L^2(0,T;\tilde\sV_F(\cdot))\\
&\tilde\sW_S(0,T)=W^{1,\infty}(0,T;L^2(0,L)) \cap L^\infty(0,T;H^2_0(0,L))\\
&\tilde\sW(0,T)=\{(\bu,\eta)\in \tilde\sW_F(0,T)\times\tilde\sW_S(0,T):\bu(t,z,R+\eta(t,z))=\partial_t\eta(t,z){\bf e}_r\}.
\end{align*} 
We will take test functions $(\bq,\psi)$ from the following space:
{
	\begin{equation*}
	\sD^\eta(0,T)=\{({\bf q},\psi) \in C^1([0,T];\tilde\sV_F(\cdot) \times H^2_0(0,L)): {\bf q}(t,z,R+\eta(t,z))=\psi(t,z){\bf e}_r\}.
	\end{equation*}
}

{Next, we derive} a weak formulation of the problem on the moving domains. We begin by considering the fluid equations \eqref{u}. We multiply these equations by $\bq$, integrate in time and space and use Reynold's transport theorem to obtain
\begin{align*}
&(\bu(t), \bq(t))_{\sO_{\eta}(t)}= (\bu(0), \bq(0))_{\sO_{\eta}(0)} +\int_0^t\int_{\sO_{\eta}(s)} \bu(s)\cdot\partial_s\bq(s)d\bx ds \\
&+ \int_0^t\int_{\Gamma_{\eta}(s)}|\bu(s)|^2\bq(s)\cdot \bn(s)dSds
-\int_0^t\int_{\sO_{\eta}(s)}(\bu(s)\cdot \nabla )\bu(s) \bq(s)d\bx ds\\
&- 2\nu\int_0^t \int_{\sO_{\eta}(s)}  \bD(\bu(s))\cdot\bD(\bq(s)) d\bx ds+\int_0^t\int_{\partial\sO_{\eta}(s)} (\sigma \bn(s))\cdot \bq(s)dS ds\\
&+\int_0^t\int_{\sO_{\eta}(s)}F^{ext}_u(s)\bq(s)d\bx ds.
\end{align*}
Let us introduce the following notation:
$$b(t,\bu,{\bf v},{\bf w}):=\frac12\int_{\sO_{\eta}(t)}\left( (\bu\cdot\nabla){\bf v}\cdot{\bf w}-(\bu\cdot\nabla){\bf w}\cdot{\bf v}\right)  d\bx.$$
{We calculate:}
\begin{align*}
-((\bu\cdot\nabla)\bu,\bq)_{\sO_{\eta}}&=-\frac12((\bu\cdot\nabla)\bu,\bq)_{\sO_{\eta}}+\frac12((\bu\cdot\nabla)\bq,\bu)_{\sO_{\eta}} -\frac12\int_{\partial\sO_{\eta}}|\bu|^2\bq\cdot \bn dS\\
&=-b(s,\bu,\bu,\bq) -\frac12\int_{\Gamma_{\eta}}|\bu|^2\bq\cdot \bn dS+ \frac12\int_{\Gamma_{in}}|\bu|^2q_zdS-\frac12\int_{\Gamma_{out}}|\bu|^2q_zdS.
\end{align*}
Since $\bn(t)=\frac{1}{\sqrt{1+\partial_z\eta(z,t)^2}}(-\partial_z\eta(t),1)$, we  obtain
$$\int_{\Gamma_{\eta}}|\bu|^2\bq\cdot \bn dS=\int_0^L(\partial_t\eta)^2\psi dz.$$
Now using the divergence free property of fluid velocity $\bu$ and the fact that $u_r=0$ on $\Gamma_{in/out}$ we have that $\partial_ru_r=-\partial_zu_z=0$ on $\Gamma_{in/out}$. Hence we obtain
\begin{align*}
\int_{\Gamma_{in/out}}\sigma \bn \cdot \bq dS=\int_{\Gamma_{in/out}}\pm p\, q_z dS = \int_{\Gamma_{in}} \left( P_{in}-\frac12|\bu|^2\right) q_z dr
-\int_{\Gamma_{out}} \left( P_{out}-\frac12|\bu|^2\right)q_z dr,
\end{align*}
whereas
$\int_{\Gamma_{b}}\sigma \bn \cdot \bq dS=0.$

Next we consider the structure equation \eqref{eta}. We multiply \eqref{eta} by $\psi$ and integrate in time and space {to obtain}
\begin{align*}
(\partial_t\eta(t),\psi(t))&=(v_0,\psi(0)) +\int_0^t\int_0^L\partial_s\eta\partial_s\psi dzds - \int_0^t\int_0^L (\partial_z\eta \partial_z\psi+  \partial_{zz}\eta\partial_{zz} \psi)dzds  \\
&-\int_0^t\int_0^LJ\sigma \bn\cdot \be_r \psi dzds +\int_0^t \int_0^LF_\eta^{ext}\psi dz ds.
\end{align*}
Thus, we have obtained the following weak formulation of the deterministic problem: 
for any test function $\bQ=(\bq,\psi) \in \sD^\eta(0,T)$ we look for  $(\bu,\eta) \in \tilde\sW(0,T)$,  such that the following equation is satisfied for {almost} every $t \in [0,T]$:
\begin{equation}
\begin{split}\label{origweakform}
&{{}\int_{\sO_{\eta(t )}}\bu(t )\bq(t ) d\bx+\int_0^L\partial_t\eta(t)\psi(t )dz}-\int_0^{t }\int_{\sO_{\eta(s)}}\bu\cdot\partial_s\bq d\bx ds\\ &+\int_0^{t} b(s,\bu,\bu,\bq)ds + 2\nu\int_0^{t} \int_{\sO_{\eta(s)}} \bD(\bu)\cdot \bD(\bq) d\bx ds\\
&-\frac12\int_0^{t}\int_{0}^L(\partial_s\eta)^2\psi dzds-\int_0^{t}\int_0^L\partial_s\eta\partial_s\psi dzds +\int_0^{t }\int_0^L (\partial_z\eta \partial_z\psi+  \partial_{zz}\eta\partial_{zz} \psi)dzds\\
&=\int_{\sO_{\eta_0}}\bu_0\bq(0) d\bx+ \int_{0}^L v_0\psi(0) dz
+\int_0^{t}P_{{in}}\int_{0}^1q_z\Big|_{z=0}drds-\int_0^tP_{{out}}\int_{0}^1q_z\Big|_{z=1}drds \\
& +\int_0^{t}\int_{\sO_{\eta}(s)}{\bq} \cdot {F_{u}^{ext}}\, d\bx ds 
+\int_0^{t}\int_0^L{ \psi} {F_{\eta}^{ext}}\, dz ds,
\end{split}\end{equation}
where ${F_{u}^{ext}}$ is the volumetric external force and $F_{\eta}^{ext}$ is the external force applied to the deformable boundary.

\subsection{The stochastic framework and definition of martingale solutions}
{{We will incorporate the stochastic effects by considering $F_u^{ext}$ and ${F_{\eta}^{ext}}$ to be multiplicative stochastic forces 
		determined by the
		{\bf noise coefficient} $G(\bu,v,\eta)$ defined on a product space made precise in Assumption~\ref{G} below, satisfying certain assumptions
		specified in \eqref{growthG}. We can then write the combined stochastic forcing $F^{ext}$ in terms of this $G$ as follows:
		\begin{equation}\label{StochasticForcing}
		{F^{ext}} := G(\bu,v,\eta) {dW},
		\end{equation}
		where $\bu$ is the fluid velocity, $v$ is the structure velocity, $\eta$ is the structure displacement, and $W$ is a Wiener process. 
		More precisely, the stochastic noise term is
}}
defined on a filtered probability space $(\Omega,\sF,(\sF_t)_{t \geq 0},\bP)$ that satisfies the usual
assumptions, i.e., $\sF_0$ is complete and the filtration is right continuous, that is, $\sF_{t}=\cap_{s \geq t}\sF_s$ for all $t \geq 0$. 
We assume that $W$ is a $U$-valued Wiener process with respect to the filtration $(\sF_t)_{t \geq 0}$, where $U$ is a separable Hilbert space. 
{{To specify the properties of the noise coefficient $G$, which will be done below in Assumption~\ref{G}, we introduce $Q$ 
		to be the covariance operator of $W$, which is a positive, trace class operator on $U$, 
		and define $U_0:=Q^{\frac12}(U)$. 
		
		We will use this framework to define a martingale solution to our stochastic FSI problem.
		In order to do this, we first transform the problem defined on moving domains onto a fixed reference domain using
		a family of Arbitrary Lagrangian-Eulerian (ALE) mappings. The mappings are defined next.
}}
\subsubsection{ALE mappings}\label{sec:ale}
As mentioned in the introduction, the geometric nonlinearity arising due to the motion of the fluid domain {will be} handled by the arbitrary Lagrangian-Eulerian (ALE) mappings which are a family of diffeomorphisms, parametrized by time $t$, from the fixed domain $\sO=(0,L) \times (0,1)$ onto {the moving domain}  $
\sO_{\eta}(t)$. Notice that the presence of the stochastic forcing implies that the domains $\sO_\eta$ are themselves random.
Hence, to take into account our stochastic setting we consider 
for any {sample} $\omega\in \Omega$, the ALE mappings defined {\it pathwise}:
\begin{align}\label{ale}
A_{\eta}^\omega(t):\sO \rightarrow \sO_{\eta}(t,\omega)\ {\rm given \ by} \ 
A_{\eta}^\omega (t)( z, r)= ( z,(R+\eta(t,z,\omega)) r).
\end{align}
Then, given any $\omega\in\Omega$, for as long as the Jacobian of the ALE mapping
\begin{equation}\label{ALE_Jacobian}
|\text{det }\nabla A_{\eta}^\omega(t)|=|R+\eta(t,z,\omega)|
\end{equation}
is bounded and bounded away from zero and continuous, 
the map
\begin{align}\label{Feta}
F_t:L^2(\sO_\eta(t))\rightarrow L^2(\sO)\ {\rm given \ by} \ 
F_t(f)(z,r)=f[(A_{\eta}^\omega(t)(z,r))]
\end{align}
is well-defined and 
$$\|F_t(f)\|_{L^2(\sO)} \leq C\|f\|_{L^2(\sO_\eta(t))},$$
where the constant $C>0$ depends on the lower and upper bounds of the Jacobian of the ALE mappings.
Hereon, we will suppress the notation $\omega$ and the dependence of the variables on $\omega$ will be understood implicitly.

{Under the transformation given in \eqref{ale}, the pathwise transformed gradient and {symmetrized gradient are} given by
	$$\nabla^\eta=\left( \partial_z-r\frac{\partial_z\eta}{R+\eta}\partial_r, \frac1{R+\eta}\partial_r\right) \text{ and } \bD^\eta(\bu)=\frac12(\nabla^\eta\bu+(\nabla^\eta)^T\bu). $$
	We use $\bw^\eta$ to denote the ALE velocity:
	$$\bw^\eta=\partial_t\eta r\be_r,$$}
and we rewrite the advection term as follows:
$$b^\eta(\bu,\bw,\bq)=\frac12\int_{\sO}(R+\eta)\left(((\bu-\bw)\cdot\nabla^\eta )\bu\cdot\bq-((\bu-\bw)\cdot\nabla^\eta )\bq\cdot\bu\right). $$
{We are now in a position to define the notion of martingale solutions to our stochastic FSI problem.}

\subsubsection{Definition of martingale solutions} 
{We start by introducing} the {\bf functional framework for the stochastic problem on the fixed reference domain} $\sO=(0,L) \times (0,1)$. {For this purpose we will denote by $$\Gamma=(0,L)\times \{1\}$$
	the reference configuration of the moving domain. The following are the function spaces for the stochastic FSI problem defined on the fixed domain $\sO$:}
\begin{align}
&\label{V} V= \{{\bf u}=(u_z,u_r)\in {\bH}^{1}(\sO): 
u_z=0 \text{ on } \Gamma, u_r=0 \text{ on }\partial \sO_{}\setminus \Gamma \},\\
&\sW_F=L^2(\Omega;L^\infty(0,T;\bL^2(\sO))) \cap L^2(\Omega;L^2(0,T;V)),\\
&\sW_S=L^2(\Omega;W^{1,\infty}(0,T;L^2(0,L)) \cap L^\infty(0,T;H^2_0(0,L))),\\
&\sW(0,T)=\{(\bu,\eta)\in \sW_F\times \sW_S:
\bu(t,z,1)=\partial_t\eta(t,z){\bf e}_r \text{ and } \nabla^\eta \cdot {\bu}=0\, \,\bP-a.s.\}.
\end{align} 
The test space is defined as follows:
{
	\begin{equation}\label{Dspace}
	\sD=\{({\bf q},\psi) \in  V \times H^2_0(0,L): {\bf q}(z,1)=\psi(z){\bf e}_r
	\}.
	\end{equation}
}
Very often we will work with the following space for the fluid and structure velocities:
\begin{align}\label{su}
\sU=\{(\bu,v)\in V\times L^2(0,L):\bu(z,1)=v(z)\be_r \}.
\end{align}
{We further introduce the following notation:
	$$\bL^2:=\bL^2(\sO)\times L^2(0,L).$$}

Before we define a martingale solution we specify the stochastic noise as follows.
{\begin{assm}\label{G}
		Let $L_2(X,Y)$ denote the space of Hilbert-Schmidt operators from a Hilbert space $X$ to another Hilbert space $Y$. The {\bf noise coefficient} $G$ is a function $G:\sU\times L^\infty(0,L) \rightarrow L_2(U_0;\bL^2)$ such that the following Lipschitz continuity assumptions hold:
		\begin{equation}\begin{split}\label{growthG}
		&\|G(\bu,v,\eta)\|_{L_2(U_0;\bL^2)} \leq \|\eta\|_{L^\infty(0,L)}\|{\bf u}\|_{\bL^2(\sO)} + \|v\|_{L^2(0,L)},\\
		&\|G(\bu_1,v_1,\eta)-G(\bu_2,v_2,\eta)\|_{L_2(U_0;\bL^2)} \leq \|\eta\|_{L^\infty(0,L)}\|{\bu_1}-{\bu_2 }\|_{\bL^2(\sO)} + \|v_1-v_2\|_{L^2(0,L)},\\
		&\|G(\bu,v,\eta_1)-G(\bu,v,\eta_2)\|_{L_2(U_0;\bL^2)} \leq \|\eta_1-\eta_2\|_{L^\infty(0,L)}\|\bu\|_{\bL^2(\sO)}.
		\end{split}\end{equation}	
	\end{assm}
	As example of such a noise coefficient $G$ is a linearly multiplicative noise transformed onto the fixed domain. More precisely, $G$ can take the form}
\begin{align}
(G(\bu,v,\eta),(\bq,\psi)):=\left( (R+\eta)\bu,\bq) + (v,\psi)\right) \cdot\Phi,
\end{align}
where $\Phi \in L_2(U_0;\mathbb{R})$ and $(\cdot,\cdot)$ is the inner product on $\bL^2$. {We remark that nonlinear noise} coefficients are allowed in our analysis as well.

We are now in position to define solutions to the system \eqref{u}-\eqref{ic} in stochastic setting. 
\begin{definition}[Martingale solution]\label{def:martingale}
	Let   $\bu_0 \in \bL^2(\sO)$ and $v_0\in L^2(0,L)$ be deterministic initial data and let $\eta_0\in H_0^2(0,L)$ be 
	a compatible initial structure configuration 
		such that
		for some $\delta>0$  the initial configuration satisfies 
	 \begin{align}	 \label{etainitial}
	{\delta<R+\eta_0(z),\quad \forall z\in [0,L],\quad \text{ and }\quad { \|R+\eta_0\|_{H^2_0(0,L)}<\frac1{\delta}}.}
	\end{align}
	We say that  $(\mathscr{S},\bu,\eta,\tau)$ is a  martingale solution to the system \eqref{u}-\eqref{ic} under the assumptions \eqref{growthG} if: 
	\begin{enumerate}
		\item  $\mathscr{S}=(\Omega,\sF,(\sF_t)_{t\geq 0},\bP,W)$ is a stochastic basis, that is, $(\Omega,\sF,(\sF_t)_{t\geq 0},\bP)$ is a filtered probability space  and $W$ is a $U$-valued Wiener process;
		\item $(\bu,\eta)\in \sW(0,T)$;
		\item $\tau$ is a $\bP$-a.s. strictly positive $(\sF_t)_{t\geq 0}-$stopping time;
		\item 
		$ \bU=(\bu,\partial_t\eta)$ and $\eta$ are  $(\sF_t)_{t \geq 0}-$progressively measurable;   
		\item	For every $(\sF_t)_{t \geq 0}-$adapted, essentially bounded process $\bQ:=(\bq,\psi)$ with $C^1$ paths in $\sD$ such that 
		$\nabla^\eta \cdot \bq=0$, the following equation holds $\bP-$a.s. for almost every $t \in[0,\tau)$:
	\begin{equation}\label{weaksol}
	\begin{split}
	&{\int_{\sO}(R+\eta(t))\bu(t)\bq(t) d\bx +\int_0^L\partial_t\eta(t)\psi(t)dz}= \int_{\sO}(R+\eta_0)\bu_0\bq(0) d\bx  + \int_{0}^L v_0\psi(0) dz \\
	&+\int_0^{t }\int_{\sO}(R+\eta(s))\bu(s)\cdot \partial_s\bq(s) d\bx ds +\frac12\int_0^{t } \int_\sO (\partial_s\eta(s))\bu(s)\cdot\bq(s) d\bx ds\\
	&-\int_0^{t }b^{\eta}(\bu(s),\bw^\eta(s),\bq(s))ds- 2\nu\int_0^{t } \int_{\sO} (R+\eta(s))\bD^\eta(\bu(s))\cdot \bD^\eta(\bq(s)) d\bx ds\\
	&+\int_0^{t }\int_0^L\partial_s\eta(s)\partial_s\psi(s) dzds -\int_0^{t }\int_0^L \partial_z\eta(s)\partial_z \psi(s) +\partial_{zz}\eta(s)\partial_{zz} \psi(s) dzds\\
	&+\int_0^{t }\left( P_{{in}}\int_{0}^1q_z\Big|_{z=0}dr-P_{{out}}\int_{0}^1q_z\Big|_{z=1}dr\right) ds+\int_0^{t }(\bQ(s),G(\bU(s),\eta(s))\,dW(s)).
	\end{split}\end{equation}
\end{enumerate}
\end{definition}
The main result of this manuscript is the proof of the existence of martingale solutions. 
The proof is constructive, and it relies on the following operator splitting scheme.
\section{Operator splitting scheme} \label{sec:splitscheme}
In this section we introduce a Lie operator splitting scheme that defines a sequence of approximate solutions {to \eqref{weaksol} by semi-discretizing the problem in time}. {The final goal is to show that up to a subsequence, approximate solutions converge in a certain sense to a martingale solution of the stochastic FSI problem. }

{\subsection{Definition of the splitting scheme} We semidiscretize the problem in time and use operator splitting to divide the coupled stochastic problem into two subproblems, a fluid and a structure subproblem.}
We denote the time step by $\Delta t=\frac{T}{N}$ and use the notation $t^n=n\Delta t$ for $n=0,1,...,N$. 

Let $(\bu^0,\eta^0,v^0)=(\bu_0,\eta_0,v_0)$ {be the initial data. At the $i^{th}$ time level}, we update the vector $(\bu^{n+\frac{i}{2}},\eta^{n+\frac{i}{2}},v^{n+\frac{i}{2}})$, where $i=1,2$ and $n=0,1,2,...,N-1$, {according to the following scheme}.
\subsection*{The structure subproblem} 
We update the structure displacement and the structure velocity while keeping the fluid velocity the same. That is, given $(\eta^n,v^n) \in H^2_0(0,L)\times L^2(0,L)$ we look for a pair $(\eta^{n+\frac12},v^{n+\frac12}) \in H^2_0(0,L) \times H^2_0(0,L)$ that satisfies the following equations pathwise i.e. for each $\omega\in\Omega$:
\begin{equation}
\begin{split}\label{first}
\bu^{n+\frac12}&=\bu^n,\\
\int_0^L(\eta^{n+\frac12}-\eta^n) \phi dz&= (\Delta t)\int_0^L v^{n+\frac12}\phi dz,\\
\int_0^L \left( v^{n+\frac12}-v^n\right)  \psi dz &+ (\Delta t)\int_0^L \partial_{z}\eta^{n+\frac12}  \partial_{z}\psi + \partial_{zz}\eta^{n+\frac12}  \partial_{zz}\psi dz=0,
\end{split}
\end{equation}
for any $\phi \in L^2(0,L)$ and $\psi \in H^2_0(0,L)$.

Before commenting on the existence of the random variables $\eta^{n+\frac12},v^{n+\frac12}$ and their measurability properties, we introduce the second subproblem.
\subsection*{The fluid subproblem}  {In this step we update the fluid and the structure velocities while keeping the structure displacement unchanged. 
	See, e.g., \cite{MC13, MC16, MC19} for the deterministic case. In the stochastic case, however, there are two major difficulties that need to be overcome	before we can even define the fluid subproblem, which are both 
	due to	the fact the the fluid domain is a random variable not known a priori. }
\begin{enumerate}
	\item First, as the problem is mapped onto the fixed, reference domain $\sO$, 
		see \eqref{weaksol}, the {\bf{Jacobian of the ALE transformation}} \eqref{ALE_Jacobian} appears is several terms. 
		Since at every time step $n$ the Jacobian depends on the random variable $\eta^n$, we cannot introduce the maximal and minimal displacements, as is done in the deterministic case, see \cite{MC13, MC16, MC19}, to obtain a uniform upper and lower bound of the Jacobian for every realization. 
		Hence we introduce an "artificial" structure displacement variable $\eta_*^n$, using a cutoff function, that has suitable deterministic bounds. This artificial structure displacement is defined and { introduced in the problem} in a way that ensures the stability of our time-marching scheme. At every time step we then solve the fluid equations with respect to this artificial domain, determined by $\eta_*^n$, and then show in Lemma~\ref{StoppingTime}  that there exists a stopping time that is a.s. strictly greater than zero, such that the cutoff in the definition of the artificial displacement is, in fact,  never used until possibly that stopping time, thereby providing a solution to the original problem until the stopping time.
	\item  Second, for the problem mapped onto the fixed domain $\sO$, the {\bf{divergence-free condition}} is a problem when working with the mapped test functions, 
	as we explain below. To avoid working with the divergence-free test functions in the approximate problems, we relax the divergence-free condition by
	supplementing the weak formulation by a penalty term of the form $\frac1{\ep}\int \text{div}^\eta \bu \ \text{div}^\eta\bq$,
	and work with artificial compressibility until the very end, when we recover the solution to the original incompressible (divergence-free) problem by letting  $\ep \rightarrow 0$. See Section~\ref{sec:limit2}. 
	To explain why working with the  divergence-free condition is a problem in this stochastic setting, we note that 
for the fluid velocity defined on the fixed domain $\sO$ to satisfy the transformed divergence-free condition at the time step $n+1$, the test functions in the fluid equations in this subproblem would normally depend on the structure displacement $\eta^n$ found in the previous subproblem. 
This means that we would use the test functions $\bq^{n+1} \in \sU$ that satisfy div$^{\eta^n_*}\bq^{n+1}=0$ (see the weak formulation \eqref{weaksol}). This dependence of the fluid test functions on $\eta^n_*$, which in our case is a random variable, causes problems in the second part of the existence proof as we construct a new probability space as a part of a martingale solution.
It is not clear that all the admissible test functions with respect to the new probability space would necessarily come from the divergence-free test functions built in this subproblem. 
This is why we introduce artificial compressibility via the penalty term mentioned above, and avoid dealing with the divergence-free condition until the
very end, when we recover the divergence-free solution in Section~\ref{sec:limit2}.
\end{enumerate} 
\vskip 0.1in
\noindent
{\bf{Construction of the cut-off displacement.}}
For $\delta>0$, let $\Theta_\delta$ be the
step function such that
$\Theta_{\delta}(x,y)=1$ if $\delta < x, \text{} y<  \frac1{\delta}$, 
and $\Theta_{\delta}(x,y)=0$ otherwise.  
Using this function,we define {a real-valued function $\theta_\delta(\eta^n)$ which tracks all the structure displacements until the time step $n$, and is equal to 1 until the step for which the structure leaves the desired bounds given in terms of $\delta$:}
\begin{align}\label{theta}
\theta_\delta(\eta^n):=\min_{k\leq n}\,
\Theta_{\delta}\left( \inf_{z\in[0,L]}(\eta^k(z)+R), \|\eta^k+R\|_{H^s(0,L)}\right), \end{align}
where $s \in (\frac32,2)$.
Now we define the artificial structure displacement random variable as follows:
\begin{align}\label{eta*}\eta^{n}_*(z,\omega)=\eta^{\max_{0\leq k\leq n}\theta_\delta(\eta^k)k}(z,\omega)\quad \text{for every } \omega \in \Omega.\end{align}
{Note that the superscript in the definition above indicates the time step and not the power of structure displacement.}
\vskip 0.1in
\noindent
{{\bf A divergence-free relaxation via penalty.} We introduce a divergence-free penalty term and define the fluid subproblem as follows}. Let $\Delta_n W:=W(t^{n+1})-W(t^n)$.\\
Then for $\ep>0$ and given $\bU^n=(\bu^n,v^n) \in \sU$ we look for $ (\bu^{n+1},v^{n+1})$ that solves
\begin{equation}
\begin{split}\label{second}
&\qquad\qquad\eta^{n+1}:=\eta^{n+\frac12}, \\
&\int_{\sO}(R+\eta^n_*)\left( \bu^{n+1}-\bu^{n+{\frac12}}\right) \bq d\bx  +\frac{1}2\int_\sO \left( \eta_*^{n+1}-\eta_*^n\right)  \bu^{n+1}\cdot\bq d\bx \\
&{\cred +\frac12(\Delta t)\int_{\sO}(R+\eta_*^n)((\bu^{n+1}-
v^{n+1}r\be_r)\cdot\nabla^{\eta_*^n}\bu^{n+1}\cdot\bq - (\bu^{n+1}-
v^{n+1}r\be_r)\cdot\nabla^{\eta_*^n}\bq\cdot\bu^{n+1})d\bx} \\
&+2\nu(\Delta t)\int_{\sO_{}}(R+\eta^n_*) \bD^{\eta_*^{n}}(\bu^{n+1})\cdot \bD^{\eta_*^{n}}(\bq) d\bx  + \frac{(\Delta t)}{\ep}\int_\sO %(R+\eta^{n}_*)
\text{div}^{\eta^n_*}\bu^{n+1}\text{div}^{\eta^n_*}\bq d\bx 
+\int_0^L(v^{n+1}-v^{n+\frac12} )\psi dz \\
&= (\Delta t)\left( P^n_{{in}}\int_{0}^1q_z\Big|_{z=0}dr-P^n_{{out}}\int_{0}^1q_z\Big|_{z=1}dr\right) +  
(G(\bU^{n},\eta_*^n)\Delta_n W, \bQ),
\end{split}
\end{equation}
for any $\bQ=(\bq,\psi) \in \sU$ with 
$$ \bu^{n+1}|_{\Gamma}=v^{n+1}\be_r.$$
Here,
$$P^n_{in/out}:=\frac1{\Delta t}\int_{t^n}^{t^{n+1}}P_{in/out}\,dt,\quad \text{div}^{\eta}\bu=tr(\nabla^\eta\bu).$$
\begin{remark}
	Observe that using the data from the second subproblem, we update $\eta^n$ in the first subproblem and not $\eta_*^n$. This is crucial for obtaining a discrete version of the energy inequality and a stable time-marching scheme. 
\end{remark}

{We are now ready to prove the existence of solutions to the two subproblems. For this purpose} we introduce the following discrete energy and dissipation for $i=0,1$:
\begin{equation}\label{EnDn}
\begin{split}
E^{n+\frac{i}2}&=\frac12\Big(\int_{\sO}(R+\eta^{n}_*)|\bu^{n+\frac{i}2}|^2 d\bx
+\|v^{n+\frac{i}2}\|^2_{L^2(0,L)}+\|\partial_z\eta^{n+\frac{i}2}\|^2_{L^2(0,L)}+\|\partial_{zz}\eta^{n+\frac{i}2}\|^2_{L^2(0,L)}\Big),\\
D^{n}&=\Delta t \int_{\sO}\left( 2\nu(R+\eta_*^n)  |\bD^{\eta_*^{n}}(\bu^{n+1})|^2+\frac1{\ep}|\text{div}^{\eta^n_*}\bu^{n+1}|^2  \right) d\bx .
\end{split}\end{equation}
\begin{lem}[Existence for the structure subproblem]
	Assume that $\eta^n$ and $v^{n}$ are $H^2_0(0,L)$ and $L^2(0,L)$-valued $\sF_{t^n}$-measurable random variables, {respectively}. Then there exist $H^2_0(0,L)$- valued $\sF_{t^n}$-measurable random variables $\eta^{n+\frac12},v^{n+\frac12}$ that solve \eqref{first}, {and the following semidiscrete energy inequality holds:
		\begin{equation}\label{energy1}
		\begin{split}
		E^{n+\frac12} +C_1^n = E^{n},
		\end{split}
		\end{equation}
		where
		$$C^{n}_1:= { \frac12\|v^{n+\frac12}-v^n\|_{L^2(0,L)}^2 } +\frac12\|\partial_z\eta^{n+\frac12}-\partial_z\eta^{n}\|_{L^2(0,L)}^2 +\frac12 \|\partial_{zz}\eta^{n+\frac12}-\partial_{zz}\eta^{n}\|_{L^2(0,L)}^2,$$ corresponds to numerical dissipation.}
\end{lem}
\begin{proof}
	The proof of existence and uniqueness of measurable solutions is straightforward thanks to the linearity of equations \eqref{first} (see \cite{KC23}). %and assumptions on $K$.
	Furthermore we can write
	$$v^{n+\frac12}=\frac{\eta^{n+\frac12}-\eta^n}{\Delta t}.$$
	Using this pathwise equality while taking $\psi=v^{n+\frac12}$ in \eqref{first}$_3$ and using $a(a-b)=\frac12(|a|^2-|b|^2+|a-b|^2)$,  we obtain
{	\begin{equation}\label{energy1_1}
	\begin{split}	
	&\|v^{n+\frac12}\|_{L^2(0,L)}^2 +\|v^{n+\frac12}-v^n\|_{L^2(0,L)}^2 + \|\partial_z\eta^{n+\frac12}\|_{L^2(0,L)}^2 +\|\partial_z\eta^{n+\frac12}-\partial_z\eta^{n}\|_{L^2(0,L)}^2\\
	& +\|\partial_{zz}\eta^{n+\frac12}\|_{L^2(0,L)}^2+\|\partial_{zz}\eta^{n+\frac12}-\partial_{zz}\eta^n\|_{L^2(0,L)}^2\\
	&=\|v^n\|_{L^2(0,L)}^2+\|\partial_{z}\eta^{n}\|_{L^2(0,L)}^2+ \|\partial_{zz}\eta^{n}\|_{L^2(0,L)}^2.
	\end{split}	\end{equation}}
	Recalling  that $\bu^n=\bu^{n+\frac12}$ and adding the relevant terms on both sides of \eqref{energy1_1} we obtain \eqref{energy1}.
\end{proof}
\begin{lem}[Existence for the fluid subproblem]\label{existu}
	For given $\delta>0$, and given $\sF_{t^n}$-measurable random variables $(\bu^{n+\frac12},v^n)$ taking values in $\sU$ and $v^{n+\frac12}$ taking values in $H^2_0(0,L)$, there exists an $\sF_{t^{n+1}}$-measurable random variable $(\bu^{n+1},v^{n+1})$ taking values in $\sU$ that solves \eqref{second}, {and the solution satisfies the following energy estimate
		\begin{equation}\label{energy2}
		\begin{split}
		E^{n+1}+D^{n}+C_2^{n}&\leq E^{n+\frac12} +C\Delta t((P^n_{in})^2+(P^n_{out})^2) + {C}\|\Delta_nW\|_{U_0}^2
		\|G(\bU^{n},\eta_*^n)\|_{L_2(U_0;\bL^2)}^2\\
		&+\|
		(G( \bU^{n},\eta_*^n)\Delta_n W, \bU^{n}) \|+\frac14\int_0^L(v^{n+\frac12}-v^n)^2dz
		\end{split}\end{equation}
		where
		$$C_2^{n}:=\frac14\int_\sO (R+\eta_*^n)\left( |\bu^{n+1}-\bu^{n}|^2 \right) d\bx +\frac14\int_0^L|v^{n+1}-v^{n+\frac12}|^2 dz$$
		is numerical dissipation, and $\eta^n_*$ is as defined in \eqref{eta*}.}
\end{lem}
\begin{proof}
The proof is based on Galerkin approximation and a fixed point argument.
	For $\omega\in \Omega$,  {introduce} the following form on $\sU$, {which is defined by the fluid subproblem: }
	\begin{align*}
	&\langle \sL^\omega_n(\bu,v),(\bq,\psi)\rangle:=\int_{\sO}(R+\eta_*^n)\left( { \bu}-\bu^{n+{\frac12}}\right) \bq d\bx +\frac{1}2\int_\sO \left( \eta_*^{n+1}-\eta_*^n\right)  \bu\cdot\bq d\bx \\
	&+\frac{(\Delta t)}2\int_{\sO}(R+\eta_*^n)(({\bu}-{v}r\be_r)\cdot\nabla^{\eta_*^n}\bu\cdot\bq - ({ \bu}-{v}r\be_r)\cdot\nabla^{\eta_*^n}\bq\cdot\bu^{})d\bx \\
	&+2\nu(\Delta t)\int_{\sO_{}} (R+{\eta^{n}_*})\bD^{\eta^{n}_*}(\bu^{})\cdot \bD^{\eta^{n}_*}(\bq) d\bx  
	+\frac{\Delta t}{\ep}\int_{\sO} \text{div}^{\eta^n_*}\bu\text{div}^{\eta^n_*}\bq d\bx 
	+\int_0^L(v^{}-v^{n+\frac12} )\psi dz \\
	&- (\Delta t)\left( P^n_{{in}}\int_{0}^1q_z\Big|_{z=0}dr-P^n_{{out}}\int_{0}^1q_z\Big|_{z=1}dr\right)  - (G(\bU^{n},\eta_*^n)\Delta_n W, \bQ).
	\end{align*}
	{{Using this form the fluid subproblem \eqref{second} can be written as $\langle \sL^\omega_n(\bu,v),(\bq,\psi)\rangle = 0$.}}
	Observe that {if we replace the test function $(\bq,\psi)$ by} $\bU=(\bu,v)$, we obtain
	\begin{align*}
	\langle \sL^\omega_n(\bu,v),(\bu,v)\rangle&=\frac12\int_{\sO}\left((R+\eta_*^n)+(R+\eta_*^{n+1})\right)  |\bu|^2  d\bx  +\int_0^Lv^2dz\\
	&+2\nu(\Delta t)\int_{\sO_{}} (R+{\eta_*^{n}})|\bD^{\eta_*^{n}}(\bu^{})|^2 d\bx  \\
	&+\frac{\Delta t}\ep\int_{\sO} |\text{div}^{\eta^n_*}\bu|^2d\bx -\int_0^Lv^{n+\frac12} v dz - \int_{\sO}(R+\eta_*^n) \bu^{n+{\frac12}}\cdot\bu   d\bx  \\
	&\hspace{-0.4in}- (\Delta t)\left( P^n_{{in}}\int_{0}^1u_z\Big|_{z=0}dr-P^n_{{out}}\int_{0}^1u_z\Big|_{z=1}dr\right)- (G(\bU^{n},\eta_*^n)\Delta_n W, \bU).
	\end{align*}
	{We now estimate the right hand-side. For this purpose, note that 
		the definition of $\eta^n_*$ implies the following crucial property} $$\frac12(R+\eta_*^{n}+R+\eta_*^{n+1}) \geq {\delta}.$$ 
	{Additionally, using Korn's Lemma
		we obtain for some constant $K=K(\sO_{\eta^n_*})>0$ that}
	\begin{align*}
	\|\nabla (\bu \circ (A^\omega_{\eta^n_*})^{-1})\|^2_{\bL^2(\sO_{\eta^n_*})} \leq K\|\bD(\bu \circ (A^\omega_{\eta^n_*})^{-1})\|^2_{\bL^2(\sO_{\eta^n_*})},
	\end{align*}
	and thus
	\begin{align*}
	\int_\sO(R+\eta^n_*)|\nabla^{\eta^n_*} \bu|^2d\bx\leq K\int_\sO(R+\eta^n_*)|\bD^{\eta^n_*}(\bu)|^2d\bx.
	\end{align*}
	Additionally observe that $\eta_*^n \in H^2(0,L)$ implies that $\|\nabla A^\omega_{\eta^n_*}\|_{\bL^\infty(\sO)}<C(\delta)$. Thus
	using the formula
	$\nabla \bu=(\nabla^{\eta^n_*}\bu) (\nabla A^\omega_{\eta^n_*})$
	we obtain for some constant, still denoted by $K=K(\delta,\eta^n_*)>0$ that
	$$\|\nabla \bu\|^2_{\bL^2(\sO)} \leq K\int_\sO (R+\eta_*^n)\|\bD^{\eta_*^{n}}(\bu)\|^2d\bx.$$
	Here the constants may depend on $n$.
	
By combining these estimates we obtain
	\begin{align*}
	\langle \sL^\omega_n(\bu,v),&(\bu,v)\rangle \geq {\delta} \|\bu\|^2_{\bL^2(\sO)} 
	+ \|v\|^2_{L^2(0,L)} +\frac1{K}\|\nabla\bu\|^2_{\bL^2(\sO)} \\
	&-\|v^{n+\frac12}\|_{L^2(0,L)}\|v\|_{L^2(0,L)}-{C}\|\bu^{n+\frac12}\|_{\bL^2(\sO)}\|\bu\|_{\bL^2(\sO)}\\
	& -|P^n_{in/out}|\|\bu\|_{\bH^1(\sO)}-C
	\|G(\bU^{n},\eta^n_*)\|_{L_2(U_0,\bL^2)}\|\Delta_n W\|_{U_0}\|\bU\|_{\bL^2(\sO)\times L^2(0,L)}.
	\end{align*}
	{By using Young's inequality and \eqref{growthG} this expression can be bounded from below as follows:}
	\begin{align*}
	&\geq \|\bu\|^2_{\bH^1(\sO)}+
	\|v\|^2_{L^2(0,L)}\\
	&-\frac{C_1}{\delta^2}\left( \|v^{n+\frac12}\|^2_{L^2(0,L)}+\|\bu^{n+\frac12}\|^2_{\bL^2(\sO)} +|P^n_{in/out}|^2+
	(\|\bu^{n}\|^2_{\bL^2(\sO)}+\|v^{n}\|^2_{L^2(0,L)})\|\Delta_n W\|^2_{U_0}+1\right),
	\end{align*}
	where $C_1$ is a positive constant.
	
	{To apply the Brouwer's fixed point theorem (Corollary 6.1.1 in \cite{GR86}), we consider a ball of radius $\rho$ in $\sU$ and consider}
	$(\bu,v)\in \sU$ such that $\|(\bu,v)\|_{\sU}={\sqrt{\|\bu\|^2_{{\bH}^{1}(\sO)} +\|v\|^2_{L^2(0,L)}} }= \rho$. {Let $\rho$ be such that} 
	\begin{align*}
	\rho^2& >\frac{C_1}{\delta^2}\left( \|v^{n+\frac12}\|^2_{L^2(0,L)}+\|\bu^{n+\frac12}\|^2_{\bL^2(\sO)} +|P^n_{in/out}|^2+
	(\|\bu^{n}\|^2_{\bL^2(\sO)}+\|v^{n}\|^2_{L^2(0,L)})\|\Delta_n W\|^2_{U_0}+1\right).
	\end{align*}
	Then we see that
	\begin{align}\label{Lpositive}
	\langle \sL^\omega_n(\bu,v),(\bu,v)\rangle \geq 0.
	\end{align}
	Thanks to the separability of $\sU$, we can consider $\{\bw_j\}_{j=1}^\infty$ an independent set of vectors whose linear span is dense in $\sU$. 
	Thus using Brouwer's fixed point theorem with \eqref{Lpositive}, we obtain the existence of $\bU_k \in$ span$\{\bw_1,..,\bw_k\}$ such that $|\bU_k|\leq \rho$ and $\sL^\omega_n(\bU_k)=0$.
	
	Then the existence of solutions $(\bu,v)$ to \eqref{second} for any $\omega\in \Omega$ and $n\in \mathbb{N}$ is standard and can be obtained by 
	 passing $k\rightarrow\infty$.
	
	{What is left to show is that among the solutions constructed above, there is a $\sF_{t^{n+1}}$-measurable random variable which solves \eqref{second}.}
	{To prove the existence of a measurable solution we will use the Kuratowski and Ryll-Nardzewski selector theorem}  (see e.g. page 52 of \cite{P72} or \cite{OPW22}). {For this purpose,} for a fixed $n$, consider $\kappa$ defined by  
	$$\kappa(\omega,\bu,v)=\kappa(\bu,v;\bu^{n+\frac12}(\omega),v^{n+\frac12}(\omega),\eta^{n}(\omega),\eta^{n+\frac12}(\omega),\Delta_nW(\omega))=\sL^\omega_n(\bu,v).$$
	Using standard techniques we notice the following properties for $\kappa$:
	\begin{enumerate}
		\item {The mapping $\omega\mapsto \kappa(\omega,\bu,v)$ is $(\sF_{t^{n+1}},\sB(\sU'))-$ measurable for any $(\bu,v)\in \sU$,}
		where $\sB(\sU')$ is the Borel $\sigma$-algebra of $\sU'$, the dual space of $\sU$.
		{This is a consequence of  the assumption \eqref{growthG} on $G$.}
		\item {The mapping $(\bu,v)\mapsto \kappa(\omega,\bu,v)$} is continuous.
	\end{enumerate}
	Next, for $\omega\in \Omega$, we consider $$F(\omega)=\{(\bu,v)\in \sU:\|\kappa(\omega,\bu,v)\|_{\sU'}=0\}.$$ 
	We have already seen that $F(\omega)$ is non-empty and one can verify that it is also closed in $\sU$. Since $\sU$ is a separable metric space, every open set $H$ in $\sU$ can be written as countable union of closed balls $K^j$ in $\sU$. Hence,
	\begin{align*}
	\{\omega\in\Omega:F(\omega)\cap {H}\neq\emptyset\}&= \cup_{j=1}^\infty\{\omega:F(\omega)\cap K^j\neq \emptyset\}\\
	&=\cup_{j=1}^\infty\{\omega:\inf_{(\bu,v)\in K^j}\left( \|\kappa(\omega,\bu,v)\|_{\sU'}\right) =0\}\in \sF_{t^{n+1}}.
	\end{align*}
	The second equality above follows as the infimum is attained because of continuity of $\kappa$ in $(\bu,v)$, and {from that fact that} each set in the countable union above belongs to $\sF_{t^{n+1}}$ because of measurability property (1) of $\kappa$.
	Hence, using results from \cite{P72} 
	we obtain the existence of a selection of $F$, given by $f$ such that $f(\omega)\in F(\omega)$ and such that $f$ is $(\sF_{t^{n+1}},\sB(\sU))$-measurable, 
	where $\sB(\sU)$ denotes the Borel $\sigma$-algebra of $\sU$.
	That is, we obtain the existence of $\sF_{t^{n+1}}$-measurable function $(\bu^{n+1},v^{n+1})$, defined to be equal to $f(\omega)$, taking values in $\sU$ (endowed with Borel $\sigma$-algebras) that solves \eqref{second}. {This completes the first part of the proof.
		
		Next we show that the solution satisfies energy estimate \eqref{energy2}. For this purpose} we will derive a pathwise inequality involving the discrete energies. We take $(\bq,\psi)=(\bu^{n+1},v^{n+1})$ in \eqref{second} and using the identity $a(a-b)=\frac12(|a|^2-|b|^2+|a-b|^2)$, we obtain
	\begin{align*}
	&\frac12\int_\sO   (R+\eta_*^n) \left(|\bu^{n+1}|^2-|\bu^{n+\frac12}|^2 + |\bu^{n+1}-\bu^{n+\frac12}|^2 \right) 
	+\frac{1}2\int_\sO \left( \eta_*^{n+1}- \eta_*^n\right)  |\bu^{n+1}|^2 d\bx\\
	&+2\nu(\Delta t)\int_{\sO_{}} (R+{\eta_*^{n}})|\bD^{\eta_*^{n}}(\bu^{n+1})|^2 d\bx + \frac{(\Delta t)}\ep\int_{\sO_{}} |\text{div}^{\eta^n_*}\bu^{n+1}|^2 d\bx\\
	&+\frac12\int_0^L|v^{n+1}|^2-|v^{n+\frac12}|^2+|v^{n+1}-v^{n+\frac12}|^2 dz \\
	&= (\Delta t)\left( P^n_{{in}}\int_{0}^1u^{n+1}_z\Big|_{z=0}dr-P^n_{{out}}\int_{0}^1u^{n+1}_z\Big|_{z=1}dr\right) \\
	&+( G( \bU^{n},\eta_*^n)\Delta_n W, (\bU^{n+1}-\bU^{n} ))+
	(G(\bU^{n},\eta_*^n)\Delta_n W, \bU^{n}).
	\end{align*}
	Observe that the discrete stochastic integral is divided into two terms. {We estimate the first term by using the Cauchy-Schwarz inequality to obtain that for some} $C(\delta)>0$ independent of $n$ {the following holds:} 
	\begin{align*}&| 
	(G(\bU^{n},\eta_*^n)\Delta_n W, (\bU^{n+1}-\bU^{n} ))| \\
	&\leq {C}\|\Delta_nW\|_{U_0}^2 \|G(\bU^{n},\eta_*^n)\|_{L_2(U_0,\bL^2)}^2  +\frac14\int_{\sO} ({R+\eta_*^n})
	\left( \bu^{n+1}-\bu^{n}\right)^2d\bx+\frac18\int_0^L(v^{n+1}-v^{n})^2dz\\
	&\leq {C}\|\Delta_nW\|_{U_0}^2 \|G(\bU^{n},\eta_*^n)\|_{L_2(U_0;\bL^2)}^2 +\frac14\int_{\sO} ({R+\eta_*^n})
	\left( \bu^{n+1}-\bu^{n}\right)^2d\bx\\
	&+ \frac14\int_0^L(v^{n+1}-v^{n+\frac12})^2dz+\frac14\int_0^L(v^{n+\frac12}-v^{n})^2dz.
	\end{align*}
	We treat the terms with $P_{in/out}$ {similarly to obtain \eqref{energy2}}. {This completes the proof of Lemma \ref{existu}.}
\end{proof}
{{
		The following theorem provides uniform estimates on the expectation of the kinetic and elastic energy
		for the full, semidiscrete coupled problem (uniform in the number of time steps $N$ and in $\ep$).
		Uniform estimates on the expectation of the numerical dissipation are derived as well.
}}
\begin{theorem}[Uniform Estimates]\label{energythm}
	For any $N>0$, $\Delta t=\frac{T}{N}$, there exists a constant $C>0$ that depends on the initial data, $\delta$, $T$, and $P_{in/out}$ and is independent of $N$ and $\ep$ such that
	\begin{enumerate}
		\item $\bE\left( \max_{1\leq n \leq N}E^{n}\right)  <C$, $ \bE\left( \max_{0\leq n\leq N-1}E^{n+\frac12}\right) <C.$
		\item $\bE\sum_{n=0}^{N-1} D^n <C$.
		\item $\bE\sum_{n=0}^{N-1}\int_\sO\left((R+\eta_*^n) |\bu^{n+1}-\bu^{n}|^2 \right) d\bx+\int_0^L|v^{n+1}-v^{n+\frac12}|^2 dz<C.$
		\item $\bE \sum_{n=0}^{N-1}\int_0^L\left( |v^{n+\frac12}-v^n|^2  +|\partial_z\eta^{n+1}-\partial_z\eta^{n}|^2 + |\partial_{zz}\eta^{n+1}-\partial_{zz}\eta^{n}|^2 \right) dz<C,$
	\end{enumerate}
	{where $E^n$ and $D^n$} are defined in \eqref{EnDn}.
\end{theorem}
\begin{proof}{We add the energy estimates for the two subproblems \eqref{energy1} and \eqref{energy2} to obtain}:
	\begin{equation}
	\begin{split}\label{discreteenergy}
	E^{n+1}+D^{n} &+C^n_1+C^{n}_2\leq E^{n}+ C\Delta t((P^n_{in})^2+(P^n_{out})^2) \\
	&+C\|\Delta_nW\|_{U_0}^2
	\|G(\bU^{n},\eta_*^n)\|_{L_2(U_0;\bL^2)}^2+\left|
	(G(\bU^{n},\eta_*^n)\Delta_n W, \bU^{n}) \right|.
	\end{split}
	\end{equation}
	Then for any $m\geq 1$, summing $0\leq n\leq m-1$ gives us
	\begin{equation}\label{energysum}
	\begin{split}
	&E^m+\sum_{n=0}^{m-1}D^{n} +\sum_{n=0}^{m-1}C_1^{n}+\sum_{n=0}^{m-1}C_2^{n}
	\leq E^0+ C\,\Delta t\sum_{n=0}^{m-1}\left( (P^n_{{in}})^2 +(P^n_{{out}})^2\right) \\
	&+\sum_{n=0}^{m-1}\left|
	(G(\bU^n,\eta_*^n)\Delta_n W, \bU^{n} )\right|
	+\sum_{n=0}^{m-1} \|G(\bU^{n},\eta_*^n)\|_{L_2(U_0;\bL^2)}^2\|\Delta_nW\|_{U_0}^2.
	\end{split}\end{equation}
	
	Next we take supremum over $1 \leq m \leq N$ and then take expectation on both sides of \eqref{energysum}. We begin by treating the right hand side terms. First we have
	$$\Delta t\sum_{n=0}^{N-1} (P^n_{{in/out}})^2\leq\|P_{in/out}\|^2_{L^2(0,T)} .$$
	Next we apply the discrete Burkholder-Davis-Gundy inequality (see e.g. Theorem 2.7 \cite{OPW22}) 
	and write $\bu^{n+\frac12}=\bu^n$ to obtain for some $C(\delta)>0$ that
	\begin{align*}
	\bE&\max_{1\leq m\leq N}|\sum_{n=0}^{m-1}
	(G(\bU^{n},\eta_*^n)\Delta_n W, \bU^{n} )| \\
	&\leq {C}\bE\left[\Delta t\sum_{n=0}^{N-1}
	\|G(\bU^{n},\eta_*^n)\|^2_{L_2(U_0,\bL^2)}\left( \left\|(\sqrt{R+\eta_*^n})\bu^{n}\right\|^2_{\bL^2(\sO)}+\|v^n\|_{L^2(0,L)}^2\right) \right]^{\frac12}\\
	& \leq {C}{}\bE\left[ \left( \max_{0\leq m\leq N}\left\|(\sqrt{R+\eta_*^m})\bu^{m}\right\|^2_{\bL^2(\sO)}+\|v^m\|^2_{L^2(0,L)}\right) \sum_{n=0}^{N-1}
	\Delta t \left( \|\sqrt{(R+\eta^n_*)}\bu^{n}\|^2_{\bL^2(\sO)}+\|v^n\|_{L^2(0,L)}^2\right) \right]^{\frac12}\\
	&\leq \frac{1}2\|(\sqrt{R+\eta_0})\bu_0\|^2_{\bL^2(\sO_{})} +\frac12\|v_0\|^2_{L^2(0,L)}+ \frac12\bE\max_{1\leq m\leq N}\left[ \left\|(\sqrt{R+\eta_*^m})\bu^{m}\right\|^2_{\bL^2(\sO)}+\|v^m\|^2_{L^2(0,L)}\right] \\
	&+ {C \Delta t}\bE\left( \sum_{n=0}^{N-1}
	\|(\sqrt{R+\eta^n_*})\bu^n\|^2_{\bL^2(\sO)}+\|v^n\|_{L^2(0,L)}^2\right).
	\end{align*}

	Also using the tower property and \eqref{growthG} for each $n=0,...,m-1$
	we write
	\begin{align}
	\bE[\|G(\bU^n,\eta_*^n)\|_{L_2(U_0,\bL^2)}^2\|\Delta_nW\|_{U_0}^2]&=\bE[\bE[\|G(\bU^n,\eta_*^n)\|_{L_2(U_0,\bL^2)}^2\|\Delta_nW\|_{U_0}^2|\sF_n]]\notag\\
	&=\bE[\|G(\bU^n,\eta_*^n)\|_{L_2(U_0,\bL^2)}^2\bE[\|\Delta_nW\|_{U_0}^2\|\sF_n]]\notag\\
	&=\Delta t(\Tr \bQ)\bE\|G(\bU^n,\eta_*^n)\|^2_{L_2(U_0,\bL^2)}\notag\\
	&\leq C\Delta t(\Tr \bQ)\bE[\|\sqrt{(R+\eta_*^n)}\bu^n\|^2_{\bL^2(\sO)}+\|v^n\|^2_{L^2(0,L)}]\label{tower}.
	\end{align}
	Thus we obtain for some $C>0$ depending on {$\delta$ and on $\Tr{\bQ}$, that the following holds:}
	\begin{equation}
	\begin{split}\label{gronwall}
	\bE\max_{1\leq n\leq N}E^{n}&+ \bE\sum_{n=0}^{N-1}D^n +\bE\sum_{n=0}^{N-1}C_1^{n}+\bE\sum_{n=0}^{N-1}C_2^{n}\leq CE^0+C\|P_{in/out}\|^2_{L^2(0,T)}
	\\&+ C\Delta t \bE\left[\sum_{n=0}^{N-1} \|(\sqrt{R+\eta^n_*})\bu^n\|^2_{\bL^2(\sO)}+\|v^n\|^2_{L^2(0,L)}\right]\\
	&+ \frac12\bE\max_{1\leq n\leq N}\left[ \left\|(\sqrt{R+\eta_*^n})\bu^{n}\right\|^2_{\bL^2(\sO)}+\|v^n\|^2_{L^2(0,L)}\right].
	\end{split}
	\end{equation}
	By absorbing the last term on the right hand-side of \eqref{gronwall} we obtain
	\begin{align*}
	\bE\max_{1\leq n\leq N}&\left( \|(\sqrt{R+\eta^n_*})\bu^n\|^2_{\bL^2(\sO)} +\|v^n\|^2_{L^2(0,L)}\right) \leq CE^0+ C\|P_{in/out}\|^2_{L^2(0,T)}\\
	& +\sum_{n=1}^{N-1} \Delta t\bE \max_{1\leq m\leq n}\left(\|(\sqrt{R+\eta^m_*})\bu^m\|^2_{\bL^2(\sO)}+\|v^m\|^2_{L^2(0,L)}\right) .
	\end{align*}
	Applying the discrete Gronwall inequality to $\bE\max_{1\leq n\leq N}(\|(\sqrt{R+\eta^n_*})\bu^n\|_{\bL^2(\sO)}^2+\|v^n\|^2_{L^2(0,L)})$ we obtain
	$$\bE\max_{1\leq n \leq N}\left( \|(\sqrt{R+\eta^n_*})\bu^n\|_{\bL^2(\sO)}^2+\|v^n\|^2_{L^2(0,L)}\right) \leq C e^T,$$
	where $C$ depends only on the given data and in particular $\delta$.
	
	Hence for $E^n,D^n$ defined in \eqref{EnDn} we have obtained the desired energy estimate:
	\begin{equation}
	\begin{split}
	\bE\max_{1\leq n\leq N}E^{n}+ \bE\sum_{n=0}^{N-1}D^n&+\bE\sum_{n=0}^{N-1}C_1^{n}+\bE\sum_{n=0}^{N-1}C_2^{n} \\
	&\leq C(E^0  + \|P_{in}\|^2_{L^2(0,T)}+\|P_{out}\|^2_{L^2(0,T)}
	+ Te^T).
	\end{split}
	\end{equation}
\end{proof}

\subsection{Approximate solutions}\label{subsec:approxsol}
In this subsection, we use the solutions $(\bu^{n+\frac{i}2},\eta^{n+\frac{i}2},v^{n+\frac{i}2})$, $i=0,1$, defined for every $N \in \mathbb{N}\setminus \{0\}$ at discrete times to {define} approximate solutions {on the entire interval $(0,T)$. We start by introducing approximate solutions} that are piecewise constant on each subinterval $ [n\Delta t, (n+1)\Delta t)$ as 
\begin{align}\label{AppSol}
\bu_{N}(t,\cdot)=\bu^n, \,
\quad \eta_{N}(t,\cdot)=\eta^n,\quad \eta_{N}^*(t,\cdot)=\eta_*^{n}, \quad v_{N}(t,\cdot)=v^n,\quad  {v}^{\#}_{N}(t,\cdot)=v^{n+\frac12}.
\end{align}
Observe that all of the processes defined above are adapted to the given filtration $(\sF_t)_{t \geq 0}$.\\
We also need to define the following time-shifted piecewise constant functions on $t \in (n\Delta t, (n+1)\Delta t]$, which will be useful in obtaining tightness results using compactness methods:
\begin{align*}
\bu^+_{N}(t,\cdot)=\bu^{n+1},\quad \eta^+_{N}(t,\cdot)=\eta^{n+1},\quad v^+_{N}(t,\cdot)=v^{n+1}.
\end{align*}
Furthermore we define the corresponding piecewise linear interpolations of the structure displacement, velocity and the fluid velocity: for $t \in [t^n,t^{n+1}]$ we let  
\begin{equation}
\begin{split}\label{approxlinear}
&\tilde\bu_{N}( t,\cdot)=\frac{t-t^n}{\Delta t} \bu^{n+1}+ \frac{t^{n+1}-t}{\Delta t} \bu^{n}, \quad \tilde v_{N}(t,\cdot)=\frac{t-t^n}{\Delta t} v^{n+1}+ \frac{t^{n+1}-t}{\Delta t} v^{n}\\
& \tilde\eta_{N}( t,\cdot)=\frac{t-t^n}{\Delta t} \eta^{n+1}+ \frac{t^{n+1}-t}{\Delta t} \eta^{n}, \quad \tilde\eta^*_{N}( t,\cdot)=\frac{t-t^n}{\Delta t} \eta^{n+1}_*+ \frac{t^{n+1}-t}{\Delta t} \eta^{n}_*.
\end{split}
\end{equation}
Observe that% with a little abuse of notation,
\begin{align}\label{etaderi}
\frac{\partial\tilde\eta_{N}}{\partial t}=v^{\#}_{N},\quad \frac{\partial\tilde\eta^*_{N}}{\partial t}
={\sum_{n=0}^{N-1}\theta_\delta(\eta^{n+1})}v^{\#}_{N}\chi_{(t^n,t^{n+1})}:=v_{N}^*
\quad a.e. \text{ on } (0,T),  
\end{align}{
	where $v^{\#}_N$ was introduced in \eqref{AppSol}.}
\begin{lem}\label{bounds}
	Given
	${ \bu_0} \in \bL^2(\sO_{})$, $\eta_0 \in  H^2_0(0,L)$, $v_0 \in L^2(0,L)$, 
	for a fixed $\delta>0$, we have that
	\begin{enumerate}
		\item $\{\eta_{N}\},\{\eta_{N}^*\}$ and thus $\{\tilde \eta_{N}\},\{\tilde\eta_{N}^*\}$ are bounded independently of $N$ and $\ep$ in\\ $L^2(\Omega;L^\infty(0,T;H^2_0(0,L)))$.
		\item $\{v_{N}\},\{v_{N}^{\#}\},\{v_{N}^{*}\},\{v_{N}^+\}$ are bounded independently of $N$ and $\ep$ in\\ $L^2(\Omega;L^\infty(0,T;L^2(0,L)))$.
		\item $\{\bu_{N}\}$ is bounded independently of $N$ and $\ep$ in $L^2(\Omega;L^\infty(0,T;\bL^2(\sO)))$.
		\item $\{\bu^+_{N}\}$ is bounded independently of $N$ and $\ep$ in 
		$		 L^2(\Omega;L^2(0,T;V) \cap L^\infty(0,T;\bL^2(\sO))).$
	{\cred		\item $\{v^+_{N}\}$ is bounded independently of $N$ and $\ep$ in 
		$		 L^2(\Omega;L^2(0,T;H^{\frac12}(0,L)) ).$}
	%	\item $\{v^+_{N}\}$ and $\{v^*_{N}\}$ are bounded independently of $N$ and $\ep$ in 		$L^2(\Omega;L^2(0,T;L^{q}(0,L)) )$ for any $q<\infty$.
		\item $\{\frac1{\sqrt{\ep}}\text{div}^{\eta^*_{N}}\bu^+_{N}\}$ is bounded independently of $N$ and $\ep$ in $L^2(\Omega;L^2(0,T;L^{2}(\sO)))$.
	\end{enumerate} 
\end{lem}
\begin{proof}
	The proofs of the first three statements and that of (6) follow immediately from Theorem \ref{energythm}. 
	For the second statement, we write $v^*_N(t,\omega,z)=\theta_N(t,\omega)v^{\#}_N(t,\omega,z)$ where we set
	$$\theta_N(t,\omega)=\sum_{n=0}^{N-1}\theta_\delta(\eta^{n+1}(\omega))\mathbbm{1}_{[t^n,t^{n+1})}(t).$$
	Since $\theta_N(t) \leq 1$ for any $t\in[0,T]$ we obtain,
	\begin{align*}
	\bE[\| v^*_N\|^2_{L^\infty(0,T;L^2(0,L))}] \leq \bE[\| v^\#_N\|^2_{L^\infty(0,T;L^2(0,L))}] \leq C.
	\end{align*}
	In order to prove (4) observe that for each $\omega\in \Omega$, $\nabla\bu^{n+1}=\nabla^{\eta^{n}_*}\bu^{n+1} (\nabla A^\omega_{\eta^{n}_*})$. Thus we have, 
	\begin{align*}
	\delta \bE\int_{\sO}& |\nabla\bu^{n+1}|^2d\bx \leq  \bE\int_{\sO}(R+\eta_*^{n})|\nabla\bu^{n+1}|^2d\bx
	=\bE\int_{\sO}(R+\eta^{n}_*)|\nabla^{\eta_*^{n}}\bu^{n+1}\cdot \nabla A^\omega_{\eta_*^{n}}|^2d\bx\\
	&\leq C(\delta)\bE\int_{\sO}(R+\eta^{n}_*)|\nabla^{\eta_*^{n}}\bu^{n+1}|^2d\bx
	 \leq {K}C(\delta)\bE\int_{\sO}(R+\eta_*^{n})|\bD^{\eta_*^{n}}\bu^{n+1}|^2d\bx,
	\end{align*}
	where $K>0$ is the universal Korn constant that depends only on the reference domain $\sO$. This result follows from Lemma 1 in \cite{V12} {because of} the uniform bound $\|R+\eta^n_*(\omega)\|_{H^s(0,L)} <\frac1{\delta}$, for $\frac32<s<2$ and every $\omega \in \Omega$, which implies that for some $C(\delta)>0$ 
	$$\|A^\omega_{\eta^n_*}\|_{\bW^{1,\infty}}<C,\quad \|(A^\omega_{\eta^n_*})^{-1}\|_{\bW^{1,\infty}}<C, \quad n=1,...,N, N\in\mathbb{N},$$ for every $\omega\in \Omega$. 
	Thus there exists $C>0$, independent of $N$ and $ \ep$, such that
	\begin{align}\label{uboundV}
	\bE\int_0^T\int_{\sO}|\nabla\bu^{+}_{N}|^2d\bx ds =\bE\sum_{n=0}^{N-1} \Delta t\int_{\sO}|\nabla\bu^{n+1}|^2d\bx \leq C(\delta).
	\end{align}
	Statement (5) is then a result of statement (4) and the fact that, by construction, $v_N^+$ is the trace of the vertical component of the time-shifted fluid velocity $\bu^+_{N}$ on the top lateral boundary
\end{proof}
\section{Passing to the limit as $N\rightarrow\infty$}\label{sec:limit}
{Our goal is to construct solutions to the coupled FSI problem as limits of subsequences of approximate solutions as $N\rightarrow\infty$ and $\ep \to 0$. 
	The $\delta$ approximation of the problem via cutoff functions will be dealt with using a stopping time argument at the very end of the manuscript.
	Our approach is to first let $N\to\infty$, obtain approximate solutions for each $\ep > 0$, and then let $\ep\to 0$. 
	
	Thus, we first fix $\ep>0$ and $\delta>0$ and let $N\to\infty$. For this purpose we recall}
the bounds obtained in Lemma \ref{bounds} that provide weakly and weakly-* convergent subsequences. In order to upgrade these convergence results to obtain almost sure convergence of the stochastic approximate solutions, which is required to be able to pass $N \rightarrow \infty$, we use compactness arguments and establish tightness of the laws of the approximate random variables defined in Section \ref{subsec:approxsol}.
\subsection{{ Tightness results}} 
 Given our stochastic setting we do not expect the fluid and structure velocities to be differentiable in time.
Hence the tightness results, i.e. Lemmas \ref{tightuv} and \ref{tightl2} below, will rely on an application of the Aubin-Lions theorem for the structure displacements and its following variant for the velocities \cite{S87} :
\begin{lem}\cite{S87} \label{compactLp}
	Let	 the translation in time by $h$ of a function $f$ be denoted by:
	$$T_h f(t,\cdot)=f(t-h,\cdot), \quad h\in \R.$$ Assume that $\mathcal{Y}_0\subset\mathcal{Y}\subset\mathcal{Y}_1$ are Banach spaces such that $\mathcal{Y}_0$ and $\mathcal{Y}_1$ are reflexive with compact embedding of $\mathcal{Y}_0$ in $\mathcal{Y}$, 
	then for any $m>0$, the embedding	$$ \{\bu \in L^2(0,T;\mathcal{Y}_0):\sup_{0<h< T} \frac1{h^m}\|T_h\bu-\bu\|_{L^2(h,T;\sY_1)}<\infty \} \hookrightarrow L^2(0,T;\mathcal{Y}),$$	is compact.
\end{lem}

%In order to obtain bounds on the fractional time derivatives of the stochastic approximate solutions, we begin by introducing

\begin{lem}\label{tightuv} {{For any $0\leq \alpha<1$ the laws of $\bu^+_N$ and $v^+_N$ are tight in $L^2(0,T;\bH^\alpha(\sO))$ and $ L^2(0,T;L^2(0,L))$, respectively. }}
\end{lem}
\begin{proof} 
 The aim of this proof is to apply Lemma \ref{compactLp} 
 by obtaining appropriate bounds for
 \begin{equation}\label{TimeShifts}
 \int_h^{T} \|T_h\bu_N-\bu_N\|^2_{\bL^2(\sO)}+\|T_hv_N-v_N\|^2_{L^2(0,L)}=	(\Delta t)\sum_{n=j}^N \|\bu^n-\bu^{n-j}\|^2_{\bL^2(\sO)} + \|v^n-v^{n-j}\|^2_{L^2(0,L)},
\end{equation}
	{{which are given in terms of powers of $h$}}, for any $h>0$, {independently of $N$. For a fixed $N$, we write $h=j\Delta t-s$ for some $1\leq j \leq N$ and  $s<\Delta t$}.
	{{More precisely, for $0\leq \alpha<1$ and any $M>0$ let us introduce the set:
	\begin{align*}
	{\mathcal{B}}_M:=&\{(\bu,v)\in L^2(0,T;\bH^\alpha(\sO))\times L^2(0,T;L^2(0,L)):
	\|{\bu}\|^2_{L^2(0,T;\bH^1(\sO))}+\|{v}\|^2_{L^2(0,T;H^{\frac12}(0,L))}
	\\
	&+\sup_{0<h<1}{h^{-\frac1{32}}}\int_h^{T}\left( \|T_h\bu-\bu\|^2_{\bL^2(\sO)}+\|T_hv-v\|^2_{L^2(0,L)}\right)  \le M\}.
\end{align*}
Notice that  Lemma \ref{compactLp}, implies that $\sB_M$ is compact in $L^2(0,T;\bH^\alpha(\sO))\times L^2(0,T;L^2(0,L))$ for each $M>0$.
Here in Lemma \ref{compactLp} we used $\mathcal{Y}_0 = \bH^1(\sO) \times H^{\frac12}(0,L)$,
$\mathcal{Y} = \bH^\alpha(\sO) \times L^2(0,L)$, and $\mathcal{Y}_1 = \bL^2(\sO) \times L^2(0,L)$.
To show tightness of laws of $\bu^+_N$ and $v^+_N$, we will then show that there exists $C>0$ independent of $N$ and $\ep$, such that
	\begin{align*}
	\bP((\bu^+_N,v^+_N) \notin \sB_M)&\leq \bP\left( \|{\bu^+_N}\|^2_{L^2(0,T;\bH^1(\sO))}+\|{v^+_N}\|^2_{L^2(0,T;H^{\frac12}(0,L))}>\frac{M}2\right) \\
	&+\bP\left( \sup_{0 <h<1 }{h^{-\frac1{32}}}\int_h^{T}\left( \|T_h\bu^+_N-\bu^+_N\|^2_{\bL^2(\sO)}+\|T_hv^+_N-v^+_N\|^2_{L^2(0,L)}\right)  > \frac{M}2\right) \\
	&\leq \frac{C}{\sqrt{M}}.
\end{align*}
}}
	
	To achieve this goal, we will construct appropriate test functions for equations \eqref{first} and \eqref{second} that will give the term on the right hand side of the equation above. This has to be done carefully since the solutions to the fluid subproblem \eqref{second} are defined on different physical domains. 
	That is, the velocity functions $\bu^k$, for any { $1 \leq k\leq N$}, cannot directly be used as test functions for the equation for $\bu^n$, { for $1 \leq n \neq k\leq N$}. 
	%{{Remind me: ${\sO}$ is the fixed domain obtained via the ALE mapping. Are we referring to the test functions which are defined on physical moving domains, which we want to map to the fixed domain ${\sO}$?}} 
	Hence, we must construct the test functions by first transforming $\bu^k$, such that appropriate boundary conditions are satisfied by these test functions and such that they have "small" divergence so that the penalty term can be bounded independently of $\ep$. This is achieved by first "squeezing" $\bu^k$ appropriately (see also \cite{CDEG}, \cite{G08}, \cite{MC19}, \cite{CanicLectureNotes}). However, using the squeezed function directly as a test function would result into testing $n^{th}$ step equation \eqref{first}$_3$ with $v^k$, which does have the required $H^2$ spatial regularity. Thus, we mollify the squeezed function to get the desired $H^2$ regularity.
	 
	Hence our plan is as follows: assume that, { for any $\delta>0$ used to define our stopped process in \eqref{eta*}}, $\sO_\delta$ 
	%{{What is delta here? Remind the reader. It corresponds to the cut-off displacements, but we denoted them with a sub-script *. Check notation for $\eta^*_N$ which is at the end of this sentence.}} 
	is the maximal rectangular domain  consisting of all the fluid domains associated with the structure displacement $\eta^*_N$, { defined in \eqref{AppSol}}, for any $N$ and $t,\omega$. We fix an $N$ and $\omega\in\Omega$, and for any $0\leq k,n \leq N$ we push $\bu^k$ onto the physical domain $\sO_{\eta^{k-1}_*}$, extend it by its trace outside of $\sO_{\eta^{k-1}_*}$ to $\sO_\delta$, then "squeeze" it in a way that {the $L^2$-bound on} its divergence is preserved, and mollify it. {The mollification is done in a way that preserves the zero boundary condition for the radial component of the velocity on the left and right boundaries 
	$\Gamma_{in/out}$, and the zero boundary value for the horizontal velocity component on the {fluid-structure interface $\Gamma_{\eta^n_*}$.}} Then we pull it back to the domain $\sO$  via the  ALE map $A^\omega_{{\eta^n_*}}$. This way we have transformed $\bu^k$ so that it can be used as a test function for the equations \eqref{second} for $\bu^n$. Details of the construction are presented next.

	First, let $ \tilde\bu_{}^k=\bu^k\circ A^{-1}_{{\eta^{k-1}_*}}$. Here we used ``tilde'' to denote the function defined on the physical domain. {For the purposes of mollification, we define $\tilde\bu^{k,ext}$ to be an extension of $\tilde\bu^k$ to the whole of $\mathbb{R}^2$, as follows:}. 
	%
	%{{ Do we want a tilde here as well?}} 	The extension is defined as follows:
	\begin{equation*}
	\tilde\bu^{k,ext} = \left\{ 
	\begin{array}{ll}
	  (\tilde u^k_z(0,r),0)\quad &{\text{for}} \ z \le 0,\\
	  (\tilde u^k_z(L,r),0)           &{\text{for}} \ z \le L,\\
	  (0,v^k)                             &{\text{above}} \ \Gamma_{\eta^{k-1}_*},\\
	   0                                     &{\text{everywhere \ else.}}
	\end{array}
	 \right.
	\end{equation*} 
	{{Then, we introduce the squeezing parameter $\sigma > 1$ and the following squeezing operator, also denoted by $\sigma$ (with a slight abuse of notation):}}
	 \begin{align*}
	 	\tilde \bu^{k}_\sigma(z,r)&:= \tilde\bu^k_{ext}\circ \sigma = (\sigma\tilde u_z^{k,ext}\left(z,\sigma r \right),\tilde u_r^{k,ext}\left(z,\sigma r \right)). %\quad 
	 \end{align*}
As mentioned earlier, notice that we scaled the $r$ coordinate i.e. squeezed it vertically, so that $\tilde\bu^{k,ext}$ assumes appropriate values on $\Gamma_{\eta^n_*}$ in order for it to be used as a test function for the $n^{th}$ equations. { To improve the spatial regularity of its trace on $\Gamma_{\eta^n_*}$}, we next introduce $\tilde\bu^k_{\sigma,\lambda}$, a space mollification of $\tilde\bu^k_{\sigma}$, where $\lambda$ denotes the space mollification parameter. The mollification is obtained using standard 2D mollifiers. 
{{Based on this space mollification $\tilde\bu^k_{\sigma,\lambda}$, for any given $N$ and  $n$ such that $1\le n \le N,$ we define 	
	 \begin{align}
	 	\tilde\bu^{k,n}_{\sigma,\lambda}=\tilde\bu^k_{\sigma,\lambda}\circ A^\omega_{\eta^n_*},
	 \end{align}
}}%	 
{{where the parameters $\lambda>0$ and $\sigma$  are chosen as follows:
\begin{enumerate}
\item  The random variable $\sigma$ is defined so that the squeezed $k^{th}$ fluid velocity function lies entirely within the $n^{th}$ fluid domain, see Fig.~\ref{fig2}:
 $$\sigma=\max_{1\leq n\leq N}\max_{n-j\leq k\leq n}\sup_{z\in[0,L]}\frac{R+\eta^k_*}{R+\eta^n_*-\delta h^{\frac14}};$$

\item The parameter $\lambda$ is then chosen to be small enough so that mollification with parameter $\lambda$ preserves the kinematic coupling condition, see Fig. \ref{fig2}:
\begin{equation}\label{lambda}
 \lambda=\frac{ \delta h^{\frac14}}{\sqrt{1+\delta^2}}.
 \end{equation}
\end{enumerate}
}}
 \begin{figure}[h]
	\includegraphics[scale=0.9]{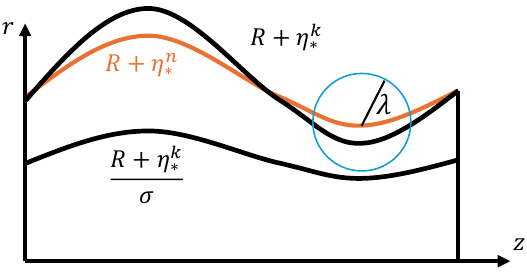}
	\caption{A realization of the vertical squeezing and the mollification operator.}\label{fig2}
\end{figure}
Due to these choices of $\sigma$ and $\lambda$ we will be able to obtain the desired estimates in terms of the powers of the translation parameter $h$,
as we shall see below.
{In particular, we observe that for this choice of $\sigma$, the vertical distance between $R+\eta^n_*$ and $\frac{R+{\eta^k_*}}\sigma$ is at least $\delta h^{\frac14}$. Moreover, thanks to the mean value theorem and the property of the { artificial structure variables $\eta^n_*$},  that
 $\|\partial_z\eta^n_*\|_{L^\infty(0,L)} \leq \frac1{\delta}$ for any $n \leq N$, the distance between these two curves is at least $\frac{ \delta h^{\frac14}}{\sqrt{1+(\delta\sigma)^{-2}}}$, which motivates our choice of $\lambda$, see Fig. \ref{fig3}. }
\begin{figure}[h]
	\includegraphics[scale=0.5]{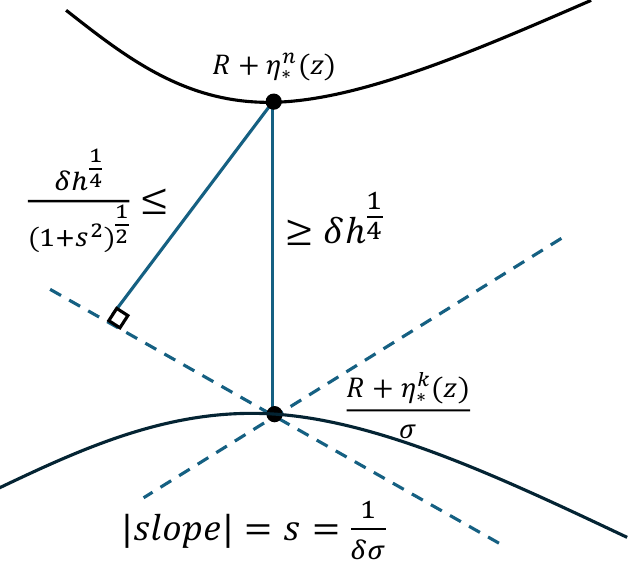}
	\caption{{{Distance between the curves $R+\eta^n_*$ and $\frac{R+{\eta^k_*}}\sigma$}}.}\label{fig3}
\end{figure}

{{{\bf{Remark.}}  Before we continue,}} { we note here that we additionally need to "squeeze" the function $\tilde\bu^{k,ext}$ horizontally so that {{when we mollify the function to obtain the desired regularity}}, the mollification does not ruin its 0 boundary value for the radial component at the boundaries $\Gamma_{in/out}$ and, for the same reasons, vertically near the bottom boundary $\Gamma_b$ as well. This can be done using the same argument, i.e. by choosing the horizontal squeezing parameter, {{say $\sigma_{horizontal}$,}} and {{the mollification parameter}} $\lambda$ appropriately small. The resulting function is in $\bH^s(\sO_\delta)$ for any $s<\frac12$ and thus the  estimates {{presented below}} can be obtained the same way. However, we choose to leave {{this part of the construction}}  out of our discussion for a cleaner presentation.}

{{We now}} define the structure equation test function $v^{k}_{\lambda}$ to be the space mollification of  $v^{k,ext}$ where $ v^{k,ext}$ is the extension of $v^k$ { to $\mathbb{R}$} by 0. %{{TODO: EXTENSION TO WHICH DOMAIN?}} 
Observe that $\partial_r v^{k}_{\lambda}=(\partial_r v^{k,ext})_{\lambda}=0$. 

Hence, we have that for any $n-j\leq k< n$, the transformed fluid velocity $\bu^{k,n}_{\sigma,\lambda}$ satisfies: 
\begin{align}
\label{kton}\|\text{div}^{\eta^n_*}\bu^{k+1,n}_{\sigma,\lambda}\|_{L^2(\sO)}\leq {C(\delta)}{\sigma}\|\text{div}^{\eta^k_*}\bu^{k+1}\|_{L^2(\sO)},\quad \text{ and} \quad \bu^{k,n}_{\sigma,\lambda}|_\Gamma=(0, v^{k}_{\lambda}).
\end{align}

{{Recall that our goal is to use the modified fluid velocity and structure displacements constructed above as test functions for our coupled problem
 \eqref{first} and \eqref{second}, based on which we will obtain the desired estimates of the right hand-side of \eqref{TimeShifts} in terms of power of $h$. For this purpose, we will need the following estimate of $\sigma$ in terms of powers of $h$: For any $j\leq n \leq N$ observe that we have
{\begin{align}
\notag\sigma-1 &= \max_{1\le n \le N}\max_{n-j\leq k\leq n}\sup_{z\in[0,L]}\frac{R+\eta^k_*}{R+\eta^n_*-\delta h^{\frac14}} -1=   \max_{1\le n \le N}\max_{n-j\leq k\leq n}\sup_{z\in[0,L]}\frac{\eta^k_*-\eta^n_*+\delta h^{\frac14}}{R+\eta^n_*-\delta h^{\frac14}} \\
&\leq  \max_{1\le n \le N}\max_{n-j\leq k\leq n} \frac{\sup_{z\in[0,L]}|\eta^n_*-\eta^k_*|+\delta h^{\frac14}}{\delta} \leq C(\frac1\delta\|\tilde\eta^*_N\|_{C^{0,\frac14}(0,T;L^\infty(0,L))}+1)h^{\frac14},\label{boundsigma}
\end{align}}
where we have used $\|T_h\eta^*_N-\eta^*_N\|_{L^\infty((0,T)\times(0,L))}\leq \|\tilde\eta^*_N\|_{C^{0,\frac14}(0,T;L^\infty(0,L))} h^\frac14$ in the last inequality. }}

%and hence, \begin{align} \frac{h^{\frac14}}{1+\frac{C}\delta\|\tilde\eta^*_N\|_{C^{0,\frac14}(0,T;L^\infty(0,L))}h^{\frac14}}\lesssim \lambda \lesssim h^{\frac14} .\label{boundlambda}\end{align}To summarize, we have chosen $\sigma-1 \sim  \|\tilde\eta^*_N\|_{C^{0,\frac14}(0,T;L^\infty(0,L))}h^{\frac14}$ and $\lambda \sim h^{\frac14}$.

{{{{In the estimates below we will also need the estimates on the}}  difference between the transformed and the original fluid {{velocity}} functions. For this purpose, we let $\rho$ be the 2D bump function.}}
Then, for any $k \leq N$, $s<\frac12$ we find:
\begin{align}
\notag	\|\tilde\bu^k_{\sigma,\lambda}&-\tilde\bu^{k,ext}\|^2_{\bL^2(\sO_\delta)}=\int_{\sO_\delta}\Big|\int_{B(0,1)}\rho(y)(\tilde\bu^k_\sigma(x+\lambda y)-\tilde\bu^{k,ext}(x)) dy\Big|^2 dx\\
\notag	&=  \int_{\sO_\delta}\Big|\int_{B(0,1)}\rho(y)(\tilde\bu^{k,ext}(\sigma(x+\lambda y))-\tilde\bu^{k,ext}(x)) dy\Big|^2 dx\\
\notag	&\leq  \int_{\sO_\delta}\int_{B(0,1)}\frac{|(\tilde\bu^{k,ext}(\sigma(x+\lambda y))-\tilde\bu^{k,ext}(x))|^{2+2s}}{|\sigma(x+\lambda y)-x|^2} {|\sigma(x+\lambda y)-x|^{2+2s}} dy dx\\
\notag	&\leq C(s,\delta,L)((\sigma-1)^{2+2s}+(\sigma\lambda)^{2+2s}) \int_{\sO_\delta}\int_{B(0,1)}\frac{|(\tilde\bu^{k,ext}(\sigma(x+\lambda y))-\tilde\bu^{k,ext}(x))|^2}{|\sigma(x+\lambda y)-x|^{2+2s}} dy dx.
\end{align}
{{After the change of variables $w=\sigma(x+\lambda y)$ we obtain:}}
\begin{align}
\notag\|\tilde\bu^k_{\sigma,\lambda}-\tilde\bu^{k,ext}\|^2_{\bL^2(\sO_\delta)} &\leq C(s,\delta,L)\left( (\sigma-1)^{2+2s}\sigma^{2+2s}\right) \sigma^{-1}\lambda^{-2} \int_{\sO_\delta}\int_{\sO_\delta^{\sigma,\lambda}}\frac{|(\tilde\bu^{k,ext}(w)-\tilde\bu^{k,ext}(x))|^2}{|w-x|^{2+2s}} dw dx\\
&	\leq C(s,\delta,L)\left( \frac{(\sigma-1)^{2+2s}}{\lambda^{2}}\sigma^{1+2s}\right) \|\tilde\bu^{k,ext}\|^2_{\bH^s(\sO^{\sigma,\lambda}_\delta)}\label{u1},
%&	\leq C(s,\delta,L) (\sigma-1)^{2s}\sigma^{2s+2} \|\tilde\bu^{k,ext}\|^2_{\bH^s(\sO^{\sigma,\lambda}_\delta)},
\end{align}
where $\sO_\delta^{\sigma,\lambda}=\sigma(\sO_\delta)+\lambda B(0,1)$. 
{
 For $s<\frac12$, we can estimate the $H^s$ norm of $\tilde\bu^{k,ext}$ on $\sO_\delta^{\sigma,\lambda}$ as follows:}
$$
\|\tilde\bu^{k,ext}\|^2_{\bH^s(\sO_\delta^{\sigma,\lambda})} %\leq C\|\bu^{k,ext}\|^2_{\bH^s(\sO_\delta)}
\leq C( \|\bu^k\|^2_{\bH^1(\sO)} +\|v^k\|^2_{H^\frac12(0,L)}),
$$ 
where $C$ depends only on $\delta$ (see \cite{Mi11}).
{ Thus, for any $n-j\leq k\leq n$ and $s<\frac12$, our choice of $\sigma$ that satisfies \eqref{boundsigma} gives us,}
\begin{align}
\notag	&\|\bu^{k,n}_{\sigma,\lambda}-\bu^{k}\|_{\bL^2(\sO)} \leq \|\tilde\bu^{k}_{\sigma,\lambda}-\tilde\bu^{k,ext}\|_{\bL^2(\sO_{\eta^n_*})}+\|\tilde\bu^{k,ext}\circ(A_{\eta^n_*}-A_{\eta^{k-1}_*})\|_{\bL^2(\sO)}\\
\notag	&\qquad \leq  \left( \frac{(\sigma-1)^{s+1}}{\lambda}\sigma^{s+\frac12} +\lambda^s\right) \|\bu^{k}\|_{\bH^1(\sO)}
	+\|\partial_r\tilde\bu^{k,ext}\|_{\bL^2(\sO_\delta)}\|\eta^n_*-\eta^{k-1}_*\|_{L^\infty(0,L)}\\
	&\qquad \leq \left( \frac{(\sigma-1)^{s+1}}{\lambda}\sigma^{s+\frac12}+\lambda^s\right)  \|\bu^{k}\|_{\bH^1(\sO)}
	+h^{\frac14}\|\tilde\eta^*_N\|_{C^{0,\frac14}(0,T;L^\infty(0,L))}\|\bu^{k}\|_{\bH^1(\sO)}\notag\\
	& \leq h^{\frac{s}4}\left( \max\{1,\frac1{\delta}\|\tilde\eta^*_N\|^{2s+\frac32}_{C^{0,\frac14}(0,T;L^\infty(0,L))}\}\right)  \|\bu^{k}\|_{\bH^1(\sO)}
	+h^{\frac14}\|\tilde\eta^*_N\|_{C^{0,\frac14}(0,T;L^\infty(0,L))}\|\bu^{k}\|_{\bH^1(\sO)}. \label{u2}
\end{align}
{{This estimate of the difference $\|\bu^{k,n}_{\sigma,\lambda}-\bu^{k}\|_{\bL^2(\sO)}$ given in terms of powers of $h$ will be used below in \eqref{I2} to estimate 
one of the integrals contributing to the final estimate of the time-shifts of $\bu_N$.}}

{ Likewise, the following property of mollification along with \eqref{lambda}, will be used in \eqref{I3}:}
\begin{align}\label{v1}
	\|v^{k}_{\lambda}-v^k\|_{L^2(0,L)} \leq \lambda^{s} \| v^{k}\|_{H^s(0,L)} \leq h^{\frac{s}{4}} \| v^{k}\|_{H^s(0,L)}
	,\quad \text{ for any } s\leq\frac12.
\end{align}

{{We are now ready to multiply \eqref{first}$_3$ and \eqref{second}$_2$ by the test functions $\bq_n$ and $\psi_n$, respectively, (see also \cite{CDEG}),
where $\bq_n$ and $\psi_n$ are defined by
$$
\bQ_n:=(\bq_n,\psi_n)=\left( (\Delta t) \sum_{k=n-j+1}^n\bu^{k,n}_{\sigma,\lambda},\,\, (\Delta t) \sum_{k=n-j+1}^nv^{k}_{\lambda}\right). 
$$
After adding the resulting equations we obtain:
}}
		 \begin{align}\label{MainWeakForm}	 	
	 	&\int_{\sO}\left((R+\eta^{n+1}_*) \bu^{n+1}-(R+\eta^{n}_*)\bu^{n}\right) \left( \Delta t \sum_{k=n-j+1}^n\bu^{k,n}_{\sigma,\lambda}\right)    
	 	+\int_0^L(v^{n+1}-v^{n} )\left( \Delta t \sum_{k=n-j+1}^nv^{k}_{\lambda}\right) \\
	 	&-\frac{1}2\int_\sO \left( \eta_*^{n+1}-\eta_*^n\right)  \bu^{n+1}\cdot \left( \Delta t \sum_{k=n-j+1}^n\bu^{k,n}_{\sigma,\lambda}\right)  + \frac{(\Delta t)}{\ep}\int_\sO 
	 	\text{div}^{\eta^n_*}\bu^{n+1}\text{div}^{\eta^n_*}\left( \Delta t \sum_{k=n-j+1}^n\bu^{k,n}_{\sigma,\lambda}\right)  \notag
	\end{align}\begin{align}
	 	&+(\Delta t) b^{\eta_n^*}(\bu^{n+1},v^{n+1}r\be_r,\left( \Delta t \sum_{k=n-j+1}^n\bu^{k,n}_{\sigma,\lambda}\right) )
		\nonumber
		\\
	 	&+2\nu(\Delta t)\int_{\sO}(R+\eta^n_*) \bD^{\eta_*^{n}}(\bu^{n+1})\cdot \bD^{\eta_*^{n}}\left( \Delta t \sum_{k=n-j+1}^n\bu^{k,n}_{\sigma,\lambda}\right) 
		\nonumber
		\\
	 	 	&+(\Delta t)\int_0^L \partial_{z}\eta^{n+\frac12}  \partial_{z}\left( \Delta t \sum_{k=n-j+1}^nv^{k}_{\lambda}\right)  + \partial_{zz}\eta^{n+\frac12}  \partial_{zz}\left( \Delta t \sum_{k=n-j+1}^nv^{k}_{\lambda}\right)
			\nonumber
			\\
	 	&= (\Delta t)\left( P^n_{{in}}\int_{0}^1q^n_z\Big|_{z=0}dr-P^n_{{out}}\int_{0}^1q^n_z\Big|_{z=1}dr\right) +  
	 	(G(\bU^{n},\eta_*^n)\Delta_n W, \bQ_n).\notag
	 \end{align}
{{We will be estimating the terms in \eqref{MainWeakForm} to get the desired estimate for the time shifts.
We start by focusing on the first two terms. Applying summation by parts to the first two terms we obtain:}}
\begin{align*}
&-\sum_{n=0}^N\left( 	\int_{\sO}\left((R+\eta^{n+1}_*) \bu^{n+1}-(R+\eta^{n}_*)\bu^{n}\right) \left( \Delta t \sum_{k=n-j+1}^n\bu^{k,n}_{\sigma,\lambda}\right)
	+\int_0^L(v^{n+1}-v^{n} )\left( \Delta t \sum_{k=n-j+1}^nv^{k}_{\lambda}\right)\right) \\
&=(\Delta t)\sum_{n=1}^N\left( \int_{\sO}(R+\eta^{n}_*) \bu^{n}(\bu^{n,n}_{\sigma,\lambda}-\bu_{\sigma,\lambda}^{n-j,n})d\bx+\int_0^L v^n(v_{\lambda}^n-v_{\lambda}^{n-j})dz\right)\\
&- \int_\sO (R+\eta^{N+1}_*)\bu^{N+1} \left( \Delta t \sum_{k=N-j+1}^N\bu^{k,N}_{\sigma,\lambda}\right)-\int_0^L v^N \left( \Delta t \sum_{k=N-j+1}^Nv^{k}_{\lambda}\right).
\end{align*}
Here the first term on the right side can be written as 
\begin{align}\label{MainExpression}
&(\Delta t)\sum_{n=0}^N\left(
\int_{\sO}(R+\eta^{n}_*) \bu^{n}(\bu^{n}-\bu^{n-j})d\bx+ \int_{\sO}(R+\eta^{n}_*) \bu^{n}(\bu^{n,n}_{\sigma,\lambda}-\bu^{n}-(\bu_{\sigma,\lambda}^{n-j,n}-\bu^{n-j}))d\bx\right)\\
&+(\Delta t)\sum_{n=0}^N\int_0^L v^n(v^n-v^{n-j})dz+\int_0^L v^n(v^n_{\lambda}-v^n-(v^{n-j}_{\lambda}-v^{n-j}))
\nonumber\\
&=(\Delta t)\sum_{n=0}^N\left(
\frac12\int_{\sO}(R+\eta^{n}_*) (|\bu^{n}|^2-|\bu^{n-j}|^2+|\bu^{n}-\bu^{n-j}|^2)d\bx\right)
\nonumber\\
&+(\Delta t)\sum_{n=0}^N \left( \int_{\sO}(R+\eta^{n}_*) \bu^{n}(\bu^{n,n}_{\sigma,\lambda}-\bu^{n}-(\bu_{\sigma,\lambda}^{n-j,n}-\bu^{n-j}))d\bx\right)\notag\\
&+(\Delta t)\sum_{n=0}^N\frac12\int_0^L (|v^n|^2-|v^{n-j}|^2+|v^n-v^{n-j}|^2)dz+\int_0^L v^n(v^n_{\lambda}-v^n-(v^{n-j}_{\lambda}-v^{n-j})),\notag
\end{align}	 
where we set $\bu^n=0$ and $v^n=0$ for $n<0$ and $n>N$.
{{Notice that the terms on the right hand-side are written in terms of time-shifts and will be used to obtain the desired estimate.}}

 {{To estimate the terms on the right hand-side of \eqref{MainExpression} we}}
  observe that {{the terms containing $|\bu^{n}|^2-|\bu^{n-j}|^2$ can be estimated as follows:}}
 \begin{align*}
I_1&:=(\Delta t)\sum_{n=0}^N\left(
	\frac12\int_{\sO}(R+\eta^{n}_*) (|\bu^{n}|^2-|\bu^{n-j}|^2)d\bx\right)\\
	&=	\frac12(\Delta t)\left( \sum_{n=N-j+1}^N\int_{\sO}(R+\eta^{n}_*)|\bu^n|^2 d\bx	+\sum_{n=0}^{N-j}\int_{\sO}(\eta^{n}_*-\eta^{n+j}_*) |\bu^{n}|^2d\bx\right) \\
		&=	\frac12(\Delta t)\left( \sum_{n=N-j+1}^N\int_{\sO}(R+\eta^{n}_*)|\bu^n|^2 d\bx	+\sum_{n=0}^{N-j}\sum_{k=n}^{n+j-1}\int_{\sO}[(\Delta t)\theta_\delta(\eta^{n+1})v^{k+\frac12} ]|\bu^{n}|^2d\bx\right) \\
	&\leq	\frac12(\Delta t)\left( \sum_{n=N-j+1}^N\int_{\sO}(R+\eta^{n}_*)|\bu^n|^2 d\bx	+(\Delta t)\sum_{n=0}^{N-j}\sum_{k=n}^{n+j-1}\|v^{k+\frac12}\|_{L^2(0,L)} \|\bu^{n}\|^2_{\bL^4(\sO)}\right) \\
	&\leq  h \left( \max_{1\leq n \leq N}\int_{\sO}(R+\eta^{n}_*)|\bu^n|^2 d\bx+ \max_{0\leq k \leq N-1}\|v^{k+\frac12}\|_{L^2(0,L)}\sum_{n=0}^N(\Delta t)\|\bu^n\|^2_{\bH^1(\sO)} \right)  
\end{align*}
{{We now want to show that the probability $\bP\left(|I_1|>M \right)$ for any $M>0$ is bounded by an expression which {depends on} $h^m$ for $m > 0$ and 
$M^{-1}$.}} 
{{For this purpose, we will}} use the fact that for any two positive random variables $A$ and $B$, {{we have that}} 
\begin{equation}\label{property}
\{\omega: A+B>M\} \subseteq \{\omega: A>\frac{M}2\}\cup \{\omega: B>\frac{M}2\}.
\end{equation}
{{We will use this property repeatedly throughout the rest of this proof.}} Property \eqref{property}  implies that
\begin{align*}
	\bP(|A|+|B|>M) \leq \bP(|A|>\frac{M}2)+\bP(|B|>\frac{M}2).
	\shortintertext{Similarly we also have,}
		\bP(|AB|>M) \leq \bP(|A|>\sqrt{M})+\bP(|B|>\sqrt{M}).
\end{align*}
Hence,
\begin{align*}
	\bP&\left(|I_1|>M \right) \leq \bP\left( h\max_{1\leq n \leq N}\int_{\sO}(R+\eta^{n}_*)|\bu^n|^2 d\bx \geq \frac{M}2\right) \\
	&+ \bP\left( h\max_{0\leq k \leq N-1}\|v^{k+\frac12}\|_{L^2(0,L)}\sum_{n=0}^N(\Delta t)\|\bu^n\|^2_{\bH^1(\sO)}\geq \frac{M}2 \right)\\
	&\leq \bP\left( h\max_{1\leq n \leq N}\int_{\sO}(R+\eta^{n}_*)|\bu^n|^2 d\bx \geq \frac{M}2\right) \\
	&+ \bP\left( \sqrt{h}\max_{0\leq k \leq N-1}\|v^{k+\frac12}\|_{L^2(0,L)} \geq \sqrt{\frac{M}2}\right) +\bP\left( \sqrt{h}\sum_{n=0}^N(\Delta t)\|\bu^n\|^2_{\bH^1(\sO)}\geq \sqrt{\frac{M}2} \right)\\
	&\leq \frac{\sqrt{h}}{\sqrt{M}}\bE\left( \max_{1\leq n \leq N}\|\bu^n\|^2_{\bL^2(\sO)}  +\max_{0\leq k \leq N-1}\|v^{k+\frac12}\|_{L^2(0,L)}+ \sum_{n=0}^N(\Delta t)\|\bu^n\|^2_{\bH^1(\sO)} \right) \leq C\frac{\sqrt{h}}{\sqrt{M}}.
\end{align*}
Next, {{we estimate the second term on the right hand-side of \eqref{MainExpression}.}} Thanks to \eqref{u2}, we see that for any $s<\frac12$ and {for some $C$ that depends only on $T$ and $\delta$}, we have
\begin{align}\label{I2}
I_2 &:=	|(\Delta t)\sum_{n=0}^N \left( \int_{\sO}(R+\eta^{n}_*) \bu^{n}(\bu^{n,n}_{\sigma,\lambda}-\bu^{n}-(\bu_{\sigma,\lambda}^{n-j,n}-\bu^{n-j}))d\bx\right)|
\nonumber
\\
	& \leq C(\Delta t)\sum_{n=0}^N\|\bu^n\|_{\bL^2(\sO)}(\|\bu^{n,n}_{\sigma,\lambda}-\bu^n\|_{\bL^2(\sO)}+\|\bu^{n-j,n}_{\sigma,\lambda}-\bu^{n-j}\|_{\bL^2(\sO)})
	\nonumber
	\\
	&\leq C  h^{\frac{s}4}  \max\{1,\frac1{\delta}\|\tilde\eta^*_N\|^{2s+\frac32}_{C^{0,\frac14}(0,T;L^\infty(0,L))}\}
  \max_{1\leq n\leq N} \|\bu^n\|_{\bL^2(\sO)} \left( (\Delta t)\sum_{n=1}^{N} \|\bu^{k}\|^2_{\bH^1(\sO)}\right)^{\frac12} .
\end{align}
Hence, we take $s=\frac14$ and use the embedding $W^{1,\infty}(0,T;L^2(0,L))\cap L^2(0,T;H^2_0(0,L)) \hookrightarrow C^{0,\frac14}(0,T;H^\frac32(0,L))\hookrightarrow C^{0,\frac14}(0,T;L^\infty(0,L))$ to obtain,
\begin{align*}
	\bP(|I_2|>M)& \leq
	\frac{h^{\frac1{32}}}{M^{\frac12}}\bE\left( \|\tilde\eta^*_N\|^2_{C^{0,\frac14}(0,T;L^\infty(0,L))}+ \max_{1\leq n\leq N}\|\bu^n\|^2_{\bL^2(\sO)}+ \left((\Delta t) \sum_{n=0}^N 
	\|\bu^{n}\|^2_{\bH^1(\sO)}\right)\right) \\&\leq C\frac{h^{\frac1{32}}}{M^{\frac12}}.
\end{align*}
{{Similarly, the third term on the right hand-side of \eqref{MainExpression} can be estimated using \eqref{v1} to obtain}} 
\begin{align}
I_3&:=	|(\Delta t)\sum_{n=0}^N \int_0^L v^n(v^n_{\lambda}-v^n-(v^{n-j}_{\lambda}-v^{n-j}))| \notag\\
&	\leq C(T) \lambda^{\frac12}   \max_{0\leq n\leq N}\|v^n\|_{L^2(0,L)}\left( (\Delta t)\sum_{n=0}^N \|v^n\|^2_{H^\frac12(0,L)}\right) ^{\frac12}.\label{I3}
\end{align}
Hence,
{{by our choice of $\lambda$, see \eqref{lambda} above, we know that $\lambda \sim h^{1/4}$, and so we get}}
\begin{align*}
\bP(|I_3|>M) \leq 	\frac{h^{\frac1{8}}}{M}\bE\left( \max_{0\leq n\leq N}\|v^n\|^2_{L^2(0,L)} + (\Delta t)\sum_{n=0}^N \|v^n\|^2_{H^\frac12(0,L)}\right) \leq C\frac{h^{\frac1{8}}}{M^{}}.
\end{align*}

{{Now we go back to \eqref{MainWeakForm} and estimate the penalty term.}}
Thanks to \eqref{kton}, we observe that
\begin{align*}
& I_4:= \left|	\frac{(\Delta t)}{\ep}\sum_{n=0}^{N}\int_\sO 
	\text{div}^{\eta^n_*}\bu^{n+1}\left( \Delta t \sum_{k=n-j+1}^n\text{div}^{\eta^n_*}(\bu^{k,n}_{\sigma,\lambda})\right) d\bx \right|  \\
	&\leq \frac{(\Delta t)}{{\ep}}\sum_{n=0}^{N}
\|	\text{div}^{\eta^n_*}\bu^{n+1}\|_{L^2(\sO)}\left( 	{(\Delta t)} \sum_{k=n-j+1}^n\|\text{div}^{\eta^n_*}(\bu^{k,n}_{\sigma,\lambda})\|^2_{L^2(\sO)}\right)^{\frac12} \sqrt{h} \\
&\leq C \sqrt{hT} \,\sigma \left(  	\frac{(\Delta t)}{{\ep}} \sum_{n=0}^N\|\text{div}^{\eta^{n}_*}(\bu^{n+1})\|^2_{L^2(\sO)}\right)\\
&\leq C\sqrt{hT}\left(  {1+\|\tilde\eta^*_N\|_{C^{0,\frac14}(0,T;L^\infty(0,L))}h^{\frac14}}\right) \left(  	\frac{(\Delta t)}{{\ep}} \sum_{n=0}^N\|\text{div}^{\eta^{n}_*}(\bu^{n+1})\|^2_{L^2(\sO)}\right) .
\end{align*}
Hence,
\begin{align*}
	\bP(I_4>M) \leq \frac{\sqrt{h}}{M}\bE[1+\|\tilde\eta^*_N\|^2_{C^{0,\frac14}(0,T;L^\infty(0,L))}h^{\frac14}]+ \frac{{h}^{\frac14}}{M^{\frac12}}\bE\left(  	\frac{(\Delta t)}{{\ep}} \sum_{n=0}^N\|\text{div}^{\eta^{n}_*}(\bu^{n+1})\|^2_{L^2(\sO)}\right) \leq C\frac{h^{\frac14}}{M^{\frac12}}.
\end{align*}
Notice that due to Theorem \ref{energythm} (2), the constant $C$ in the estimate above does not depend on $N$ or $\ep$. % Next, thanks to the definition of the ALE velocity, we have $\|\bw^n\|_{\bL^p(\sO)}^p \leq  \|v^n\|^p_{L^p(0,L)}$. %Since the embedding $H^\alpha(\sO) \hookrightarrow L^q(\sO)$ for any $\alpha<\frac12$, $q\leq \frac2{1-\alpha}$, is continuous, we know that $\bu^{k,ext}$ and thus $\bu^k_\sigma$ are in $\bL^3(\sO_\delta)$. Moreover 

{{Next we estimate the term in \eqref{MainWeakForm} appearing after the penalty term.
For this purpose, notice that the continuity of}} 
the embedding
 $H^{\frac12}(0,L) \hookrightarrow L^p(0,L)$ for any $p<\infty$ {{implies}}
\begin{align}\label{boundl3}
	\|\tilde\bu^k_\sigma\| _{\bL^3(\sO_\delta)} \leq C(\delta)\left( \|\tilde\bu^{k,ext}\|_{\bL^3(\sO_{\eta^{k-1}_*})}+\|v^{k}\|_{L^3(0,L)}\right) \leq C(\delta)(\|\bu^{k}\|_{\bH^1(\sO)}+\|v^k\|_{H^\frac12(0,L)}).
\end{align}
%{{TODO: Add a sentence here about where is this used below.}}

Hence, for some $C>0$ that depends only on $\delta$, we obtain
\begin{align*}
	&I_5:=\left|	(\Delta t)\sum_{n=0}^N b^{\eta^n_*}\left( \bu^{n+1},v^{n+1}r\be_r,\left( \Delta t \sum_{k=n-j+1}^n\bu^{k,n}_{\sigma,\lambda}\right) \right) \right| \\
	&\leq\sqrt{h}	(\Delta t)\sum_{n=0}^N  \|\bu^{n+1}-v^{n+1}r\be_r\|_{\bL^6(\sO)} \|\nabla^{\eta^n_*}\bu^{n+1}\|_{\bL^2(\sO)} \left( \Delta t \sum_{k=n-j+1}^n\|\bu^{k,n}_{\sigma,\lambda}\|^2_{\bL^3(\sO)}\right)^{\frac12}\\
	&\leq\sqrt{h}	(\Delta t)\sum_{n=0}^N  \left( \|\bu^{n+1}\|_{\bH^1(\sO)}+\|v^{n+1}\|_{H^\frac12(0,L)}\right)  \|\nabla^{\eta^n_*}\bu^{n+1}\|_{\bL^2(\sO)} \left( \Delta t \sum_{k=n-j+1}^n\|\bu^{k,n}_{\sigma,\lambda}\|^2_{\bL^3(\sO)}\right)^{\frac12}\\
	&\leq { C(\delta)}\sqrt{h}\left( 	(\Delta t)\sum_{n=0}^N  \|\bu^{n+1}\|^2_{\bH^1(\sO)}\right)^{\frac32}. %+\sqrt{hT}\left( (\Delta t)\sum_{n=0}^N \|v^{n}\|^2_{H^\frac12(0,L)} \right)^{\frac12} \left( (\Delta t) \sum_{n=0}^N\|\bu^{n+1}\|^2_{\bH^1(\sO)} \right).
	\end{align*}
	{We used \eqref{boundl3} in the last step of the estimate above.}
	Hence for any $M>0$, we have
	\begin{align*}
		\bP\left( |	I_5 | ^{} \geq M^{}\right) \leq \frac{h^{\frac13}}{M^{\frac23}}\bE\left( (\Delta t)\sum_{n=0}^N  \|\bu^{n+1}\|^2_{\bH^1(\sO)}\right)\leq C\frac{h^{\frac13}}{M^{\frac23}}.
	\end{align*}
	Similar calculations give us {{that the term just before the penalty term in \eqref{MainWeakForm} can be estimated as follows:}}
		\begin{align*}
		&I_6:=	\left|\sum_{n=0}^N\int_\sO \left( \eta_*^{n+1}-\eta_*^n\right)  \bu^{n+1}\cdot \left( \Delta t \sum_{k=n-j+1}^n\bu^{k,n}_{\sigma,\lambda}\right) d\bx \right| \\
		&\leq (\Delta t)\sum_{n=0}^N\|v^{n+\frac12}\|_{L^2(0,L)}\|\bu^{n+1}\|_{\bL^6(\sO)}\left( (\Delta t)\sum_{k=n-j+1}^n    \|  \bu^{k,n}_{\sigma,\lambda} \|^2_{\bL^3(\sO)}\right) ^{\frac12}\sqrt{h}\\
	%	&\leq C\sqrt{hT} \max_{1\leq n\leq N}\|v^{n+\frac12}\|_{L^2(0,L)}  \left( (\Delta t)\sum_{n=0}^N    \|  \bu^{n} \|^2_{\bH^1(\sO)}\right)^{\frac12}\left( (\Delta t)\sum_{k=n-j+1}^n    \|  \bu^{k,n}_{\sigma,\lambda} \|^2_{\bL^3(\sO)}\right) ^{\frac12}\\
		&\leq C\sqrt{hT} \max_{1\leq n\leq N}\|v^{n+\frac12}\|_{L^2(0,L)}  \left( (\Delta t)\sum_{n=0}^N    \|  \bu^{n} \|^2_{\bH^1(\sO)}\right).
	\end{align*}
	Hence,
	\begin{align*}
	\bP\left( |	I_6 | \geq M^{}\right) \leq \frac{h^{\frac14}}{M^{\frac12}}\bE\left(\max_{1\leq n\leq N}\|v^{n+\frac12}\|_{L^2(0,L)}+ (\Delta t)\sum_{n=0}^N  \|\bu^{n+1}\|^2_{\bH^1(\sO)}\right)\leq  C\frac{h^{\frac14}}{M^{\frac12}}.
\end{align*}

{{Next, we focus on the integral in \eqref{MainWeakForm} which contains the second-order derivative of displacement $\partial_{zz}\eta^{n+1}$.
	Since $\|\phi_\lambda\|_{H^2} \leq \frac{C}{\lambda^{2}}\|\phi\|_{L^2}$, we obtain}}
	\begin{align*}
&I_7:=	|(\Delta t)\sum_{n=0}^N	\int_0^L  \partial_{zz}\eta^{n+1}  \partial_{zz}\left( \Delta t \sum_{k=n-j+1}^nv^{k}_{\lambda}\right)|%\leq C T{h}\left( \|\partial_{zz}\eta^n \|_{L^2(0,L)} \|\partial_{zz}v^n_{\lambda}\|_{L^2(0,L)} \right)  \\&\qquad
\leq {CT}\frac{h}{\lambda^2}\max_{0\leq n\leq N}\left( \|\partial_{zz}\eta^n \|_{L^2(0,L)} \|v^n\|_{L^2(0,L)} \right)  .
	\end{align*}
	Now since $\lambda \sim h^{\frac14}$, we have
	\begin{align*}
		\bP(I_7 \geq M)  &\leq \frac{Ch^{\frac12}}{M}  \bE\left( %\|\tilde\eta^*_N\|^2_{C^{0,\frac14}(0,T;L^\infty(0,L))}+
		\max_{0\leq n\leq N}\|\partial_{zz}\eta^n \|^2_{L^2(0,L)}+ \max_{0\leq n\leq N}\|v^n\|^2_{L^2(0,L)}\right) 
		\leq C\frac{h^{\frac12}}{M} .
	\end{align*}
	
	{{What is left to estimate in \eqref{MainWeakForm} are the integral containing the transformed symmetrized gradients and the stochastic term. 
	We first focus of the term with the transformed symmetrized gradients.}}
Similarly as before, since $\|\phi_\lambda\|_{H^1} \leq \frac{C}{\lambda}\|\phi\|_{L^2}$, the term with the transformed symmetrized gradients yields:
	\begin{align*}
		 I_8:=&|	{(\Delta t)}{}\sum_{n=0}^{N}\int_\sO 
		\bD^{\eta^n_*}\bu^{n+1}\left( \Delta t \sum_{k=n-j+1}^n\bD^{\eta^n_*}(\bu^{k,n}_{\sigma,\lambda})\right) d\bx |  \\
		&\leq (\Delta t)\sum_{n=0}^{N}
		\|	\bD^{\eta^n_*}\bu^{n+1}\|_{\bL^2(\sO)}\left( {(\Delta t)} \sum_{k=n-j+1}^n\|\bD^{\eta^n_*}(\bu^{k,n}_{\sigma,\lambda})\|_{\bL^2(\sO)}\right)^{} \\
			&\leq \frac1{\lambda}(\Delta t)\sum_{n=0}^{N}
		\|	\bD^{\eta^n_*}\bu^{n+1}\|_{\bL^2(\sO)}\left( {(\Delta t)} \sum_{k=n-j+1}^n\|\bu^{k}\|_{\bL^2(\sO)}+\|v^{k}\|_{L^2(0,L)}\right) \\
		&\leq C \frac{{h}}{\lambda} \left( \max_{0\leq n\leq N}(\|\bu^n\|_{\bL^2(\sO)}+\|v^k\|_{L^2(0,L)}) 	{(\Delta t)}{{}} \sum_{n=0}^N\|\bD^{\eta^{n}_*}(\bu^{n+1})\|_{\bL^2(\sO)}\right).
	\end{align*}
	Hence, again, since we have
	 $\lambda\sim h^{\frac14}$, Young's inequality gives us,
	\begin{align*}
		\bP(I_8>M) & \leq \frac{{h}^{\frac34}}{M}\bE[\max_{0\leq n\leq N}(\|\bu^n\|^2_{\bL^2(\sO)}+\|v^k\|^2_{L^2(0,L)}) ]+ \frac{{h}^{\frac34}}{M^{}}\bE\left(  	{(\Delta t)} \sum_{n=0}^N\|\bD^{\eta^{n}_*}(\bu^{n+1})\|^2_{\bL^2(\sO)}\right) \\
		&\leq C\frac{h^\frac34}{M}.
	\end{align*}
	
	Finally, we treat the stochastic term using the same argument  as above. To bound the expectation we use Young's inequality and the argument presented in \eqref{tower} to obtain,
	\begin{align*}
		&\bE\left( \sum_{n=0}^N	|(G(\bU^{n},\eta_*^n)\Delta_n W, \bQ_n)|\right) \\
		&	\leq 		\bE\left( \sum_{n=0}^N	\|G(\bU^{n},\eta_*^n)\|_{L_2(U_0;\bL^2)}\|\Delta_n W\|_{U_0}  \left( (\Delta t) \sum_{k=n-j+1}^n\|\bu^{k,n}_{\sigma,\lambda}\|_{\bL^2(\sO)}^2 +\|v^{k}_{\lambda}\|_{L^2(0,L)}^2\right)^{\frac12}h^{\frac12}\right) 		\end{align*}	\begin{align*}
		& 	\leq 	h^{\frac12}	\bE\left( \sum_{n=0}^N	\|G(\bU^{n},\eta_*^n)\|^2_{L_2(U_0;\bL^2)}\|\Delta_n W\|^2_{U_0}  +\left( (\Delta t) \sum_{k=n-j+1}^n\|\bu^{k,n}_{\sigma,\lambda}\|_{\bL^2(\sO)}^2 +\|v^{k}_{\lambda}\|_{L^2(0,L)}^2\right)\right) 	\\
		&\leq h^{\frac12} C(\delta,\Tr \bQ)\bE\sum_{n=0}^N(\Delta t)[\|\bu^n\|^2_{\bL^2(\sO)}+\|v^n\|^2_{L^2(0,L)}] \leq Ch^{\frac12}.
	\end{align*}
	
	{{By combining these estimates we are now in a position to show}}
that the laws of the random variables mentioned in the statement of the theorem are tight. 
{{For this purpose we recall the definition of the set ${\mathcal{B}}_M$  for $0\leq \alpha<1$ and any $M>0$:}}
\begin{align*}
	{\mathcal{B}}_M:=&\{(\bu,v)\in L^2(0,T;\bH^\alpha(\sO))\times L^2(0,T;L^2(0,L)):
	\|{\bu}\|^2_{L^2(0,T;\bH^1(\sO))}+\|{v}\|^2_{L^2(0,T;H^{\frac12}(0,L))}
	\\
	&+\sup_{0<h<1}{h^{-\frac1{32}}}\int_h^{T}\left( \|T_h\bu-\bu\|^2_{\bL^2(\sO)}+\|T_hv-v\|^2_{L^2(0,L)}\right)  \le M\}.
\end{align*}
Observe that, due to Lemma \ref{compactLp}, $\sB_M$ is compact in $L^2(0,T;\bH^\alpha(\sO))\times L^2(0,T;L^2(0,L))$ for each $M>0$.
Now, an application of Chebyshev's inequality gives the desired result:
\begin{align*}
	\bP((\bu^+_N,v^+_N) \notin \sB_M)&\leq \bP\left( \|{\bu^+_N}\|^2_{L^2(0,T;\bH^1(\sO))}+\|{v^+_N}\|^2_{L^2(0,T;H^{\frac12}(0,L))}>\frac{M}2\right) \\
	&+\bP\left( \sup_{0 <h<1 }{h^{-\frac1{32}}}\int_h^{T}\left( \|T_h\bu^+_N-\bu^+_N\|^2_{\bL^2(\sO)}+\|T_hv^+_N-v^+_N\|^2_{L^2(0,L)}\right)  > \frac{M}2\right) \\
	&\leq \frac{C}{\sqrt{M}},
\end{align*}
where $C>0$ depends only on $\delta$, Tr$Q$ and the given data and is independent of $N$ and $\ep$.
This completes the proof of Lemma~\ref{tightuv}.
\end{proof}
Next we will state the rest of the tightness results. These will be used in Section \ref{almostsure} to obtain almost sure convergence via an application of Prohorov's theorem and the Skorohod representation theorem. 
\begin{lem}\label{tightl2}
	For fixed $\delta>0$, the following statements hold:
	\begin{enumerate}
	%	\item The laws of $\{\bu^+_{N}\}_{N\in \mathbb{N}}$ are tight in $L^2(0,T,\bL^2(\sO))$.
		%\item The laws of $\{v^{\#}_{N}\}_{N\in \mathbb{N}}$ and that of $\{v^{*}_{N}\}_{N\in \mathbb{N}}$ 		are tight in $L^2(0,T,L^2(0,L))$.
		\item The laws of $\{\tilde\eta_{N}\}_{N\in \mathbb{N}}$ and that of $\{\tilde\eta^*_{N}\}_{N\in \mathbb{N}}$ are tight in $C([0,T],H^{s}(0,L))$ for any $s<2$.
		\item The laws of $\{\|\bu^+_N\|_{L^2(0,T;V)}\}_{N\in \mathbb{N}}$ are tight in $\R$.
		\item The laws of $\{\|v^*_N\|_{L^2(0,T;L^2(0,L))}\}_{N\in \mathbb{N}}$ are tight in $\R$.
	\end{enumerate}
\end{lem}
\begin{proof}
\iffalse 
	For $M>0$ consider
	\begin{align*}
	{\mathcal{B}}_M:=\{\bu\in L^2(0,T;\bH^1(\sO))\cap H^{\alpha}(0,T;\bL^2(\sO)):\|{\bu}\|^2_{L^2(0,T;\bH^1(\sO))}+\|{\bu}\|^2_
	{H^{\alpha}(0,T;\bL^2(\sO))}\le M\}.
	\end{align*}
	Thanks to Lemma \ref{compactLp}, we know that ${\mathcal{B}}_M$ is a compact subset of $L^2(0,T;\bL^2(\sO))$. By using Chebyshev's inequality, we see that	
	\begin{equation}
	\begin{aligned}
	\bP\left[\bu^+_{N}\not\in {\mathcal{B}}_M\right]&\le \bP\left[
	{\|\bu^+_{N}\|}^2_{L^2(0,T;\bH^1(\sO))}\ge \frac{M}{2}\right]+
	\bP\left[
	{\|\bu^+_{N}\|}^2_{H^{\alpha}(0,T;\bL^2(\sO))}\ge \frac{M}{2}\right]	\\
	&\le  \frac{4}{M^2}\bE\left[
	{\|\bu^+_{N}\|}^2_{L^2(0,T;\bH^1(\sO))}+{\|\bu^+_{N}\|}^2_
	{H^{\alpha}(0,T;\bL^2(\sO))}\right]
	\leq \frac{C}{M^2}.
	\end{aligned}
	\end{equation}
	{This proves statement (1).} Statement (2)  follows similarly. 
	\fi
	To prove the first statement we observe, thanks to Lemma \ref{bounds}, that $\tilde\eta_{N}$ and $\tilde\eta_{N}^*$ are bounded independently of $N$ and $\ep$ in $L^2(\Omega;L^\infty(0,T;H^2_0(0,L))\cap W^{1,\infty}(0,T;L^2(0,L)))$. Now the Aubin-Lions theorem {implies} that for $0<s<2$
	$$L^\infty(0,T;H^2_0(0,L))\cap W^{1,\infty}(0,T;L^2(0,L)) \subset\subset C([0,T];H^{s}(0,L)). 
	$$
	Hence consider
	\begin{align*}\mathcal{K}_M:=&\{\eta\in L^\infty(0,T;H^2_0(0,L))\cap W^{1,\infty}(0,T;L^2(0,L)):\\
	&\|\eta\|^2_{L^\infty(0,T;H^2_0(0,L))} + \|\eta\|^2_{ W^{1,\infty}(0,T;L^2(0,L))}\leq M\}.\end{align*}
	Using the Chebyshev inequality once again we obtain for some $C>0$ independent of $N,\ep$ that {the following holds:}
	\begin{equation}
	\begin{aligned}
	\bP\left[\tilde\eta_{N}\not\in {\mathcal{K}}_M\right]&\le \bP\left[
	{\|\tilde\eta_{N}\|}^2_{L^\infty(0,T;H^2_0(0,L))}\ge \frac{M}{2}\right]+
	\bP\left[
	{\|\tilde\eta_{N}\|}^2_{W^{1,\infty}(0,T;L^2(0,L))}\ge \frac{M}{2}\right]	\\
	&\le  \frac{4}{M^2}\bE\left[
	{\|\tilde\eta_{N}\|}^2_{L^\infty(0,T;H^2_0(0,L))}+\|\tilde\eta_{N}\|^2_
	{W^{1,\infty}(0,T;L^2(0,L))}\right]
	\leq \frac{C}{M^2}.
	\end{aligned}
	\end{equation}
Similarly to prove the second statement we use the Chebyshev inequality again to write for any $M>0$,
	\begin{align*}
	\bP[\| \bu^+_N\|_{L^2(0,T;V)}>M] \leq \frac1{M^2}\bE[\| \bu^+_N\|^2_{L^2(0,T;V)}] \leq \frac{C}{M^2}.
	\end{align*}
The third statement is proven identically.
		This completes the proof of the tightness results stated in Lemma \ref{tightl2}.
\end{proof}
So far we have been successful at obtaining tightness results only for a subset of the random variables defined at the beginning of Section \ref{subsec:approxsol}. However when passing to the limit we will also require almost sure convergence of the rest of the random variables for which we will use the following lemma.
\begin{lem}\label{difference}
	For a fixed $\delta>0$, the following convergence results hold: 
	\begin{enumerate}
		\item $\lim_{N\rightarrow\infty} \bE\int_0^T\|\bu_{N}-\bu^+_{N}\|^2_{\bL^2(\sO)}dt=0$,
		\quad $\lim_{N\rightarrow\infty} \bE\int_0^T\|\bu_{N}-\tilde\bu_{N}\|^2_{\bL^2(\sO)}dt=0$
		\item $\lim_{N\rightarrow\infty} \bE\int_0^T\|\eta_{N}-\tilde\eta_{N}\|^2_{H_0^2(0,L)}dt=0$,
		\quad $\lim_{N\rightarrow\infty} \bE\int_0^T\|\eta^+_{N}-\tilde\eta_{N}\|^2_{H_0^2(0,L)}dt=0$
		\item $\lim_{N\rightarrow\infty} \bE\int_0^T\|\eta^*_{N}-\tilde\eta^*_{N}\|^2_{H_0^2(0,L)}dt=0$
		\item  $\lim_{N\rightarrow\infty} \bE\int_0^T\|v_{N}-\tilde v_{N}\|^2_{L^2(0,L)}dt=0$, \quad $\lim_{N\rightarrow\infty} \bE\int_0^T\|v_{N}-v^{\#}_{N}\|^2_{L^2(0,L)}dt=0$. 
	\end{enumerate}
\end{lem}
\begin{proof} Statement (1)$_1$ follows immediately from Theorem \ref{energythm} (3). We prove (1)$_2$ below. The rest follows similarly {from the uniform estimates stated in} Theorem \ref{energythm}.
	\begin{align*}
	&\bE\int_0^T\|\bu_{N}-\tilde\bu_{N}\|^2_{\bL^2(\sO)}dt=
	\bE\sum_{n=0}^{N-1}\int_{t^n}^{t^{n+1}}\frac1{\Delta t}\|(t-t^n)\bu^{n+1}+(t^{n+1}-t-\Delta t)\bu^{n}\|^2_{\bL^2(\sO)}dt\\
	&=\bE\sum_{n=0}^{N-1}\|\bu^{n+1}-\bu^n\|_{\bL^2(\sO)}^2\int_{t^n}^{t^{n+1}} \left( \frac{t-t^n}{\Delta t}\right) ^2dt\leq \frac{CT}{\delta N} \rightarrow 0 \quad N\rightarrow \infty.
	\end{align*}
\end{proof}
While formulating an approximate system in terms of the piecewise linear and piecewise constant functions \eqref{AppSol} and \eqref{approxlinear} respectively, defined in Section \ref{subsec:approxsol}, we come across typical error terms generated as a result of discretizing the stochastic term. {To obtain the final convergence result we need estimates of those terms. We derive those error terms and obtain their estimates next.}

Observe that for given $N$ we can write
$$\tilde \eta_N^*=\eta_N^*+{(t-t^n)}\frac{\partial\tilde\eta_N^*}{\partial t},\quad  t\in(t^n,t^{n+1}),$$
and that the same {equation} holds true when $\eta^*$ is replaced by $\bu$. 
This gives for each $\omega\in\Omega$ and $t \in (t^n,t^{n+1})$ that
\begin{align*}
\frac{\partial((R+\tilde\eta_N^*)\tilde\bu_N)}{\partial t}&=\frac{\partial \tilde\bu_N}{\partial t}(R+\eta_N^*) +\frac{\partial \tilde\eta_N^*}{\partial t}(2\tilde\bu_N-\bu_N)\\
&=\sum_{n=0}^{N-1}\frac1{\Delta t}(\bu^{n+1}-\bu^n)(R+\eta^n_*) \,\chi_{(t^n,t^{n+1})} +v^*_N (2\tilde\bu_N-\bu_N).
\end{align*}
By integrating in time both sides and adding the equation for $\tilde v_N$ from \eqref{first} we obtain for 
any $(\bq,\psi)\in\sU_1$:
\begin{align}\label{I123}
&(	(R+\tilde\eta_N^*(t))\tilde\bu_N(t),\bq)+(\tilde v_N(t),\psi)\notag\\
&=\int_\sO\bu_0(R+\eta_0)\cdot\bq d\bx+\int_0^Lv_0\psi dz-\int_0^t\int_0^L(\partial_z\eta^+_N\partial_z\psi+\partial_{zz}\eta^+_N,\partial_{zz}\psi)dzds\notag\\
&+\sum_{n=0}^{N-1}\int_{0}^{t}\chi_{(t^n,t^{n+1})}\left( \int_{\sO}\frac{(\bu^{n+1}-\bu^n)}{\Delta t}(R+\eta^n_*) \,\bq \,d\bx 
+ \int_{0}^L\frac{(v^{n+\frac12}-v^n)}{\Delta t} \,\psi dz\right) ds
\notag\\
&
+\int_0^t\int_{\sO}v^*_N (2\tilde\bu_N-\bu_N)\bq dxds\notag\\
& =:(\bu_0(R+\eta_0),\bq)+(v_0,\psi) + (I_1(t)+I_2(t)+I_3(t),\bQ).
\end{align}
{We estimate the right hand-side to be able to pass to the limit as $N\to \infty$. We start by considering the integral $I_2$.} Thanks to \eqref{second}, $I_2$ on the right hand side is equal to
\begin{equation}\begin{split}\label{I2i}
&(I_2(t),\bQ)=-\frac{1}2\int_0^t\int_\sO 
v^*_N \bu^{+}_N\bq d\bx ds \\
&-2\nu\int_0^t\int_{\sO}(R+\eta_N^*)\bD^{\eta_N^*}(\bu^{+}_N)\cdot \bD^{\eta_N^*}(\bq) d\bx ds-\frac1{\ep}\int_0^t \int_\sO  \text{div}^{\eta^*_N}\bu^+_N\text{div}^{\eta^*_N}\bq d\bx ds \\
&-\frac12\int_0^t\int_{\sO}(R+\eta_N^*)((\bu^+_N-
v^{+}_Nr\be_r)\cdot\nabla^{\eta_N^*}\bu^{+}_N\cdot\bq
- (\bu^{+}_N-
v^{+}_N 
r\be_r)\cdot\nabla^{\eta_N^*}\bq\cdot\bu^{+}_N)d\bx ds\\
&+ 
\int_0^t\left( P_{{in}}\int_{0}^1q_z\Big|_{z=0}dr-P_{{out}}\int_{0}^1q_z\Big|_{z=1}dr\right)%\chi_{(t_n,t_{n+1})}
ds +\int_0^t
(G(\bu_N,v_N,\eta^*_N)dW, \bQ )+(E_N(t),\bQ)\\
&=:(I_{2,1}(t)+I_{2,2}(t)+I_{2,3}(t)+I_{2,4}(t)+I_{2,5}(t)+I_{2,6}(t)+E_N(t),\bQ),
\end{split}\end{equation}
where the error term is given by
$$E_N(t)=\sum_{m=0}^{N-1}\left( \frac{t-t^m}{\Delta t}G(\bu^m,v^m,\eta^m_*)\Delta_mW-\int_{t^m}^tG(\bu^m,v^m,\eta^m_*)dW\right) \chi_{[t^m,t^{m+1})}.$$
{We will study the behavior of each term on the right hand-side of the expression for $I_2$ as { $N\to\infty$}. 
	We start by showing that the numerical error term $E_N$ approaches $0$ as $N\to\infty$. The remaining integrals will be considered
	in Lemma \ref{uvtightC} below. 
}
\begin{lem}\label{u*diff}
	{The numerical error $E_N$ of the stochastic term has the following property:}
	$$ {\bE\int_0^T\|E_N(t)\|^2_{\bL^2(\sO)\times L^2(0,L)}dt  \rightarrow } 0 \text{ as } {N\rightarrow\infty}.$$
\end{lem}
\begin{proof}
	First, for any $N$, we have 
	$$E_N(t)=\sum_{m=0}^{N-1}\left( \frac{t-t^m}{\Delta t}G(\bu^m,v^m,\eta^m_*)\Delta_mW-\int_{t^m}^tG(\bu^m,v^m,\eta^m_*)dW\right) \chi_{[t^m,t^{m+1})}=:E_N^1+E_N^2$$
	{We estimate $E^1_N$ and $E^2_N$. Recall our notation $\bL^2=\bL^2(\sO)\times L^2(0,L)$.	Observe that $E^1_N$ satisfies (see also \eqref{tower})},
	\begin{align*}
	&\bE\int_0^T\|E_N^1(t)\|_{\bL^2}^2dt	=\bE\sum_{n=0}^{N-1} \|G(\bu^n,v^n,\eta^n_*)\Delta_nW\|_{\bL^2}^2 \int_{t^n}^{t^{n+1}}|\frac{t-t^m}{\Delta t}|^2dt\\
	&=\bE\sum_{n=0}^{N-1} \|G(\bu^n,v^n,\eta^n_*)\Delta_nW\|_{\bL^2}^2 \frac{\Delta t}{3 }\leq \bE\sum_{n=0}^{N-1} \|G(\bu^n,v^n,\eta^n_*)\|_{L_2(U_0;\bL^2)}^2\|\Delta_nW\|_{U_0}^2 {\Delta t}\\
	&	= (\Delta t)^2\bE\sum_{n=0}^{N-1} \|G(\bu^n,v^n,\eta^n_*)\|_{L_2(U_0;\bL^2)}^2 \leq C \Delta t,
	\end{align*}
	where $C>0$ does not depend on $N$ or $\ep$, as a consequence of Theorem \ref{energythm}.
	
	{To estimate $E^2_N$ we use the It\^{o} isometry to obtain}
	\begin{align*}
	&\bE\int_0^T\|E_N^2(t)\|_{\bL^2}^2dt= \bE\sum_{n=0}^{N-1}\int_{t^n}^{t^{n+1}}\|\int_{t^n}^tG(\bu^n,v^n,\eta_*^n)dW\|_{\bL^2}^2dt\\
	&= \bE\sum_{n=0}^{N-1}\int_{t^n}^{t^{n+1}}\int_{t^n}^t\|G(\bu^n,v^n,\eta_*^n)\|_{L_2(U_0;\bL^2)}^2dsdt
	\\&= \frac12\bE\sum_{n=0}^{N-1}\|G(\bu^n,v^n,\eta_*^n)\|_{L_2(U_0;\bL^2)}^2(\Delta t)^2 \leq C\Delta t.
	\end{align*}	
	{Thus}, as $N \rightarrow \infty$ we {obtain}
	$\bE\int_0^T\|E_N(t)\|^2_{\bL^2(\sO)\times L^2(0,L)}dt \rightarrow 0.$	
\end{proof}  
{To be able to pass to the limit a $N\to\infty$  in the weak formulation, we consider the following random variable:
	\begin{align}\label{UN}
	\tilde\bU_N(t):=(	(R+\tilde\eta_N^*(t))\tilde\bu_N(t),\tilde v_N(t))-E_N(t),
	\end{align}
	and study its behavior as $N\to \infty$. We} prove the following tightness result, Lemma \ref{uvtightC}, for the laws of $\tilde\bU_N$. {We do this by employing appropriate stopping time arguments (see \cite{GTW} for a similar argument) and thus bypass the need for higher moments.}\\
We begin by defining 
\begin{align}\label{U1}
\sU_1:=\sU\cap (\bH^2(\sO)\times H^2_0(0,L)).
\end{align}
{where $\sU$ is given in \eqref{su}.}
Next for any $\sV_1\subset\subset \sU_1$, we denote by $\mu^{u,v}_N$ the probability measure of $\tilde\bU_N$: $$\mu^{u,v}_{N}:=\bP\left( \tilde\bU_{N}\in \cdot \right) \in Pr(C([0,T];\sV'_1)) ,$$
where $Pr(S)$, here and onwards, denotes the set of probability measures on a metric space $S$. Then we have the following tightness result.
\begin{lem}\label{uvtightC}
	For fixed $\ep>0$ and $\delta>0$, the laws $\{\mu_N^{u,v}\}_N$ of {the random variables } $\{\tilde\bU_N\}_{N}$ 
	are tight in 
	$C([0,T];\sV_1')$.
\end{lem}
\begin{proof}
	{Recall that the random variables $\{\tilde\bU_{N}\}_N$ consist of two parts, a deterministic part and a stochastic part. 
		We will show that the stochastic terms belong to 
		$$
		B^S_M := \{{\bf Y} \in C([0,T];\sV_1');\|{\bf Y}\|_{W^{\alpha,q}(0,T;\sU_1')}\leq M\}
		$$ 
		while the deterministic terms belong to 
		$$B^D_M:=\{{\bf X} \in C([0,T];\sV_1');\|{\bf X}\|_{H^{1}(0,T;\sU'_1)}\leq M\},
		$$
		for some $M>0$ and $\alpha$ such that $0<\alpha <\frac12$ and $\alpha q>1$ where $q>2$.
		To do this, we first define
		$$
		B_M := B^D_M+B^S_M
		$$
		and note that } 	
	$$\{\tilde\bU_N\in B_M\} \supseteq \{I_1+\sum_{i=1}^5I_{2,i}+I_3 \in B^D_M\}\cap\{I_{2,6} \in B^S_M\},$$
	{where the integrals $I_1, I_{2,i}$ and $I_3$ are defined in \eqref{I123} and \eqref{I2i}.}
	{This implies that} 
	for the complement of the set $B_M$ we have,
	$$ \mu^{u,v}_{N} \left( B^c_M\right) \leq     \bP\left( \|I_1+\sum_{i=1}^5 I_{2,i}+I_3\|_{H^{1}(0,T;\sU_1')} > M\right) + \bP\left( \|I_{2,6}\|_{W^{\alpha,q}(0,T;\sU_1')} > M\right) .$$
	We want to show that for some $C=C(\delta,T,\ep)>0$ independent of $N$ we have 
		$$\mu^{u,v}_{N}(B_M)\geq 1-\frac{C}{M}.$$
		To do this we need to bound the integrals $I_1$, $I_{2,i}$ and $I_3$.
		These bounds rely heavily on the uniform bounds obtained in Lemma \ref{bounds}.
		To estimate these terms we introduce the following simplified notation:
		$
		\bar\sup_{\bQ}:=\sup_{\bQ \in \sU_1, |\bQ|_{\sU_1}=1}.
		$

	{We start with $I_1$. The estimate of $I_1$ follows by observing that there exists some $C>0$ independent of $N,\ep$ such that}	
	\begin{align*}
	\bE\|\partial_t I_{1}(t)\|^2_{L^2(0,T;\sU_1')} &\leq \bE\,
	\bar\sup_{\bQ}\int_0^T|\int_0^L(\partial_z\eta^+_{N}\partial_z\psi+\partial_{zz}\eta^+_{N}\partial_{zz}\psi)dz|^2dt\\
	& \leq \bE\int_0^T\|\eta^+_{N}(t)\|^2_{H^2_0(0,L)}dt \leq C.
	\end{align*}
	Using the continuous Sobolev embedding $\bH^2(\sO) \xhookrightarrow{} \bL^\infty(\sO)$  we similarly treat { $I_{2,1}+I_3$},
	\begin{align*}
	&	\bE\|\partial_t ({ I_{2,1}+I_3})\|_{L^2(0,T;\sU_1')} \leq \bE\left(\bar\sup_{\bQ}\int_0^T\|v^*_{N}\|^2_{L^2(0,L)}\|(2\tilde\bu_{N}-\bu_{N}-\frac12\bu^+_{N})\|_{\bL^2(\sO)}^2\|\bq\|_{\bL^\infty(\sO)}^2dt\right) ^{\frac12}\\
	& \leq \bE\left( \|v^*_{N}\|^2_{L^\infty(0,T;L^2(0,L))}\left( \|\tilde\bu_{N}\|^2_{L^\infty(0,T;\bL^2(\sO))}+\|\bu_{N}\|^2_{L^\infty(0,T;\bL^2(\sO))}+\|\bu^+_{N}\|^2_{L^\infty(0,T;\bL^2(\sO))}\right) \right)^{\frac12} \\
	& \leq \left( \bE\|v^*_{N}\|^2_{L^\infty(0,T;L^2(0,L))}\left( \bE\|\tilde\bu_{N}\|^2_{L^\infty(0,T;\bL^2(\sO))}+\bE\|\bu_{N}\|^2_{L^\infty(0,T;\bL^2(\sO))}+\bE\|\bu^+_{N}\|^2_{L^\infty(0,T;\bL^2(\sO))}\right)\right) ^{\frac12} 
	\\	& \leq C.
	\end{align*}
	{	To estimate $I_{2,2}$ we first recall that we can write} $\nabla^{\eta_N^*}\bq =\nabla\bq(\nabla A^\omega_{\eta_N^*})^{-1}$. {Now we note that} for $\frac32<s<2$ we have that $\|R+\eta^*_N\|_{H^s(0,L)}<\frac1{\delta}$ for any $\omega\in\Omega$ and $t\in[0,T]$, which implies that $\sup_{t\in[0,T]}\| (\nabla A^\omega_{\eta_N^*})^{-1}\|_{\bL^{\infty}(\sO)}<C$ and thus, \begin{align}\label{gradq}\sup_{t\in[0,T]}\|\nabla^{\eta_N^*}\bq\|_{\bL^2(\sO)} \leq C\|\nabla\bq\|_{\bL^2(\sO)}\quad \forall \omega\in\Omega,
	\end{align}
	where $C$ depends only on $\delta$. The same argument gives us that $$\sup_{t\in[0,T]}\|\text{div}^{\eta^*_N}\bq\|_{L^2(\sO)}\leq \sup_{t\in[0,T]}\sqrt{2}\|\nabla^{\eta^*_N}\bq\|_{\bL^2(\sO)}\leq C\|\nabla \bq\|_{\bL^2(\sO)}.$$
	
	{	Thus for $I_{2,2}$ we obtain the following bounds}:
	\begin{align*}
	\bE\|\partial_t I_{2,2}(t)\|^{2}_{L^{2}(0,T;\sU_1')} &\leq \bE\left(\bar\sup_{\bQ}\int_0^T\|\bD^{\eta^*_{N}}\bu^+_{N}\|^{2}_{\bL^{2}(\sO)}\|\bD^{\eta^*_{N}}\bq\|_{\bL^2(\sO)}^{2}dt\right) \\
	&\leq C(\delta)\bE\left(\bar\sup_{\bQ}\int_0^T\|\bD^{\eta^*_N}\bu^+_{N}\|^{2}_{\bL^{2}(\sO)}\|\nabla\bq\|_{\bL^{2}(\sO)}^{2}dt\right) 
	 \leq C(\delta).
	\end{align*}
	{To estimate $I_{2,3}$ we note that for a fixed $\ep>0$, Lemma \ref{bounds} implies the existence of a constant $C$ dependent on $\delta$ such that:}
	\begin{align*}
	\bE\|\partial_t I_{2,3}(t)\|_{L^{2}(0,T;\sU_1')}^{2} &\leq C(\delta) \frac1{\ep^2}\bE\left(\bar\sup_{\bQ}\int_0^T\|\text{div}^{\eta^*_{N}}\bu^+_{N}\|^{2}_{L^2(\sO)}\|\text{div}^{\eta^*_{N}}\bq\|^{2}_{L^2(\sO)}\right)^{}\\
	&\leq C(\delta) \frac1{\ep^2}\bE\left(\int_0^T\|\text{div}^{\eta^*_{N}}\bu^+_{N}\|^{2}_{L^2(\sO)}\right)^{}
	\leq \frac{C(\delta)}{\ep} .
	\end{align*} Observe that the estimate above depends on $\ep$ and hence will not be available in the next section when we pass $\ep\to 0$.
	{The integral $I_{2,4}$ can be estimated similarly after the use of the Ladyzenskaya inequality \cite{T_NSE} as follows:}
	\begin{align*}
	\bE\|\partial_t I_{2,4}(t)&\|_{L^{2}(0,T;\sU_1')} 
	\leq C(\delta)\bE\,\Big( \bar\sup_{\bQ}\int_0^T \Big( (\|\bu^+_{N}\|_{\bL^2(\sO)}+{\|v^{+}_{N}\|_{L^2(0,L)}})\times\\
	&\qquad\qquad(\|\nabla^{\eta^*_N} \bu^+_{N}\|_{\bL^2(\sO)}\|\bq\|_{\bL^\infty(\sO)}+\|\nabla^{\eta^*_N} \bq\|_{\bL^4(\sO)}\|\bu^+_{N}\|_{\bL^4(\sO)})\Big) ^{2}dt\Big) ^{\frac12}\\
	&\leq C(\delta) \bE\,\left( \int_0^T (\|\bu^+_{N}\|_{\bL^2(\sO)}+\|v^{+}_{N}\|_{L^2(0,L)})^{2}\|\nabla \bu^+_{N}\|_{\bL^2(\sO)}^{2}dt\right)^{\frac12} \\
	&\leq C(\delta) \bE \left[\left( \|\bu^+_{N}\|_{L^\infty(0,T;\bL^2(\sO))}+\|v^{+}_{N}\|_{L^\infty(0,T;L^2(0,L))}\right) ^{}\left( \int_0^T \|\nabla \bu^+_{N}\|_{\bL^{2}(\sO)}^2dt\right) ^{\frac12}\right]\\
	&\leq C(\delta)\left(  \bE \left( \|\bu^+_{N}\|^2_{L^\infty(0,T;\bL^2(\sO))}+\|v^{+}_{N}\|^2_{L^\infty(0,T;L^2(0,L))}\right) \bE\left( \int_0^T \|\nabla \bu^+_{N}\|_{\bL^{2}(\sO)}^2dt\right)\right) ^{\frac12}
	\\& \leq C(\delta).
	\end{align*}
	The treatment of the term $I_{2,5}$ is trivial.
	
	{Hence, by combining all the bounds derived so far we obtain the following estimate for the deterministic part of $\tilde\bU_N$:}
	\begin{equation}\begin{split}
	&\bP\left( \|I_1+\sum_{i=1}^5I_{2,i} +I_3\|_{H^{1}(0,T;\sU_1')} > M\right)\\ 
	&\qquad\quad\leq \bP\left( \|I_1\|_{H^{1}(0,T;\sU_1')}+\|\sum_{i=1}^5 I_{2,i}\|_{H^{1}(0,T;\sU_1')}+\|I_3\|_{H^{1}(0,T;\sU_1')} > M\right) \\
	&\qquad\quad\leq \frac1M\bE\left( \|I_1\|_{H^{1}(0,T;\sU_1')} + \sum_{i=1}^5\|I_{2,i}\|_{H^{1}(0,T;\sU_1')}+\|I_3\|_{H^{1}(0,T;\sU_1')} \right)\\
	&\qquad\quad\leq \frac{C(\delta,\ep)}{M} \label{dterms}.
	\end{split}		\end{equation}
	{	To treat the stochastic part of $\tilde\bU_N$ we estimate the stochastic term $I_{2,6}$ as follows. First we use the fact that for any $0<\alpha<\frac12$ and $\Phi$ any progressively measurable process in $L^p(\Omega;L^p(0,T;L_2(U_0;X)))$, we have} (see \cite{NTT22},\cite{FG95}),
	\begin{align}\label{FGBDG}
	\bE\|	\int_0^{\cdot} \Phi dW\|^q_{W^{\alpha,q}(0,T;X)} &\leq C\bE	\int_0^T\|\Phi\|^q_{L_2(U_0,X)} ds.
	\end{align}
	To deal with the $q^{th}$ power appearing in the inequality \eqref{FGBDG}, we define the following stopping times:
	$$\tau_M:=\inf_{t\geq 0}\{\sup_{s\in[0,t]}\|(\bu_{N},v_N)\|_{\bL^2}>M\} \wedge T.$$
	Observe that, since $(\bu_{N},v_N)$ has \cadlag sample paths in $\bL^2(\sO) \times L^2(0,L)$ by definition, and {since} the filtration $(\sF_t)_{t \geq 0}$ is right continuous by assumption, $\tau_M$ is indeed a stopping time. Using the Chebyshev inequality we then obtain that
	\begin{align*}
	&\bP\left(\|\int_0^{\cdot}%(R+\eta^*_{N})
	G(\bu_N,v_N,\eta^*_{N})dW\|_{W^{\alpha,q};(0,T;\sU')} >M \right) \\
	&\leq 	\bP\left(\|\int_0^{\cdot\wedge \tau_M}
	G(\bu_N,v_N,\eta^*_{N})dW\|_{W^{\alpha,q}(0,T;\sU_1')} >M, \tau_M = T\right) 
	+\bP\left(\tau_M < T\right)\\
	&\leq 	\bP\left(\|\int_0^{\cdot\wedge \tau_M}
	G(\bu_N,v_N,\eta^*_{N})dW\|_{W^{\alpha,q}(0,T;\sU_1')} >R, \tau_M = T\right) 
	+\bP\left(\sup_{t\in[0,T]}\|(\bu_N(t),v_N(t))\|_{\bL^2}\geq M\right)\\
	&\leq \frac{1}{M^q}	\bE\left(\|\int_0^{\cdot\wedge \tau_M}
	G(\bu_N,v_N,\eta^*_{N})dW\|^q_{W^{\alpha,q}(0,T;\sU_1')} \right) +\frac1{M^2}\bE\sup_{t\in[0,T]}\|(\bu_N(t),v_N(t))\|^2_{\bL^2}.
	\end{align*}
	Notice that $\eta^*_{N}$ and $(\bu_{N},v_N)$ are $(\sF_t)_{t\geq 0}$-adapted. 
	Hence using \eqref{FGBDG} and \eqref{growthG} we obtain the following upper bounds for the first term on the right hand side of the above inequality,
	\begin{align*}
	&\leq \frac{C}{ M^q} \bE\left( \int_0^{T\wedge \tau_M}\|G(\bu_{N},v_N,\eta^*_{N})\|_{L_2(U_0,\bL^2)}^q\right)\\
	&\leq \frac{C}{ M^q} \bE\left( \int_0^{T\wedge \tau_M}\|\eta^*_{N}\|_{L^\infty(0,L)}^q\|\bu_{N}\|^q_{\bL^2(\sO)}+\|v_{N}\|_{L^2(0,L)}^q\right)\\
	&\leq \frac{CTM^{q-2}}{\delta^q M^q}\bE[\sup_{s\in[0,T]}\|\bu_{N}\|^2_{\bL^2(\sO)}+\|v_{N}\|_{L^2(0,L)}^2].
	\end{align*}
	Hence for small enough $\delta>0$ we {have}
	\begin{align}
	\bP\left(\|\int_0^{\cdot}G(\bu_{N},v_N,\eta^*_{N})dW\|_{W^{\alpha,q};(0,T;\sU')} >M \right) \leq \frac{CT}{\delta^q M^2}\bE\sup_{s\in[0,T]}[\|\bu_{N}\|^2_{\bL^2(\sO)}+\|v_{N}\|_{L^2(0,L)}^2].
	\end{align}
	{Finally, the uniform bounds obtained in Lemma \ref{bounds} imply}
	\begin{equation}\begin{split}
	\bP\left(\|I_{2,6}\|_{W^{\alpha,q}(0,T;\sU_1')}> M\right) 
	\leq \frac{CT}{M^2\delta} \label{sterms}.
	\end{split}		
	\end{equation}
	To sum up, for some $C=C(\delta,T,\ep)>0$ independent of $N$ we have obtained that
	$$\mu^{u,v}_{N}(B_M)\geq 1-\frac{C}{M}.$$
	Now since $B_M$ is compactly embedded in $C([0,T];\sV_1')$ (see e.g. Theorem 2.2 \cite{FG95}) we conclude that $\{\mu^{u,v}_{N}\}_{N}$ is tight in $C([0,T];\sV'_1)$.
\end{proof}

{
	\begin{cor}\label{rem:weakconv} The following weak convergence results hold:
		\begin{enumerate}
			\item	The laws of the random variables $(\bu^+_{N}, v^+_{N})$
			converge weakly, up to a subsequence, to some probability measures on the spaces $L^2(0,T;\bL^2(\sO))\times L^2(0,T;L^2(0,L))$, $0\leq \alpha<1$
			respectively. 
			\item The laws of $\bu_{N},\tilde\bu_{N}$ converge weakly in $L^2(0,T;\bL^2(\sO))$, and the laws of $v_{N}, \tilde v_{N}$ converge weakly in $L^2(0,T;L^2(0,L))$ to the same probability measures as the weak limit of the laws of $\bu^+_N$ and $v^\#_N$ respectively. 
			\item The laws of structure variables $\eta_N$, $\eta_N^+$, and $\tilde\eta_N$ converge weakly to the same probability measure in $L^2(0,T;H^s(0,L))$ for any $0\le s<2$.
		\end{enumerate}
	\end{cor}
	The first statement follows from the Prohorov's theorem and Lemma \ref{tightuv}.
	Furthermore, by combining these convergence results with Lemma \ref{difference} and elementary tools from probability (see e.g. Theorem 3.1 \cite{Billingsley99}) we can see that the laws of $\bu_{N},\tilde\bu_{N}$ converge weakly in $L^2(0,T;\bL^2(\sO))$, and the laws of $v_{N}, \tilde v_{N}$ converge weakly in $L^2(0,T;L^2(0,L))$ to the same probability measures as the weak limit of the laws of $\bu^+_N$ and $v^+_N$ respectively, which is the second statement of the corollary. 
	The same argument implies that the laws of structure variables converge weakly to the same probability measure in $L^2(0,T;H^s(0,L))$ for any $0\le s<2$.
	
	To recover the weak solution in the limit of these subsequences, 
	we need to upgrade the weak convergence results above to  almost sure convergence. This is done next.
}
\subsection{Almost sure convergence}\label{almostsure}
Let 
$\mu_{N}$ be the joint law of the random variable\\
$\mathcal{U}_{N}:=(\bu_{N},\bu^+_{N},v_{N},v^{\#}_{N},\eta_{N}, \eta^*_{N},\eta^+_{N},\tilde\eta^*_{N},\tilde\eta_{N},\|\bu^+_N\|_{L^2(0,T;V)},\|v^*_{N}\|_{L^2(0,T;L^2(0,L))},
\tilde\bU_N,
W)$ taking values in the phase space
\begin{align*}
\Upsilon&:=[L^2(0,T;\bL^2(\sO))]^2\times[L^2(0,T;L^2(0,L))]^3
\\ &\times[L^2(0,T;H^s(0,L))]^3\times [C([0,T],H^s(0,L))]^2\times\R^2 \\
&\times C([0,T];\sV_1')\cap L^2(0,T;\bL^2(\sO)\times L^2(0,L))\times C([0,T];U),
\end{align*}
for some fixed $\frac32<s<2$.

First observe that, since $C([0,T];U)$ is separable and metrizable by a complete metric,  the sequence of Borel probability measures, $\mu^W_{N}(\cdot):=\bP(W\in\cdot)$, that are constantly equal to one element, is tight on $C([0,T];U)$.

Next, recalling Lemmas \ref{tightuv}, \ref{tightl2}, \ref{uvtightC}, \ref{u*diff} and Remark \ref{rem:weakconv}, and using Tychonoff's theorem it follows that the sequence of the joint probability measures $\mu_{N}$ of the approximating sequence $\mathcal{U}_{N}$ is tight on the Polish space $\Upsilon$. 

{To state the next result we  will be using 
	the notation ``=d'' to denote random variables that are ``equal in
	distribution'' i.e., the random variables have the same laws as random variables taking values on the same
	given phase space  $\Upsilon$.}
\begin{theorem}\label{skorohod} There exists a probability space $(\bar\Omega,\bar\sF,\bar\bP)$ and random variables $\bar{\mathcal{U}}_{N}:=(
	\bar{\bu}_{N},\bar{\bu}^+_{N},\bar v_{N},\bar v^{\#}_{N},\bar\eta_{N}, \bar\eta^*_{N},\bar\eta^+_{N},\bar{\tilde\eta}^*_{N},\bar{\tilde\eta}_{N},m_N, k_N,\bar{\tilde\bU}_N,\bar{W}_N)$ and\\
	$\bar{\mathcal{U}}:=(
	\bar{\bu},\bar{\bu}^+,\bar v,{\bar v}^{\#},\bar\eta, \bar\eta^*,\bar\eta^+,\bar{\tilde\eta}^*,\bar{\tilde\eta},m,k,\bar{\tilde\bU},\bar{W})$
	defined on this new probability space, such that
	\begin{enumerate}
		\item $\bar{\mathcal{U}}_{N}=^d\mathcal{U}_{N}$.
		\item $\bar{\mathcal{U}}_{N} \rightarrow \bar{\mathcal{U}}$ $\bar\bP$-a.s. in the topology of $\Upsilon$.
		\item $\bar\bu=\bar{\bu}^+ ,\quad \bar\eta=\bar{\eta}^+=\bar{\tilde\eta},\quad \bar{\eta}^*=\bar{\tilde\eta}^*,\quad \bar v=\bar v^{\#}$ and $\bar{\tilde\bU}=((R+\bar\eta^*)\bar\bu,\bar v) $.
		\item $\partial_t\bar\eta=\bar v$ in the sense of distributions, almost surely.
	\end{enumerate}
\end{theorem}
\begin{proof}
	The existence of the probability space {and the existence of} the new random variables, {as well as the} first two statements {of the theorem} follow from an application of the Skorohod representation theorem. In fact, thanks to Theorem 1.10.4 in \cite{VW96}, the probability space and the random variables $\bar{\mathcal{U}}_N$ can be chosen such that for every $N$, \begin{align}\label{newrv}
	\bar{\mathcal{U}}_N(\bar\omega)={\mathcal{U}}_N(\phi_N(\bar\omega)) \quad \text{ for }\bar\omega\in\bar\Omega,
	\end{align} and $\bar \bP\circ\phi_N^{-1}=\bP$, where $\phi_N:\bar\Omega\rightarrow\Omega$ is measurable.

	Thanks to these explicit maps we can identify the real-valued random variables $m_N$ as ${m}_N= \|\bar \bu^+_N\|_{L^2(0,T;V)}$. Hence almost sure convergence of ${m}_N$ implies that $\|\bar \bu^+_N\|_{L^2(0,T;V)}$ is bounded a.s. {and this boundedness will be used in passing $N\to \infty$. }
	We thus have, up to a subsequence, 	\begin{align}\label{uweak}	\bar \bu^+_N \rightharpoonup \bar \bu \quad\text{ weakly in } {L^2(0,T;V)} \quad \bar\bP-a.s.	\end{align}
	Similarly, there exists $\bar v^*$ such that,
	\begin{align}\label{vweak}	\bar v^*_N \rightharpoonup \bar v^* \quad\text{ weakly in } {L^2(0,T;L^2(0,L))} \quad \bar\bP-a.s.	\end{align}
	
	The third statement then follows from part (1), Lemmas \ref{difference}, \ref{u*diff} and an application of the Borel-Cantelli lemma. Hence we have
	\begin{align}\label{uvcont1}((R+\bar\eta^*)\bar\bu,\bar v) \in C([0,T];\sV'_1) \quad \bar\bP-a.s.\end{align}
	To prove the fourth statement we begin with the fact that
	$$\int_0^T\int_0^L \partial_t\tilde\eta_{N}\phi dzdt=\int_0^T\int_0^L { v^{\#}_{N}\phi} dzdt,\qquad \forall \phi\in C_0^1(0,T;L^2(0,L)),$$
	and thus integration by parts yields,
	$$
	-\int_0^T\int_0^L \tilde\eta_{N}\partial_t\phi dzdt=\int_0^T\int_0^L { v^{\#}_{N}\phi}dzdt,\quad \forall \phi\in C_0^1(0,T;L^2(0,L)).$$
	Statement (1) (see also \eqref{newrv}) then implies that
	$$
	-\int_0^T\int_0^L \bar{\tilde\eta}_{N}\partial_t\phi dzdt=\int_0^T\int_0^L {\bar{ v}^{\#}_{N}\phi}dzdt,\quad \forall \phi\in C_0^1(0,T;L^2(0,L)),$$
	and thus passing $N \rightarrow \infty$ and using statement (3) we come to the desired conclusion. The same argument along with \eqref{vweak} also gives us
	that,
	$$\partial_t\bar\eta^*=\bar v^*,\quad \text{ almost surely.}$$
\end{proof}

{Next we will construct a complete, right-continuous filtration $(\bar{\mathcal{F}}_t)_{t \geq 0}$ on the new probability space $(\bar \Omega,\bar{\mathcal{F}},\bar \bP)$ given in Theorem \ref{skorohod}, to which the noise processes and the candidate solutions are adapted. This filtration will be used to define the new stochastic basis that appears in the definition of martingale solutions given in Definition \ref{def:martingale}. With this in mind we denote by $\sF_t'$ the $\sigma$-field generated by the random variables $(\bar{\bu}(s),\bar v(s)),\bar\eta(s),\bar{W}(s)$ for all $s \leq t$. Then we define
	\begin{align}\label{Ft}
	\mathcal{N} &:=\{\mathcal{A}\in \bar{\mathcal{F}} \ |  \bar \bP(\mathcal{A})=0\},\qquad
	\bar{\mathcal{F}}^0_t:=\sigma(\mathcal{F}_t' \cup \mathcal{N}),\qquad
	\bar{\mathcal{F}}_t :=\bigcap_{s\ge t}\bar{\mathcal{F}}^0_s.
	\end{align}
	We note here that \eqref{uvcont1} gives us stochastic processes $((R+\bar\eta^*)\bar\bu,\bar v)$ that are $(\bar\sF_t)_{t\geq 0}-$progressively measurable and thus helps in identifying the limit of the stochastic integral as we pass {to the limit as} $N \rightarrow\infty$ (see also Lemma \ref{conv_G} where we pass {to the limit as} $\ep\rightarrow 0$). 
	
	For each $N$ we also define a filtration $\left(\bar{\mathcal{F}}^{N}_t \right)_{t \geq 0}$ on  $(\bar \Omega,\bar{\mathcal{F}},\bar \bP)$ the same way as above but using the processes $(\bar{\bu}_{N}(s),\bar{v}_{N}(s)),\bar\eta_N(s),\bar W_N(s)$  instead. Because of \eqref{newrv}, we can see that $\bar{\bu}_{N},\bar{\eta}_{N},\bar{v}_{N}$ are \cadlag and that the aforementioned pointwise definition of the filtration $(\bar\sF^N_t)_{t\geq 0}$ makes sense.  Moreover, using usual arguments we can see that $\bar W_N$ is an $\bar\sF^N_t$-Wiener process (see e.g. \cite{BFH18}).
	
	Next, relative to the new stochastic basis $(\bar\Omega,\bar\sF,(\bar\sF^{N}_t)_{t\geq0},\bar\bP,\bar W_N)$, 
	thanks to \eqref{newrv} we can see that for each $N$, the following equation holds $\bar\bP$-a.s. for every $t \in [0,T]$ and any $\bQ\in \sD$:
	\begin{equation}\begin{split}\label{approxsystem}
	&(\bar{\tilde\bU}_N,\bQ)=
	(\bu_0(R+\eta_0),\bq)+( v_0,\psi)-\int_0^t\int_0^L(\partial_z\bar\eta^+_{N}\partial_z\psi+\partial_{zz}\bar\eta^+_{N}\partial_{zz}\psi)\\
	&-\frac{1}2\int_0^t\int_\sO 
	\bar{v}^*_{N} \bar\bu^{+}_{N}\cdot \bq  +\int_0^t\int_{\sO}
	\bar v^*_{N} (2\bar{\tilde\bu}_{N}-\bar{\bu}_{N})\cdot\bq \\
	&-\frac12\int_0^t\int_{\sO}(R+\bar\eta_{N}^*)((\bar\bu^+_{N}-\bar v^{+}_{N} 
	r\be_r)\cdot\nabla^{\bar\eta_{N}^*}\bar\bu^{+}_{N}\cdot\bq
	- (\bar\bu^{+}_{N}-
	\bar{v}^{+}_{N} 
	r\be_r)\cdot\nabla^{\bar\eta_{N}^*}\bq\cdot\bar\bu^{+}_{N})\\
	&-2\nu\int_0^t\int_{\sO}(R+\bar{\eta}_{N}^*)\bD^{\bar\eta_{N}^*}(\bar\bu^{+}_{N})\cdot \bD^{\bar\eta_{N}^*}(\bq) dxds -\frac1{\ep}\int_0^t \int_\sO   \text{div}^{\bar\eta^*_{N}}\bar\bu^{+}_{N}\text{div}^{\bar\eta^*_{N}}\bq \\
	&+\int_0^t\left( P_{{in}}\int_{0}^1q_z\Big|_{z=0}-P_{{out}}\int_{0}^1q_z\Big|_{z=1}\right)
	+\int_0^t 
	(G({\bar\bu}_{N},\bar v_N,\bar\eta^*_{N})d\bar W_N, \bQ).
	\end{split}
	\end{equation}

	Observe also that the bounds obtained in Lemma \ref{bounds} hold for the new random variables $\bar{\mathcal{U}}_{N}$ as well. 
		As a result, weak convergence results  up to a subsequence (due to the bounds in Lemma \ref{bounds}) give us the following: 
	\begin{equation}
	\begin{split}\label{sols1}
	&\bar \bu \in L^2(\bar\Omega;L^\infty(0,T;\bL^2(\sO)))\cap L^2(\Omega;L^2(0,T;V)),
	\\	&	\bar v,\bar v^*\in L^2(\bar\Omega;L^\infty(0,T;L^2(0,L))), 
	\\	&\bar {\eta},	\bar \eta^* \in L^2(\bar\Omega;L^\infty(0,T;H^2_0(0,L)))\cap L^2(\bar\Omega;W^{1,\infty}(0,T;L^2(0,L))),
	\end{split}	\end{equation}
	{where $\bar \bu, \bar v,\bar v^*, \bar {\eta}$, and $\bar \eta^*$ are the limits defined in Theorem~\ref{skorohod}.
		We also have the following upgraded convergence results for the displacements. Namely,}
	{notice that} Theorem \ref{skorohod} statement (2) implies that for given $\frac32<s<2$ (see \cite{MC13} Lemma 3),
	\begin{align}\label{etaunif1}
	\bar\eta_{N} \rightarrow \bar\eta \text{ and } \bar\eta^*_{N} \rightarrow \bar\eta^* \,\,\,\text { in } L^\infty(0,T;H^s(0,L))\, a.s.
	\end{align}
	and thus {the following uniform convergence result holds}
	\begin{align}\label{etaunif}
	\bar\eta_{N} \rightarrow \bar\eta \text{ and } \bar\eta^*_{N} \rightarrow \bar\eta^*\,\,\, \text { in } L^\infty(0,T;C^1[0,L])\, a.s.
	\end{align}
	
	{Combining these regularity and convergence results we can now pass to the limit 
		as $N\to\infty$ in the weak formulations of the approximate problems \eqref{approxsystem} stated on the fixed reference domain $\sO$, defined
		on the probability space $(\bar\Omega,\bar\sF,\bar\bP)$, while keeping $\ep>0$ fixed.
		We consider the random variables $\bar{\mathcal{U}}_{N}$ defined on the probability space $(\bar\Omega,\bar\sF,\bar\bP)$ and pass to
		the limit $N \rightarrow \infty$ in \eqref{approxsystem} and show that the limit satisfies the { weak formulation \eqref{second}} with divergence-free penalty.
		Except for the divergence-free penalty term,
		this is the same weak formulation as stated in the definition of martingale solutions Definition~\ref{def:martingale}. 
		More precisely, the following theorem holds true.
	}
	\begin{theorem}[Existence for the problem with penalized compressibility]\label{exist1} 
		For the stochastic basis $(\bar\Omega,\bar\sF,(\bar\sF_t)_{t \geq 0},\bar\bP,\bar W)$ as found in Theorem \ref{skorohod}, given any fixed $\ep,\delta>0$,  {the processes $(\bar\bu,\bar\eta,\bar\eta^*)$ obtained in Theorem \ref{skorohod} (see \eqref{sols1}) are such that} $\bar\eta^*$ and $((R+\bar\eta^*)\bar\bu,\partial_t\bar\eta)$ are $(\bar\sF_t)_{t \geq 0}$-progressively measurable with $\bar\bP$-a.s. continuous paths in $H^s(0,L)$, $\frac32<s<2$ and $\sV_1'$ respectively {and such that the following weak formulation 
			holds $\bar\bP$-a.s. for every $t\in[0,T]$ and for every $\bQ\in \sD$}:
		\begin{equation}\begin{split}\label{martingale1}
		&(	(R+\bar{\eta} ^*(t))\bar\bu (t),\bq)+(\partial_t\bar{ \eta}(t),\psi)
		=(\bu_0(R+\eta_0),\bq)+( v_0,\psi)\\
		&-\int_0^t\int_0^L\partial_z\bar\eta \partial_z\psi+\partial_{zz}\bar\eta \partial_{zz}\psi 
		+\frac12\int_0^t\int_{\sO}\partial_s\bar {\eta} ^*\bar{\bu} \cdot\bq \\
		&-\frac12\int_0^t\int_{\sO}(R+\bar\eta ^*)((\bar\bu-\partial_t\bar {\eta}
		r\be_r)\cdot\nabla^{\bar\eta ^*}\bar\bu \cdot\bq
		- (\bar\bu^{}-\partial_t\bar {\eta} 
		r\be_r)\cdot\nabla^{\bar\eta ^*}\bq\cdot\bar\bu )\\
		&-2\nu\int_0^t\int_{\sO}(R+\bar{\eta} ^*) \bD^{\bar\eta ^*}(\bar\bu )\cdot \bD^{\bar\eta ^*}(\bq)  +\frac1{\ep}\int_0^t \int_\sO   \text{div}^{\bar\eta^*}\bar\bu\text{ div}^{\bar\eta^*}\bq\\
		&+ 
		\int_0^t\left( P_{{in}}\int_{0}^1q_z\Big|_{z=0}-P_{{out}}\int_{0}^1q_z\Big|_{z=1}\right)
		+\int_0^t
		(	G(\bar\bU ,\bar\eta^*)d\bar W, \bQ).
		\end{split}
		\end{equation}
	\end{theorem}
	The proof of the following result is similar to and slightly simpler than the proof of Theorem \ref{exist2} and thus we refer the reader to the proof of Theorem \ref{exist2} for the details.
	
	{Before we proceed and consider the limit as the divergence-free penalty parameter $\ep \to 0$, we argue that 
		\begin{align}\label{etasequal1}
		\bar\eta^*(t)=\bar\eta(t) \quad \text{ for any } t<\tau, \quad \bar\bP-a.s.
		\end{align}	
		{ where $$	\tau  := T\wedge\inf\{t>0:\inf_{z\in[0,L]}(R+ \bar\eta(t,z))\leq \delta \text{ or } \|R+\bar\eta(t)\|_{H^s(0,L)}\geq \frac1{\delta} \}.$$}
		Indeed, to show that this is true, let us introduce the following stopping times.
		For $\frac32<s<2$ we define
		\begin{align*}
		\tau_N &:=T\wedge \inf\{t> 0:\inf_{z\in[0,L]}(R+ \bar\eta_N(t,z))\leq \delta \text{ or } \|R+\bar\eta_N(t)\|_{H^s(0,L)}\geq \frac1{\delta}\}.
		\end{align*}
	Then thanks to  \eqref{etaunif1}, $\tau \leq \liminf_{N\rightarrow\infty}\tau_N$ a.s. 
		Observe further that for almost any $\omega\in\bar\Omega$ and $t<\tau$, and for any $\epsilon>0$, there exists an $N$ such that
		\begin{align*}
		\|\bar\eta(t)-\bar\eta^*(t)\|_{H^s(0,L)}&<\|\bar\eta(t)-\bar\eta_N(t)\|_{H^s(0,L)}+\|\bar\eta^*_N(t)-\bar\eta_N(t)\|_{H^s(0,L)}+\|\bar\eta^*(t)-\bar\eta^*_N(t)\|_{H^s(0,L)}\\
		&<\epsilon.
		\end{align*}
		This is true because for any $\epsilon>0$ there exists an $N_1 \in \mathbb{N}$ such that the first and the last terms on the right side of the above inequality
		are each bounded by $\frac{\epsilon}2$ for any $N\geq N_1$ thanks to the uniform convergence \eqref{etaunif}.
		Furthermore, since $t<\tau_N$, for infinitely many $N$'s, the second term is equal to 0. Hence we conclude that indeed
		\begin{align*}
		\bar\eta^*(t)=\bar\eta(t) \quad \text{ for any } t<\tau, \quad \bar\bP-a.s.
		\end{align*}
	}
	\section{{Passing to the limit as $\ep \rightarrow 0$}}\label{sec:limit2}
	In this section, to emphasize the dependence on the parameter $\ep>0$, we will use the notation $(\bar \bu_\ep,\bar v_\ep,\bar v^*_\ep, \bar \eta_\ep,\bar\eta^*_\ep,\bar W_\ep)$ and $(\Omega^\ep,\sF^\ep,(\sF^\ep_t)_{t\geq 0},\bP^\ep)$ to denote the solution and the filtered probability space found in the previous section.  
	In what follows, we will pass $\ep \rightarrow 0$ in \eqref{martingale1} with appropriate test functions. Most of the results in the first half of this section can be proven as in the previous section. Hence we only summarize the important theorems without proof here.
	
	First observe that, thanks to the weak lower-semicontinuity of norm, the uniform estimates obtained in the previous section still hold. That is, as a consequence of Lemmas \ref{bounds} and Theorem \ref{skorohod}, we have the following {uniform boundedness} result. 
	\begin{lem}[Uniform boundedness]\label{boundsep}
		For a fixed $\delta>0$ we have for some $C>0$ {\bf independent of} $\ep$ that
		\begin{enumerate}
			\item $\bE^\ep\|\bar\bu_\ep\|^2_{L^\infty(0,T;\bL^2(\sO))\cap L^2(0,T;V)}<C.$
			\item $\bE^\ep\|\bar v_\ep\|^2_{L^\infty(0,T;L^2(0,L))\cap L^2(0,T;H^\frac12(0,L))}<C$.
			\item $\bE^\ep\|\bar v^*_\ep\|^2_{L^\infty(0,T;L^2(0,L))}<C$.
			\item $\bE^\ep\|\bar \eta_\ep\|^2_{L^\infty(0,T;H^2_0(0,L)\,\cap\, W^{1,\infty}(0,T;L^2(0,L)))}<C.$
			\item $\bE^\ep\|\bar \eta^*_\ep\|^2_{L^\infty(0,T;H^2_0(0,L)\,\cap\, W^{1,\infty}(0,T;L^2(0,L)))}<C.$
			\item $\bE^\ep\|\text{div}^{\bar\eta^*_\ep}\bar\bu_\ep\|^2_{L^2(0,T;L^2(\sO))} < C{\ep}$.
		\end{enumerate}
	\end{lem}
	Recall that the bounds obtained in the proofs of Lemmas \ref{tightuv} and \ref{tightl2} were independent of $\ep$. Hence the same results still hold true and we obtain the following lemma.
	\begin{lem}[Tightness of the laws]\label{tightep}
		\begin{enumerate}
			\item The laws of $\bar\bu_\ep$ and $\bar v_\ep$ are tight in $L^2(0,T;\bH^\alpha(\sO))$ for any $0\leq \alpha<1$ and $ L^2(0,T;L^2(0,L))$ respectively.
			\item The laws of $\bar \eta_\ep$ and that of $\bar \eta^*_\ep$ are tight in $C([0,T];H^s(0,L))$ for $\frac32<s<2$.
			\item The laws of $\|\bar \bu_\ep\|_{L^2(0,T;V)}$ are tight in $\mathbb{R}$.
			\item The laws of $\|\bar v^*_\ep\|_{L^2(0,T;L^2(0,L))}$ are tight in $\mathbb{R}$.
		\end{enumerate}
	\end{lem}

	Now for an infinite denumerable set of indices $\Lambda$, we let 
	$\mu_{\ep}$ be the law of the random variable $\bar{\mathcal{U}}_{\ep}:=(
	\bar{\bu}_{\ep},\bar v_{\ep},\bar\eta_{\ep}, \bar\eta^*_{\ep}, 
	\|\bar \bu_\ep\|_{L^2(0,T;V)},\|\bar v^{*}_{\ep}\|_{L^2(0,T;L^2(0,L))},
	\bar{W}_\ep)$
	taking values in the phase space
	\begin{align*}
	\sS:=L^2(0,T;\bH^{\frac34}(\sO))\times L^2(0,T; L^2(0,L)) \times[ C([0,T],H^s(0,L))]^2\times \mathbb{R}^2\times
	C([0,T];U),
	\end{align*}
	for $\frac32<s<2$.
	
	Then tightness of $\mu_\ep$ on $\sS$ and an application of the Prohorov theorem and the Skorohod representation theorem gives us the following almost sure representation and convergence.
	\begin{theorem}[Almost sure convergence in $\ep$]\label{skorohod2} There exists a probability space $(\hat\Omega,\hat\sF,\hat\bP)$ and random variables $\hat{\mathcal{U}}_{\ep}=(
		\hat{\bu}_{\ep},\hat v_{\ep},\hat\eta_{\ep}, \hat\eta^*_{\ep}, m_\ep,k_\ep,
		\hat{W}_\ep)$ and
		$\hat{\mathcal{U}}=(
		\hat{\bu},\hat v,\hat\eta, \hat\eta^*,m,k,
		\hat{W})$
		such that
		\begin{enumerate}
			\item $\hat{\mathcal{U}}_{\ep}=^d\bar{\mathcal{U}}_{\ep}$ for every $\ep \in \Lambda$.
			\item $\hat{\mathcal{U}}_{\ep} \rightarrow \hat{\mathcal{U}}$ $\hat\bP$-a.s. in the topology of $\sS$ as $\ep\rightarrow 0$.
			\item $\partial_t\hat\eta=\hat v$ and $\partial_t\hat\eta^*=\hat v^*$, in the sense of distributions, almost surely.
		\end{enumerate}
	\end{theorem}
	
	To pass to the limit as $\ep \to 0$ we will need stronger convergence of the fluid velocity random variables to compensate for the fact that our construction of		test functions presented below in \eqref{q_ep} does not lead to uniform convergence of the test functions as $\ep\to 0$. See \eqref{convq}-\eqref{convqt}.		For this purpose we start by recalling 
	 again that Theorem 1.10.4 in \cite{VW96} implies that the random variables $\hat{\mathcal{U}}_\ep$ can be chosen such that for every $\ep \in \Lambda$,
	\begin{align}\label{newrv2}
	\hat{\mathcal{U}}_\ep(\omega)=\bar{\mathcal{U}}_\ep(\phi_\ep(\omega)), \quad \omega \in \hat\Omega,
	\end{align} and $\hat\bP\circ\phi_\ep^{-1}=\bP^\ep$, where $\phi_\ep:\hat\Omega\rightarrow \Omega^\ep$ is measurable. 
	
	Thanks to these explicit maps we identify the real-valued random variables $m_\ep$ as $m_\ep= \|\hat \bu_\ep\|_{L^2(0,T;V)}$ and notice that $m_\ep$ converge almost surely due to  Theorem~\ref{skorohod2}.}
The almost sure convergence of $m_\ep$ implies that $\|\hat \bu_\ep\|_{L^2(0,T;V)}$ is bounded a.s. and thus also that, up to a subsequence,
\begin{align}\label{uweak2}
\hat \bu_\ep \rightharpoonup \hat\bu \quad\text{ weakly in } {L^2(0,T;V)} \quad \hat\bP-a.s.
\end{align}
Similarly,
\begin{align}\label{vweak2}
	\hat v^*_\ep \rightharpoonup \hat v^* \quad\text{ weakly in } {L^2(0,T;L^2(0,L))} \quad \hat\bP-a.s.
\end{align}

Observe also that, since $\sO$ is a Lipschitz domain, the space ${\bH^\gamma(\sO)}$ is continuously embedded in ${\bL^q(\sO)}$  for any $q\in[1,\frac2{1-\gamma}]$ and  $\gamma<1$ and thus statement (2) implies that
\begin{align}\label{ul4}
\hat\bu_\ep \rightarrow \hat\bu \quad \text{in } L^2(0,T;\bL^4(\sO)), \quad \hat\bP-a.s.
\end{align}
\iffalse
{\cred Similarly, since $H^{\frac34}(0,L) \hookrightarrow C^{0,\frac14}(0,L)$ , we obtain
\begin{align}\label{vlq}
	\hat v^*_\ep \rightarrow \hat v^* \quad\text{ in }\, L^2(0,T;L^\infty(0,L)), \quad \hat\bP-a.s.
	\end{align}

\fi
As mentioned earlier the convergence results  \eqref{uweak2} and \eqref{ul4} will be used in passing to the limit as $\ep \to 0$ in the weak formulation
in conjunction with the convergence of the test functions given below in  \eqref{convq}-\eqref{convqt}.
{Next, we discuss additional regularity and convergence properties for the random variables discussed in Theorem~\ref{skorohod2},
	which we will use to pass to the limit as $\ep\to 0$.}

{First, we notice that equivalence of laws} implies that for some $C>0$ independent of $\ep$,
\begin{equation}\label{uniformep}
\begin{split}
&\hat\bE\|\hat\bu_\ep\|^2_{L^\infty(0,T;\bL^2(\sO))\cap L^2(0,T;V)}<C.\\
&	\hat\bE\|\hat v_\ep\|^2_{L^\infty(0,T;L^2(0,L))}<C,\quad \hat\bE\|\hat v^*_\ep\|^2_{L^\infty(0,T;L^2(0,L))}<C.\\
&\hat\bE\|\hat \eta_\ep\|^2_{L^\infty(0,T;H^2_0(0,L))\,\cap\, W^{1,\infty}(0,T;L^2(0,L))}<C,\quad \hat\bE\|\hat \eta^*_\ep\|^2_{L^\infty(0,T;H^2_0(0,L))\,\cap\, W^{1,\infty}(0,T;L^2(0,L))}<C.\\
&\hat\bE\|\text{div}^{\hat\eta^*_\ep}\hat\bu_\ep\|^2_{L^2(0,T;L^2(\sO))} < C{\ep}.
\end{split}\end{equation}
{Furthermore, for the limiting functions we have}
\begin{align}\label{sols2}
&\hat \bu \in L^2(\hat\Omega;L^\infty(0,T;\bL^2(\sO)))\cap L^2(\hat\Omega;L^2(0,T;V)),\notag\\
&	\hat v, \hat v^* \in L^2(\hat\Omega;L^\infty(0,T;L^2(0,L))),\\
&\hat {\eta},	\hat \eta^* \in L^2(\hat\Omega;L^\infty(0,T;H^2_0(0,L)))\cap L^2(\hat\Omega;W^{1,\infty}(0,T;L^2(0,L))).\notag
	\end{align}
{Now, notice that \eqref{uniformep}$_4$} implies that, up to a subsequence, we have
\begin{align}\label{div0}
\text{div}^{\hat\eta^*_\ep}\hat\bu_\ep \rightarrow 0 \quad \text{ in } L^2(0,T;L^2(\sO)), \quad \hat\bP- a.s.
\end{align}
Furthermore, as shown in the previous section, we also have that
 \begin{align}\label{etaunif2}
&\hat\eta_{\ep} \rightarrow \hat\eta \text{ and } \hat\eta^*_{\ep} \rightarrow \hat\eta^* \,\,\,\text { in }  L^\infty(0,T;H^s(0,L)) \text{ and } L^\infty(0,T;C^1([0,L]))\,\quad\hat\bP-a.s.
\end{align}
and, for given $\frac32<s<2$ {we have that}
\begin{align}\label{etastarbound}
\inf_{t\in[0,T],z\in[0,L]}(R+ \hat\eta^*(t,z))\geq \delta \,\,\,\text{ and }\,\,\, \|R+\hat\eta^*\|_{L^\infty(0,T;H^s(0,L))}\leq \frac1{\delta},\quad \hat\bP- a.s.
\end{align}
Finally, we note using \eqref{newrv2} that $(\hat\bu_\ep,\hat v_\ep,\hat v_\ep^*,\hat\eta_\ep,\hat\eta^*_\ep,\hat W_\ep)$ solve 
$\hat\bP-$a.s.  the following weak formulation:
\begin{equation}\begin{split}\label{approxsystem2}
&(	(R+\hat{\eta}_\ep ^*(t))\hat\bu_\ep (t),\bq)+(\partial_t\hat{ \eta}_\ep(t),\psi)
=(\bu_0(R+\eta_0),\bq)+( v_0,\psi)\\
&-\int_0^t\int_0^L\partial_z\hat\eta_\ep \partial_z\psi+\partial_{zz}\hat\eta_\ep \partial_{zz}\psi 
+\frac12\int_0^t\int_{\sO}\partial_s\hat{\eta}_\ep ^*\hat{\bu}_\ep \cdot\bq \\
&-\frac12\int_0^t\int_{\sO}(R+\hat\eta_\ep ^*)((\hat\bu_\ep-\partial_t\hat {\eta}_\ep
r\be_r)\cdot\nabla^{\hat\eta ^*_\ep}\hat\bu_\ep \cdot\bq
- (\hat\bu_\ep-\partial_t\hat {\eta} _\ep
r\be_r)\cdot\nabla^{\hat\eta_\ep ^*}\bq\cdot\hat\bu_\ep )\\
&-2\nu\int_0^t\int_{\sO}(R+\hat{\eta}_\ep ^*) \bD^{\hat\eta_\ep ^*}(\hat\bu_\ep )\cdot \bD^{\hat\eta_\ep ^*}(\bq)  +\frac1{\ep}\int_0^t \int_\sO   \text{div}^{\hat\eta_\ep^*}\hat\bu_\ep\text{ div}^{\hat\eta_\ep^*}\bq\\
&+ 
\int_0^t\left( P_{{in}}\int_{0}^1q_z\Big|_{z=0}-P_{{out}}\int_{0}^1q_z\Big|_{z=1}\right)
+\int_0^t
(	G(\hat\bu_\ep,\hat v_\ep ,\hat\eta_\ep^*)d\hat W_\ep, \bQ),
\end{split}
\end{equation}
for every $t\in[0,T]$ and for every $\bQ=(\bq,\psi)\in\sD$.

{Before we can pass to the limit as $\ep\to 0$ in \eqref{approxsystem2}, 
	we have one more obstacle to deal with. Namely,}
observe that the candidate solution for fluid and structure velocities, $\hat\bU=(\hat\bu,\hat v)$, is not regular enough to be a stochastic process in the classical sense as it only belongs to the space $L^2(0,T;\bL^2(\sO))\times L^2(0,T;L^2(0,L))$ (note that the equivalent of Lemma \ref{uvtightC} does not apply). Hence  we need to be
careful in our construction of an appropriate filtration. For this purpose we construct an appropriate filtration as follows. 
	Define the $\sigma-$fields
$$
 \sigma_t(\hat\bU):=\bigcap_{s\geq t}\sigma\left(\bigcup_{\bQ\in C^\infty_0((0,s);\sD)}\{(\hat\bU,\bQ)<1\}\cup \mathcal{N} \right),
$$
 where $N=\{\mathcal{A}\in \bar{\mathcal{F}} \ |  \hat\bP(\mathcal{A})=0\}.$\\
Let $\hat\sF_t'$ be the $\sigma-$ field generated by the random variables $\hat\eta(s),\hat{W}(s)$ for all $0\leq s \leq t$.
Then we define
\begin{align}\label{Ft1}
\hat{\mathcal{F}}^0_t:=\bigcap_{s\ge t}\sigma(\hat{\sF}_s' \cup \mathcal{N}),\qquad
\hat{\mathcal{F}}_t :=\sigma(\sigma_t(\hat\bU)\cup \hat{\mathcal{F}}^0_t).
\end{align}
This gives a complete, right-continuous filtration $(\hat{\mathcal{F}}_t)_{t \geq 0}$, on the probability space $(\hat\Omega,\hat{\mathcal{F}},\hat\bP)$, to which the noise processes and candidate solutions are adapted. Construction of filtration $(\hat\sF^\ep_t)_{t\geq0}$ is postponed to later; see \eqref{Ftep}. Now we state the following result from \cite{BFH18}.
\begin{lem}\label{representative}
	There exists a stochastic process in $
	L^2(0,T;\bL^2(\sO))\times L^2(0,T; L^2(0,L))$ a.s. which is a $(\hat\sF_t)_{t\geq 0}$-progressively measurable representative of $\hat\bU=(\hat\bu, \hat v)$.
\end{lem}
\begin{remark}\label{rem:integral} 
	Observe that thanks to Lemma \ref{representative} and continuity of the coefficient $G$, we can deduce that $G(\hat{\bu}(s),\hat v(s),\hat\eta^*(s))$ is $\{\hat{\sF}_t\}_{t\geq 0}-$progressively measurable. This measurability property along with the growth assumptions \eqref{growthG}  
	implies that $\int_0^t(G(\hat{\bu}(s),\hat v(s),\hat\eta^*(s))d\hat{W}(s),\bQ(s))$ is a well-defined stochastic integral for any $(\hat\sF_t)_{t \geq 0}$-adapted continuous $\sD$-valued process $\bQ$ (see e.g. \cite{DaPrato} or Lemma 2.4.2 in \cite{Prevot}),
	where $\sD$ is defined in \eqref{Dspace}.
\end{remark}
{We are now ready to state the main result of this section, which shows that the limiting stochastic processes $(\hat\bu_\ep,\hat\eta_\ep,\hat W_\ep)$ for each fixed $\delta > 0$, converge to a martingale solution of our stochastic FSI problem, as $\ep\to 0$.}

{ The following theorem is a key intermediate step in establishing the final existence result as it shows that the processes $(\hat\bu, \hat\eta,\hat\eta^*)$, obtained in Theorem \ref{skorohod2}, solve the weak formulation \eqref{weaksol} for almost every time $t\in [0,T]$, except with $\hat\eta$ replaced by the artificial structure random variable $\hat\eta^*$ in {the appropriate} terms. This implies that we obtain the desired martingale solution in the sense of Definition \ref{def:martingale} for as long as $\hat\eta$ and $\hat\eta^*$ are equal. }
\begin{theorem}[Main convergence result as $\ep \rightarrow 0$]\label{exist2} 
	For any fixed $\delta>0$, {the stochastic processes $(\hat \bu,\hat \eta,\hat\eta^*)$ constructed in Theorem~\ref{skorohod2}}
	satisfy
	\begin{equation}\begin{split}\label{martingale2}
	&(	(R+\hat{\eta} ^*(t))\hat\bu (t),\bq(t))+(\partial_t\hat{ \eta}(t),\psi(t))
	=(\bu_0(R+\eta_0),\bq(0))+( v_0,\psi(0))\\
	&+\int_0^t\int_\sO(R+{\hat\eta}^*)\hat\bu\cdot \partial_t{\bq} +\int_0^t\int_0^L \partial_t\hat \eta \,\partial_t{\psi}
	+\frac12\int_0^t\int_{\sO}\partial_t\hat {\eta} ^*\hat{\bu} \cdot\bq\\
	&-\int_0^t\int_0^L\partial_z\hat\eta \partial_z\psi+\partial_{zz}\hat\eta \partial_{zz}\psi 
	-2\nu\int_0^t\int_{\sO}(R+\hat{\eta} ^*) \bD^{\hat\eta ^*}(\hat\bu )\cdot \bD^{\hat\eta ^*}(\bq)   \\
	&-\frac12\int_0^t\int_{\sO}(R+\hat\eta ^*)((\hat\bu-\partial_t\hat {\eta}
	r\be_r)\cdot\nabla^{\hat\eta ^*}\hat\bu \cdot\bq
	- (\hat\bu^{}-\partial_t\hat {\eta} 
	r\be_r)\cdot\nabla^{\hat\eta ^*}\bq\cdot\hat\bu )\\
	&+ 
	\int_0^t\left( P_{{in}}\int_{0}^1q_z\Big|_{z=0}dr-P_{{out}}\int_{0}^1q_z\Big|_{z=1}dr\right)
	ds +\int_0^t
	(	G(\hat\bu,\hat v ,\hat\eta^*)d\hat W , \bQ)  ,
	\end{split}
	\end{equation}
	$\hat\bP$-a.s. for almost every $t\in[0,T]$ and for any $(\hat\sF_t)_{t \geq 0}-$adapted process $\bQ=(\bq,\psi)$ with continuous paths in $\sD$ such that div$^{\hat\eta^*}\bq=0$. 
\end{theorem}
\begin{proof}[Proof of Theorem \ref{exist2}]
{Passing to the limit as $\ep\to0$ in \eqref{approxsystem2} will be done in the following four steps. In step 1 we will construct appropriate test functions $\bQ_\ep$ for the weak formulation \eqref{approxsystem2}. In step 2 we
	discuss the limit $\ep \to0$ of the stochastic integral. In step 3 we show that the limit $\hat\bu$ satisfies the transformed
	incompressibily condition. In step 4 we pass to the limit in the remaining terms, of which
	the nonlinear advection term is the only one which requires discussion.}
We present these steps next.

{\bf Step 1.}	To pass to the limit in \eqref{approxsystem2}, we {will}  consider test functions $(\bq_{\ep},\psi_{\ep})$ taking values in $\sD$, {where $\sD$ is defined in \eqref{Dspace},} such that div$^{\hat{\eta}^*_{\ep}}\bq_{\ep}=0$ so that the penalty term drops out. At the same time we want \eqref{martingale2} to hold for every $\sD-$valued process $(\bq,\psi)$ such that div$^{\hat\eta^*}\bq=0$. This dependence of test functions on $\hat\eta^*$ leads us to construct test functions, specific to our problem, on the maximal rectangular domain $\sO_\delta$ consisting of all the fluid domains associated with the structure displacements $\hat\eta_\ep^*$. 
	
	We begin by constructing an appropriate test function for the limiting equation \eqref{martingale2} as follows: Consider a smooth, essentially bounded, $(\hat\sF_t)_{t\geq0}$-adapted process $\br=(r_1,r_2)$ on $\bar\sO_{\hat\eta^*}$ such that $\nabla\cdot \br=0$ and such that $\br$ satisfies the required boundary conditions  $r_2=0 \text{ on } z=0,L, r=0$  and $\partial_r r_1=0 \text{ on }\Gamma_{b}$. Assume also that on the top lateral boundary of the moving domain associated with $\hat\eta^*$, the function $\br$ satisfies $ \br(t,z,R+\hat\eta^*(t,z))=\psi(t,z){\be}_r$.  Next, we define
	$$ \bq(t,z,r,\omega) = \br(t,\omega) \circ\, A^\omega_{\hat\eta^*}(t)(z,r) .$$	
	Now, we will show that $(\bq,\psi)$ is an appropriate test function as it appears in the statement of Theorem \ref{exist2}.
	
	For that purpose, for any $t\in[0,T]$ and given process $\br$ as mentioned above, let $\sC_\br:\hat\Omega\times C([0,L]) \rightarrow \bC^1(\bar\sO)$ be defined as $$\sC_{\br}(\omega,\eta)
	=F_\eta(\br(t,\omega)),$$
	where $F_\eta({\bf f})(z,r):={\bf f}(z,(R+\eta(z))r)$ is a well-defined map from $\bC(\bar\sO_{\eta})$ to $\bC(\bar\sO)$ for  any $\eta\in C([0,L])$. Thanks to the continuity of the composition operator $F_\eta$ and the assumption that $\br(t)$ is $\hat\sF_t$-measurable, we obtain that, for any $\eta$, the $\bC^1(\bar\sO)$-valued map $\omega \mapsto \sC_\br(\omega,\eta)$ is $\hat\sF_t$-measurable (where $\bC^1(\bar\sO)$ is endowed with Borel $\sigma$-algebra).
	Note also that for any fixed $\omega$, the map  $\eta\mapsto \sC_\br(\omega,\eta)$ is continuous.
	Hence we deduce that $\sC_\br$ is a Carath\'eodory function.
	Now by the construction of the filtration $(\hat\sF_t)_{t\geq 0}$ in \eqref{Ft} we know that $\hat\eta^*$ is $(\hat\sF_t)_{t\geq 0}$-adapted. Therefore, we conclude that the $\bC^1(\bar\sO)$-valued process $\bq$, which by definition is $\bq(t,\omega)=\sC_\br(\omega,\hat\eta^*(t,\omega))$,
	 is $(\hat\sF_t)_{t\geq 0}$-adapted as well. The same conclusions follow for the process $\psi$, using the same argument.
	
	{ Now we will begin our construction of the approximate test functions $\bQ_\ep$ to be used in \eqref{approxsystem2} to pass to the limit $\ep\rightarrow 0$.  These test functions $\bQ_\ep$ need to satisfy the divergence-free condition on the domains associated with the approximate functions $\hat\eta^*_\ep$ along with the right boundary conditions. Furthermore, they need to converge to the test function $\bQ$ in an appropriate sense. As mentioned in the introduction, we need to be careful as these test functions $\bQ_\ep$ are also required to satisfy appropriate measurability properties (see Lemma \ref{conv_G}). 
		Hence, as done earlier in the proof of Lemma \ref{tightuv}, we construct these test functions by extending and then "squeezing" the function $\br$ while also ensuring that its desired properties are preserved.}
	
	Hence, define $\sC_{ext}(\omega,\eta)= E_\eta(\br(t,\omega))$ where $E_\eta$ is the operator that extends in the vertical direction $\be_r$, the boundary data of functions defined on $\sO_\eta$ to $\bar\sO_\delta$. Let $\br^0(t,\omega)= \sC_{ext}(\omega,\hat\eta^*(t,\omega))$,
	that is,
	\begin{align*}
	(r^0_z,r^0_r)=\br^0:=\begin{cases}
	\br(t,z,r,\omega), &\text{if } r \leq R+\hat\eta^*(t,z,\omega),\\
	(0,\psi(t,z,\omega)),& \text{elsewhere in } \bar\sO_\delta.
	\end{cases}
	\end{align*}
	Adaptedness of $\br^0$ then follows using an argument identical to the one above by realizing that $\sC_{ext}$ is a Carath\'eodory function. Notice also that since  div$(0,\psi(t,z))=0$ we have that  div $\br^0=0$ as well. For properties of $\br^0$, see e.g. Section 11.3 in \cite{LM72} and \cite{Mi11}.
	
	Now we will scale $\br^0=(r^0_z,r^0_r)$
	to construct its suitable $\ep$-approximations while also preserving its desired properties. Define for any $\omega\in\hat\Omega$, the random variable $$\beta_\ep(\omega)=%\frac{h(\omega)}{h(\omega)-\alpha(\ep,\omega)},
\max\Big\{\sup_{(t,z)\in [0,T]\times[0,L]}	\frac{R+\hat\eta^*}{R+\hat\eta^*_\ep}, 1\Big\}
	$$ 
%	where $h(\omega)=\inf_{(t,z)\in[0,T]\times [0,L]}[R+\hat\eta^*(t,z,\omega)]$ and $\alpha(\ep,\omega)=\max\{0,\sup_{(t,z)\in[0,T]\times [0,L]}(\hat\eta^*-\hat\eta^*_\ep)\}$.
Since $|\beta_\ep-1| \leq \sup_{(t,z)\in [0,T]\times[0,L]}	\frac{|\hat\eta^*-\hat\eta^*_\ep|}{R+\hat\eta^*_\ep} \leq \frac1\delta{\|\hat\eta^*-\hat\eta^*_\ep\|_{L^\infty(0,T)\times(0,L)}}$, thanks to the uniform convergence \eqref{etaunif2}, we know that $\beta_\ep \rightarrow 1$ almost surely.
	
	Using this definition we set,
	\begin{align}\label{r_ep}
	\br^0_\ep(t,z,r,\omega)=\begin{cases}
	(\beta_\ep(\omega) r^0_z(t,z,\beta_\ep(\omega) r,\omega),r^0_r(t,z,\beta_\ep(\omega) r,\omega)), &\text{ if }  r\leq \frac{1}{\delta\beta_\ep},\\
	(0,\psi(t,z,\omega)), &\text{ elsewhere in }\bar\sO_\delta.
	\end{cases}
	\end{align}
	Notice that by scaling $\br^0$ in this fashion we still have that
	$\nabla\cdot \br^0_\ep=0$  for every $\omega\in\hat\Omega$. 
	It is easy to see that $\br^0_\ep(t,z,R+\hat\eta^*_\ep(t,z,\omega),\omega)=(0,\psi(t,z,\omega))$.  
	
	We will next use $\br^0_\ep$ to build the test {function} $\bq_\ep$  for the fluid equations on the fixed domain $\sO$. To study the properties of $\bq_\ep$ we will first give the construction of an appropriate filtration $\left(\hat{\mathcal{F}}^{\ep}_t \right)_{t \geq 0}$ on the probability space $(\hat\Omega,\hat{\mathcal{F}},\hat\bP)$ as follows: For a fixed $\ep \in \Lambda$, let $\tilde\sF^\ep_t$ be the $\sigma$-field generated by the random variables 
	$(\hat{\bu}_\ep(s),\hat v_\ep(s)),\hat\eta_\ep(s),\hat{W}_\ep(s), \beta_\ep$ for all $s \leq t$ and define
	\begin{align}\label{Ftep}
	\hat{\mathcal{F}}^\ep_t :=\sigma\left( \bigcap_{s\ge t}\sigma(\tilde{\sF}^\ep_s\cup \mathcal{N})\cup \hat\sF_t\right).
	\end{align}
	Finally, we define the approximate test functions as follows:
	\begin{align}\label{q_ep}
	\bq_{\ep}(t,z,r,\omega) =  \br^0_\ep(t,\omega)|_{\sO_{\hat\eta^*_{\ep}}} 
	\circ A^\omega_{\hat\eta^*_{\ep}}(t)(z,r)
	.\end{align}
	Adaptedness of the process $\bq_\ep$ to the filtration $(\hat\sF^\ep_t)_{t\geq 0}$ follows from the adaptedness of $\br^0_\ep$ and by using the same arguments as above.	Additionally, since every realization of $\br^0_\ep$ is divergence-free in $\sO_{\hat\eta^*_\ep}$, we have that $\text{div}^{\hat\eta^*_{\ep}(t)} \bq_{\ep}(t)=0$. Admissibility of the boundary values $\bq_\ep|_{\partial\sO}$ can be verified easily. Hence, the $\sD$-valued process $\bQ_{\ep}=(\bq_{\ep},\psi_{})$ will serve as our random test function for the $\ep$-approximation.	
	
	Now we study the convergence of $\bQ_\ep$ as $\ep \to 0$.
	Observe that, for any $\omega\in\hat\Omega$, using the mean value theorem we can write
	\begin{equation}\label{mvtq}
	\begin{split}
	|\bq_\ep(t,z,r)-\bq(t,z,r)|&=|\br^0_\ep(t,z,(R+\hat\eta_\ep^*(t,z))r)-\br^0(t,z,(R+\hat\eta^*(t,z))r)|\\
	&\leq|\partial_r\br^0_\ep(t,z,\rho)r||\hat\eta_\ep^*(t,z)-\hat\eta^*(t,z)|\\
	&+|\br^0_\ep(t,z,(R+\hat\eta^*(t,z))r)-\br^0(t,z,(R+\hat\eta^*(t,z))r)|.
	\end{split}	\end{equation} 
	{To simplify notation, in the following calculation we will denote $\tilde r=(R+\hat\eta^*(t,z))r$, for any $r\in[0,1]$.}
	We obtain
	\begin{align*}
	&|\br^0_\ep(t,z,(R+\hat\eta^*(t,z))r)-\br^0(t,z,(R+\hat\eta^*(t,z))r)|\\
	&=\left( |(\beta_\ep r^0_z(t,z,\beta_\ep \tilde r),r^0_r(t,z,\beta_\ep \tilde r))-( r^0_z(t,z, \tilde r),r^0_r(t,z, \tilde r))|\right) \mathbbm{1}_{\{r\leq \frac{1}{\beta_\ep}\}}+|\br^0(t,z,\tilde r)|\mathbbm{1}_{\{ \frac{1}{\beta_\ep}< r\leq 1\}}\\
	&\leq \left( |{\br^0}(t,z,\beta_\ep \tilde r)- {\br^0}(t,z,\tilde r)|+ |((\beta_\ep -1)r^0_z(t,z,\beta_\ep \tilde r),0)|\right) \mathbbm{1}_{\{r\leq \frac{1}{\beta_\ep}\}}+|\br^0(t,z,\tilde r)|\mathbbm{1}_{\{ \frac{1}{\beta_\ep}< r\leq 1\}}\\
	&\leq [|\partial_r\br^0(t,z,\rho_1)(R+\hat\eta^*(t,z))r|+|(r^0_z(t,z,\beta_\ep \tilde r),0)|](\beta_\ep-1)\mathbbm{1}_{\{r\leq \frac{1}{\beta_\ep}\}}+|\br^0(t,z,\tilde r)|\mathbbm{1}_{\{ \frac{1}{\beta_\ep}< r\leq 1\}}.
	\end{align*}
	Thus using \eqref{etaunif2} we pass $\ep\rightarrow 0$ in \eqref{mvtq} to obtain that
	\begin{align}
	\bq_{\ep} \rightarrow \bq \quad \text{ in }  L^\infty(0,T;\bL^{p}(\sO)) 
	\quad \hat\bP-  a.s. , \quad \text{ for any } p<\infty.
	\end{align}
	Now, since by definition $\nabla^{\hat\eta^*_\ep} \bq_\ep= \nabla\br^0_\ep$, we can carry out similar calculations for the transformed gradient to obtain
	\begin{align}
	\nabla^{\hat\eta^*_\ep}\bq_\ep\rightarrow \nabla^{\hat\eta^*}\bq \quad \text{ in }  L^\infty(0,T;\bL^{p}(\sO)) 	\quad \hat\bP-  a.s. , \quad \text{ for any } p<\infty.
	\end{align}
	Recalling that $\nabla\bq_\ep=(\nabla^{\hat\eta^*_\ep}\bq_\ep)(\nabla A^\omega_{\hat\eta^*_\ep})$ and that $\nabla A_{\hat\eta^*_\ep}^\omega \rightarrow \nabla A_{\hat\eta^*}^\omega $ in $L^\infty(0,T;\bC(\bar\sO))$ for almost every $\omega\in\hat\Omega$, we summarize our convergence results for the test functions below:
	\begin{align}\label{convq}
	\bq_{\ep} \rightarrow \bq \quad \text{ in }  L^\infty(0,T;\bW^{1,p}(\sO)) 
	\quad \hat\bP-  a.s. ,\quad \text{ for any } p<\infty.
	\end{align}
	Observe also that 
	$\partial_t \bq_\ep=\partial_t\br^0_\ep\circ A_{\hat\eta^*_\ep}+\nabla^{\hat\eta^*_\ep}\bq_\ep \cdot\partial_t\hat\eta^*_\ep r\be_r$.
	Since we know that $\hat v_\ep^* \rightharpoonup \hat v^*$ in $L^\infty(0,T;L^2(0,L))$ a.s., we further infer
	\begin{align}\label{convqt}
	\partial_t\bq_\ep \rightharpoonup \partial_t\bq \quad\text{ weakly in } L^2(0,T;\bL^q(\sO)) \quad \text{ for any } q<2, \quad\hat\bP-a.s.
	\end{align}
	%As mentioned earlier, we do not obtain uniform convergence results for the test functions constructed in this manner {as was done in \cite{MC13}}. However, the convergence results obtained in \eqref{uweak2} and \eqref{ul4} compensate for this fact.
	
	{We are now in a position to take the limit as $\ep\to 0$ in the weak formulation \eqref{approxsystem2}. Namely, we consider \eqref{approxsystem2}  
		and use $\bQ_{\ep}=(\bq_{\ep},\psi)$ as the test functions.
		This requires} a special version of the It\^o product rule which can be proven by using a regularization argument as outlined in Lemma 5.1 in \cite{BO13}. We obtain that
	\begin{equation}\begin{split}\label{epeq}
	&(	(R+\hat{\eta}_\ep ^*(t))\hat\bu_\ep (t),\bq_\ep(t))+(\hat{ v}_\ep(t),\psi(t))%+\ep(\hat{ p} (t),\varphi)
	=(\bu_0(R+\eta_0),\bq_\ep(0))+( v_0,\psi(0))\\
	&+\int_0^t\int_\sO(R+{\hat\eta_\ep}^*)\hat\bu_\ep \partial_t{\bq}_\ep +\int_0^t\int_0^L \hat v_\ep \partial_t{\psi}+\frac12\int_0^t\int_{\sO}\hat {v}_\ep ^*\hat{\bu}_\ep \cdot\bq_\ep\\
	&-\int_0^t\int_0^L(\partial_z\hat\eta_\ep \partial_z\psi+\partial_{zz}\hat\eta_\ep \partial_{zz}\psi )
	-2\nu\int_0^t\int_{\sO}(R+\hat{\eta} _\ep^*) \bD^{\hat\eta_\ep ^*}(\hat\bu_\ep )\cdot \bD^{\hat\eta_\ep ^*}(\bq_\ep)  \\
	&-\frac12\int_0^t\int_{\sO}(R+\hat\eta_\ep ^*)((\hat\bu_\ep-\hat {v}_\ep
	r\be_r)\cdot\nabla^{\hat\eta_\ep ^*}\hat\bu_\ep \cdot\bq_\ep
	- (\hat\bu_\ep-\hat{v} _\ep
	r\be_r)\cdot\nabla^{\hat\eta_\ep ^*}\bq_\ep\cdot\hat\bu_\ep )\\
	&+ 
	\int_0^t\left( P_{{in}}\int_{0}^1(q_\ep)_z\Big|_{z=0}dr-P_{{out}}\int_{0}^1(q_\ep)_z\Big|_{z=1}dr\right)
	ds +\int_0^t
	(	G(\hat\bu_\ep,\hat v_\ep ,\hat\eta_\ep^*)d\hat W_\ep , \bQ_\ep),
	\end{split}
	\end{equation}
	holds $\hat\bP-$a.s. for every $t\in[0,T]$.
	
	Now we begin passing {$\ep \rightarrow 0$ in \eqref{epeq}. 
		
		{\bf Step 2.} We start with the stochastic integral.} Recall Remark \ref{rem:integral} and the fact that, by construction, $\bQ$ is $(\hat\sF_t)_{t\geq 0}$-adapted.
	\begin{lem}\label{conv_G}
		The sequence of processes $
		\left( \int_0^t(G(\hat{\bu}_{\ep}(s),\hat v_\ep(s),\hat\eta^*_{\ep}(s))d\hat{W}_{\ep}(s),\bQ_{\ep}(s))\right) _{t \in [0,T]}$ converges to $\left( \int_0^t(G(\hat{\bu}(s),\hat v(s),\hat\eta^*(s))d\hat{W}(s),\bQ(s))\right) _{t \in [0,T]}$  in $L^1(\hat{ \Omega};L^1(0,T;\mathbb{R}%\bL^2(\sO)\times L^2(0,L)
		))$ as ${\ep}\rightarrow 0$.
	\end{lem}
	\begin{proof}
		First, by using \eqref{growthG}$_{2,3}$ we observe that 
		\begin{align*}
		&	\int_0^T\|{(G(\hat{\bu}_{\ep},\hat v_\ep,\hat\eta^*_{\ep}),\bQ_{\ep})-( G(\hat \bu,\hat v,\hat\eta^*),\bQ)}\|^2_{L_2(
			U_0,\R)}d s\\
		&\leq \int_0^T \|(G(\hat{\bu}_{\ep},\hat v_\ep,\hat{\eta}^*_{\ep})-G(\hat{\bu},\hat v,\hat{\eta}^*_{\ep}),\bQ_{\ep})\|^2_{L_2(U_0;\mathbb{R})} + \int_0^T\|(G(\hat{\bu},\hat v,\hat{\eta}^*_{\ep}),\bQ_{\ep}-\bQ)\|^2_{L_2(U_0;\mathbb{R})}\\
		&\qquad +\int_0^T\|(G(\hat\bu,\hat v,\hat\eta^*_{\ep})-G(\hat\bu,\hat v,\hat\eta^*),\bQ)\|^2_{L_2(U_0;\mathbb{R})}\\
		&\le \int_0^T\left( \|\hat\eta^*_{\ep}\|^2_{L^\infty(0,L)}\|\hat{\bu}_{\ep}-\hat\bu\|^2_{\bL^2(\sO)}+\|\hat v_{\ep}-\hat v\|_{L^2(0,L)}^2\right) \|\bQ_{\ep}\|_{\bL^2}^2\\
		&+  \int_0^T\left( \|\hat\eta^*_{\ep}\|^2_{L^\infty(0,L)}\|\hat\bu\|^2_{\bL^2(\sO)}+\|\hat v\|_{L^2(0,L)}^2\right) \|\bQ_{\ep}-\bQ\|_{\bL^2}^2+ \int_0^T\|\hat\eta^*_{\ep}-\hat\eta^*\|^2_{L^\infty(0,L)}\|\hat{\bu}\|^2_{\bL^2(\sO)}\|\bQ\|_{\bL^2}^2\\
		&\leq C(\delta)\Big( \|\bQ_{\ep}\|^2_{L^\infty(0,T;\bL^2)}(\|\hat\bu_{\ep}-\hat\bu\|^2_{L^2(0,T;\bL^2(\sO))}+\|\hat v_{\ep}-\hat v\|^2_{L^2(0,T;L^2(0,L))})\\
		&+\|\bQ_{\ep}-\bQ\|^2_{L^\infty(0,T;\bL^2)}\left( \|\hat{\bu}\|_{L^2(0,T;\bL^2(\sO))}^2 +\|\hat v\|^2_{L^2(0,T;L^2(0,L))}\right) \\
		&+ \|\hat\eta^*_{\ep}-\hat\eta^*\|_{L^\infty((0,T)\times(0,L))}\|\hat{\bu}\|^2_{L^2(0,T;\bL^2(\sO))}\|\bQ\|_{L^\infty(0,T;\bL^2)}^2\Big).
		\end{align*}	
		Thanks to Theorem \ref{skorohod2}, the right hand side of the inequality above converges to 0, $\hat\bP-$a.s. as ${\ep}\rightarrow 0$.	Hence we have proven that 
		\begin{align}\label{g1}
		(G(\hat{\bu}_{\ep},\hat v_\ep,\hat\eta^*_{\ep}),\bQ_{\ep})\rightarrow( G(\hat \bu,\hat v,\hat\eta^*),\bQ), \qquad \text{$\hat\bP-$a.s \quad in $L^2(0,T;L_2(U_0,\R		)).$ }\end{align}
		Now using classical ideas from \cite{Ben} (see also Lemma 2.1 of \cite{DGHT} for a proof), the convergence \eqref{g1} implies that
		\begin{align}\label{G2}
		\int_0^t(G(\hat{\bu}_{\ep},\hat v_\ep,\hat\eta^*_{\ep})d\hat W_\ep,\bQ_{\ep}) \rightarrow \int_0^t( G(\hat \bu,\hat v,\hat\eta^*)d\hat{W},\bQ)
		\end{align}
		in probability in $L^2(0,T;\R)$.
		
		Furthermore observe that for some $C>0$ independent of $\ep$ we have the following bounds which follow from the It\^{o} isometry:
		\begin{align}
		\hat \bE\int_0^T |\int_0^t (G(\hat{\bu}_{\ep},\hat v_\ep,\hat\eta^*_{\ep})d\hat W_\ep(s),\bQ_{\ep})&|^2d t=\int_0^T \hat\bE\int_0^t\|(G(\hat{\bu}_{\ep},\hat v_\ep,\hat\eta^*_{\ep}),\bQ_{\ep})\|^2_{L_2(U_0,\R)}d s d t \notag\\
		& \le  T
		\hat \bE\left(\int_0^T\left( \|{\hat\eta^*_{\ep}}\|_{L^\infty(0,L)}^2\|\hat{\bu}_{\ep}\|^2_{\bL^2(\sO)}+\|\hat v_{\ep}\|^2_{L^2(0,L)}\right) d s\right)\label{G3}\\
		&\le   C(\delta).\notag
		\end{align}
		Combining \eqref{G2}, \eqref{G3} and using the Vitali convergence theorem, we thus conclude the proof of Lemma \ref{conv_G}.
	\end{proof}	
	
	{{\bf{Step 3.}} We now show that the limit $\hat\bu$ solves the transformed incompressibily condition.}
	{To do that, observe} that \eqref{uniformep}$_{4}$ implies that for any measurable set $A\subset \hat\Omega\times[0,T]$  we have $\hat\bE\int_0^T\chi_A\left( \text{div}^{\hat\eta^*_\ep}\hat\bu_\ep,\varphi\right)  \rightarrow 0$ for any $\varphi \in L^2(\sO)$. 
	Observe also that, thanks to \eqref{etaunif2} and \eqref{etastarbound}, we can infer that $\frac{\partial_z\hat\eta^*_\ep}{R+\hat\eta^*_\ep} \rightarrow \frac{\partial_z\hat\eta^*}{R+\hat\eta^*} $ uniformly on $[0,T]\times\sO$ a.s.
	and  also that $\hat\bE\|\frac{\partial_z\hat\eta^*_\ep}{R+\hat\eta^*_\ep}\|^p_{L^\infty(0,T;L^\infty(\sO))}<C$ for any $p>2$. An application of the Vitali convergence theorem thus gives us that $$\frac{\partial_z\hat\eta^*_\ep}{R+\hat\eta^*_\ep} \rightarrow \frac{\partial_z\hat\eta^*}{R+\hat\eta^*} 
	\quad \text{in } L^2(\hat\Omega;L^\infty(0,T;L^\infty(\sO))).$$ Combining this with the weak convergence (up to a subsequence) result that we obtain as a consequence of \eqref{uniformep}$_{4}$, we deduce that $\hat\bE\int_0^T\chi_A\left( \text{div}^{\hat\eta^*}\hat\bu,\varphi\right)  = 0$ for any $\varphi \in L^2(\sO)$ and measurable set $A$. Thus div$^{\hat\eta^*(t)}\hat\bu(t)=0$ for almost every $t\in[0,T]$ almost surely.
	
	{\bf{Step 4.}} We now pass to the limit $\ep \rightarrow 0$ in the remaining terms. For the first term we write,
	\begin{align*}
		|\int_0^t\int_\sO \hat\bu_\ep\partial_t\bq_\ep-\hat\bu\partial_t\bq| \leq 	|\int_0^t\int_\sO (\hat\bu_\ep-\hat\bu)\partial_t\bq_\ep| + 	|\int_0^t\int_\sO \hat\bu(\partial_t\bq_\ep-\partial_t\bq)|\to 0 \quad a.s. \text{ in } L^\infty(0,T).
	\end{align*}
	The first term on the right side converges to 0 in $L^\infty(0,T)$ almost surely thanks to \eqref{ul4} whereas the second term converges to 0 thanks to the weak convergence \eqref{convqt} in $L^{2}(0,T;\bL^\frac43(\sO))$.
	 We now focus on a part of the nonlinear advection term on the right side of \eqref{epeq} and leave the rest of the proof of convergence to the reader, which is carried out in the same way using the results Theorem \ref{skorohod2}, \eqref{uweak2}, \eqref{ul4}, \eqref{div0}, \eqref{etaunif2} and  \eqref{convq}, \eqref{convqt}.
	
	Observe that by integrating by parts we can write
	\begin{align*}
	&\int_0^t\int_\sO (R+\hat\eta^*_\ep)((\hat\bu_\ep\cdot\nabla^{\hat\eta^*_\ep})\hat\bu_\ep\cdot\bq_\ep-(\hat\bu_\ep\cdot\nabla^{\hat\eta^*_\ep})\bq_\ep\cdot\hat\bu_\ep)\\
	&=-2\int_0^t\int_\sO (R+\hat\eta^*_\ep)(\hat\bu_\ep\cdot\nabla^{\hat\eta^*_\ep})\bq_\ep\cdot\hat\bu_\ep -\int_0^t\int_\sO(R+\hat\eta^*_\ep) 
	{\text{div}^{\hat\eta^*_\ep}\hat\bu_\ep \ \hat\bu_\ep\cdot\bq_\ep } +\int_0^t\int_0^L ( \hat v_\ep)^2\psi dz.
	\end{align*}
	It is easy to see that the last term on the right hand side converges to $\int_0^t\int_0^L\hat v^2\psi dz$ in $L^\infty(0,T)$ a.s. Now, for the other terms, writing $\hat\bu_\ep=(\hat u^1_\ep,\hat u^2_\ep)$ (likewise $\hat\bu$) and $\nabla^{\hat\eta^*_\ep}=(\partial_1^{\hat\eta^*_\ep},\partial_2^{\hat\eta^*_\ep})$, we observe that,
	\begin{align*}
	&|\int_0^t\int_\sO (R+\hat\eta^*_\ep)(\hat\bu_\ep\cdot\nabla^{\hat\eta^*_\ep})\bq_\ep\cdot\hat\bu_\ep-\int_0^t\int_\sO (R+\hat\eta^*)(\hat\bu\cdot\nabla^{\hat\eta^*})\bq\cdot\hat\bu| \\
	&=|\int_0^t\int_\sO (R+\hat\eta^*_\ep)  \sum_{i,j=1}^2 (\hat u^i_\ep\hat u^j_\ep-\hat u^i\hat u^j)\partial_i^{\hat\eta^*_\ep}q^j_\ep+ \int_0^t\int_\sO (\hat\eta^*_\ep-\hat\eta^*) (\hat\bu\cdot \nabla^{\hat\eta^*})\bq\cdot \hat\bu  \\
	&+\int_0^t\int_\sO (R+\hat\eta^*_\ep) (\hat\bu\cdot (\nabla^{\hat\eta^*_\ep}\bq_\ep-\nabla^{\hat\eta^*}\bq))\cdot \hat\bu|.
	\end{align*}
	The first two terms on the right side converge to 0 a.s thanks to Theorem \ref{skorohod2} and \eqref{convq}. Similarly for the third term, we use \eqref{convq} and obtain that
	\begin{align*}
	|\int_0^t\int_\sO (R+\hat\eta^*_\ep) (\hat\bu\cdot (\nabla^{\hat\eta^*_\ep}\bq_\ep-\nabla^{\hat\eta^*}\bq))\cdot \hat\bu| &\leq C\|\hat\bu\|^2_{L^2(0,T;\bL^4(\sO))}\|\nabla^{\hat\eta^*_\ep}\bq_\ep-\nabla^{\hat\eta^*}\bq\|_{L^\infty(0;T;\bL^2(\sO))}\\
	&\rightarrow 0, \qquad \hat\bP-a.s.
	\end{align*}
	Almost sure convergence of the third integral to 0 in $L^\infty(0,T)$ is immediate. 
	Finally, thanks to \eqref{uweak2} we have
	\begin{align*}
	\sup_{t\in[0,T]}|\int_0^t\int_\sO(R+\hat\eta^*_\ep) \text{div}^{\hat\eta^*_\ep}&\hat\bu_\ep \hat\bu_\ep\cdot\bq_\ep| \leq \int_0^T\int_\sO|(R+\hat\eta^*_\ep) \text{div}^{\hat\eta^*_\ep}\hat\bu_\ep \hat\bu_\ep\cdot\bq_\ep|
	\\&\leq C(\delta)  \|\text{div}^{\hat\eta^*_\ep}\hat\bu_\ep\|_{L^2(0,T;L^2(\sO))}\|\hat\bu_\ep\|_{L^2(0,T;\bL^4(\sO))}\|\bq_\ep\|_{L^\infty(0,T;\bL^4(\sO))}\\
	&\rightarrow 0, \qquad \hat\bP-a.s.
	\end{align*}
	{This completes the proof of Theorem \ref{exist2}.}
	
\end{proof}

{Notice that in the statement of Theorem~\ref{exist2} we still have the function $\hat\eta^*(t)$ in the weak formulation, which keeps the displacement uniformly bounded. 
	We now show that in fact $\hat\eta^*(t)$ can be replaced by the limiting stochastic process $\hat\eta(t)$ to obtain the desired weak formulation 
	and martingale solution until some stopping time 
	$\tau$, which we show is strictly greater than zero almost surely.
	
	\begin{lem}[Stopping time] \label{StoppingTime}
		Let the deterministic initial data $\eta_0$ satisfy the assumptions \eqref{etainitial}.
		Then, for any $\delta > 0$ and for a given $\frac32<s<2$,
		there exists an almost surely positive $(\hat\sF_t)_{t\geq 0}-$stopping time $\tau$, given by
		\begin{equation}\label{stoppingT}
		\tau:=T\wedge\inf\{t>0:\inf_{z\in[0,L]}(R+ \hat\eta(t,z))\leq \delta\} \wedge\inf\{t>0:\|R+\hat\eta(t)\|_{H^s(0,L)}\geq \frac1{\delta} \},
		\end{equation}
		such that
		\begin{align}\label{etasequal2}
		\hat	\eta^*(t)=\hat\eta(t) \quad \text{for } t<\tau.
		\end{align}
	\end{lem}
	
	\begin{proof} In {\bf{Step 1}} below we first show that the stopping time \eqref{stoppingT} is strictly positive a.s., and then is {\bf{Step 2}} we show that 
		$\hat	\eta^*(t)=\hat\eta(t) \quad \text{for } t<\tau.$
		
		{{\bf{Step 1.}}}
		{We start by showing that} under the assumptions \eqref{etainitial} on the deterministic initial data $\eta_0$, the stopping time $\tau$ is {almost surely strictly positive  for any $\delta>0$.}
		{For this purpose let us write the stopping time as 
			$$\tau = T\wedge\tau^{(1)}+\tau^{(2)}, $$
			where $\tau^{(1)}$ and $\tau^{(2)}$ are defined by:
			$$	
			\tau^{(1)}=\inf\{t>0:\inf_{z\in[0,L]}(R+ \hat\eta(t,z))\leq \delta\}; \quad \tau^{(2)}=\inf\{t>0:\|R+\hat\eta(t)\|_{H^s(0,L)}\geq \frac1{\delta} \}.
			$$
		}
		
		We start with $\tau^{(2)}$. Observe that using the triangle inequality, for any $\delta_0>{\delta}$, we obtain
		\begin{align*}
		\hat\bP[\tau^{(2)}=0, \|R+\eta_0\|_{H^2(0,L)}&<\frac1{\delta_0}] =\lim_{\epsilon\rightarrow 0}\hat\bP[\tau^{(2)}<\epsilon,\|R+\eta_0\|_{H^2(0,L)}<\frac1{\delta_0}]\\
		&\leq \limsup_{\epsilon \rightarrow 0^+}\hat\bP[\sup_{t\in[0,\epsilon)}\|R+\hat\eta(t)\|_{H^s(0,L)}>\frac1{\delta},\|R+\eta_0\|_{H^2(0,L)}<\frac1{\delta_0}]\\
		&\leq \limsup_{\epsilon \rightarrow 0^+}\hat\bP[\sup_{t\in[0,\epsilon)}\|\hat\eta(t)-\eta_0\|_{H^s(0,L)}>\frac1{\delta}-\frac1{\delta_0}] \\
		&\leq \frac1{(\frac1{\delta}-\frac1{\delta_0})}\limsup_{\epsilon\rightarrow 0}\hat\bE[\sup_{t\in[0,\epsilon)}\|\hat\eta(t)-\eta_0\|_{H^s(0,L)}]\\
		&\leq \frac1{(\frac1{\delta}-\frac1{\delta_0})}\limsup_{\epsilon\rightarrow 0}\hat\bE[\sup_{t\in[0,\epsilon)}\|\hat\eta(t)-\eta_0\|^{1-\frac{s}{2}}_{L^2(0,L)}\|\hat\eta(t)-\eta_0\|^{\frac{s}2}_{H^2(0,L)}]\\
		&\leq \frac1{(\frac1{\delta}-\frac1{\delta_0})}\limsup_{\epsilon\rightarrow 0}\hat\bE[\sup_{t\in[0,\epsilon)}\epsilon\|\hat v(t)\|^{1-\frac{s}2}_{L^2(0,L)}\|\hat\eta(t)-\eta_0\|^{\frac{s}2}_{H^2(0,L)}]\\
		&\leq \limsup_{\epsilon\rightarrow 0} \frac{\epsilon}{(\frac1{\delta}-\frac1{\delta_0})}\left( \hat\bE[\sup_{t\in[0,\epsilon)}\|\hat v(t)\|^{2}_{L^2(0,L)}]\right)^{\frac{2-s}{4}}\left(\hat\bE[\sup_{t\in(0,\epsilon)}\|\hat\eta(t)-\eta_0\|^{2}_{H^2(0,L)}]\right)^{\frac{s}4} \\
		&=0.
		\end{align*}
		Hence, by continuity from below, we deduce that for any $\delta>0$, 
		\begin{align}
		\hat\bP[\tau^{(2)}=0,  \|R+\eta_0\|_{H^2(0,L)}<\frac1{\delta}]=0.
		\end{align}
		
		{To estimate $\tau^{(1)}$ we observe that, similarly, } since for any $t\in[0,T]$ we have that $\inf_{z\in(0,L)}(R+\hat\eta(t))\geq \inf_{z\in(0,L)}(R+\eta_0)-\|\hat\eta(t)-\eta_0\|_{L^\infty(0,L)}$, we write for any $\delta_0>\delta$ % and $1<s<2$ interpolating $H^s$ between $L^2$ and $H^2$
		\begin{align*}
		\hat\bP[\tau^{(1)}=0, \inf_{z\in(0,L)}(R+\eta_0)>\delta_0] &\leq \limsup_{\epsilon \rightarrow 0^+}\hat\bP[\inf_{t\in[0,\epsilon)}\inf_{z\in(0,L)}(R+\hat\eta(t))<\delta,\inf_{z\in(0,L)}(R+\eta_0)>\delta_0]\\
		&\leq \limsup_{\epsilon \rightarrow 0^+}\hat\bP[\sup_{t\in[0,\epsilon)}\|\hat\eta(t)-\eta_0\|_{L^\infty(0,L)}>\delta_0-\delta] \\
		&\leq \frac1{(\delta_0-\delta)^2}\limsup_{\epsilon\rightarrow 0}\hat\bE[\sup_{t\in[0,\epsilon)}\|\hat\eta(t)-\eta_0\|^2_{H^1(0,L)}]\\
		&=0.
		\end{align*}
		Hence for given $\delta>0$, 
		\begin{align}
		\hat\bP[\tau^{(1)}=0, \inf_{z\in(0,L)}(R+\eta_0)>\delta]=0.
		\end{align}
		That is we have
		\begin{align}
		\hat\bP[\tau^{}=0, \inf_{z\in(0,L)}(R+\eta_0)>\delta,  \|R+\eta_0\|_{H^2(0,L)}<\frac1{\delta}]=0.
		\end{align}
		{ Notice that the probability space constructed in Theorem \ref{skorohod2} may depend on $\delta>0$. However, if a pathwise solution (and not only martingale solutions) exists that is defined on the original filtered probability space $(\Omega,\sF, (\sF_t)_{t\geq 0},\bP)$, using the calculations above, we can in fact pass $\delta\to 0$ and prove the existence of a maximal solution that exists until the time the walls of the tube collapse.}
		
		{\bf {Step 2.}} The statement  $
		\hat	\eta^*(t)=\hat\eta(t) \  \text{if } t<\tau,
		$
		can be obtained the same way as in \eqref{etasequal1}.
		
	\end{proof}
}

Thus, by combining Theorem \ref{exist2}, Lemma~\ref{StoppingTime}, and  \eqref{etasequal2}, we obtain the main result of this work.
	\begin{theorem}[Main result]\label{MainTheorem}
		For any given $\delta>0$, if the deterministic initial data  $\eta_0$ satisfies \eqref{etainitial}, then the stochastic processes $(\hat\bu,\hat\eta,\tau)$
		obtained in the limit specified in Theorem~\ref{skorohod2},  along with the stochastic basis constructed in Theorem \ref{skorohod2}, 
		determine a martingale solution 
		in the sense of Definition \ref{def:martingale} of the stochastic FSI problem \eqref{u}-\eqref{ic}, with the stochastic forcing given by \eqref{StochasticForcing}.\\
	\end{theorem}

\bibliographystyle{plain}
\bibliography{stochfsi}

\end{document}